\documentclass[reqno]{amsart}

\usepackage{amssymb, amsmath}
\usepackage{mathrsfs}
\usepackage{amscd}
\usepackage{verbatim}
\usepackage{nicefrac}
\usepackage{enumitem}
\usepackage[utf8]{inputenc}
\usepackage[usenames, dvipsnames]{color}
\usepackage{todonotes}

\theoremstyle{plain}
\newtheorem{theorem}{Theorem}[section]
\theoremstyle{remark}
\newtheorem{remark}[theorem]{Remark}

\newtheorem{example}[theorem]{Example}

\theoremstyle{plain}
\newtheorem{corollary}[theorem]{Corollary}
\newtheorem{lemma}[theorem]{Lemma}
\newtheorem{proposition}[theorem]{Proposition}
\newtheorem{definition}[theorem]{Definition}

\newtheorem{setting}[theorem]{Setting}
\numberwithin{equation}{section}
\usepackage{dsfont}
\def\F{{\mathbb F}}
\def\G{{\mathbb G}}
\def\N{{\mathbb N}}

\def\R{{\mathbb R}}
\def\C{{\mathbb C}}

\newcommand{\E}{{\mathbb E}}
\renewcommand{\P}{{\mathbb P}}
\newcommand{\eps}{\varepsilon}
\renewcommand{\l}{\lambda}
\newcommand{\om}{\omega}

\newcommand{\Ito}{{\hbox{\rm It\^o}}}

\newcommand{\maxsym}{\vee}
\newcommand{\minsym}{\wedge}

\newcommand{\calL}{{\mathcal L}}
\newcommand{\n}{\Vert}
\newcommand{\one}{{{\mathds 1}}}

\newcommand{\s}{^*}
\newcommand{\lb}{\langle}
\newcommand{\rb}{\rangle}

\newcommand{\limn}{\lim_{n\to\infty}}
\newcommand{\limk}{\lim_{k\to\infty}}

\newcommand{\sumn}{\sum_{n\in \N}}
\newcommand{\sumk}{\sum_{k\in \N}}

\newcommand{\wh}{\widehat}
\newcommand{\supp}{\text{\rm supp\,}}

\newcommand{\cont}{\ensuremath{\mathcal{C}}}

\newcommand{\abs}[1]{\ensuremath{\lvert #1 \rvert}}
\newcommand{\Abs}[1]{\ensuremath{\big\lvert  #1 \big\rvert}}
\newcommand{\nrm}[1]{\ensuremath{\lVert #1 \rVert}}

\newcommand{\ssgnrm}[1]{\ensuremath{\bigg\lVert #1 \bigg\rVert}}		
		
\newcommand{\nnrm}[2]{\ensuremath{\lVert #1 \rVert_{#2}}}			
\newcommand{\gnnrm}[2]{\ensuremath{\big\lVert #1 \big\rVert_{#2}}}	
	
\newcommand{\ssgnnrm}[2]{\ensuremath{\bigg\lVert #1 \bigg\rVert_{#2}}}	
\newcommand{\rnnrm}[2]{\ensuremath{\Bigg\lVert #1 \Bigg\rVert_{#2}}}

\newcommand{\const}{\ensuremath{C}}

\newcommand{\rklam}[1]{\left(#1\right)}					
\newcommand{\grklam}[1]{\big(#1\big)}
\newcommand{\sgrklam}[1]{\Big(#1\Big)}
\newcommand{\ssgrklam}[1]{\bigg(#1\bigg)}
\newcommand{\rrklam}[1]{\Bigg(#1\Bigg)}						
\newcommand{\ggklam}[1]{\big\{#1\big\}}						

\newcommand{\ssggklam}[1]{\bigg\{#1\bigg\}}

\newcommand{\geklam}[1]{\big [#1 \big ]}					
\newcommand{\sgeklam}[1]{\Big [#1 \Big ]}
\newcommand{\ssgeklam}[1]{\bigg [#1 \bigg ]}					
\newcommand{\reklam}[1]{\Bigg [#1 \Bigg ]}					

\newcommand{\dt}{\mathrm d t}

\newcommand\bP{\mathbb{P}}
\newcommand\bR{\mathbb{R}}

\newcommand\bN{\mathbb{N}}

\newcommand\cA{\mathcal{A}}
\newcommand\cB{\mathcal{B}}
\newcommand\cC{\mathcal{C}}
\newcommand\cD{\mathcal{D}}
\newcommand\cF{\mathcal{F}}
\newcommand\cH{\mathcal{H}}
\newcommand\cL{\mathcal{L}}
\newcommand\cP{\mathcal{P}}

\newcommand\cO{\mathcal{O}}
\newcommand\cQ{\mathcal{Q}}

\newcommand{\domain}{\ensuremath{\mathcal{O}}}   % beschränktes Lipschitz-Gebiet
\newcommand{\linop}[2]{\ensuremath{\mathcal{L}(#1,#2)}}
\newcommand{\gsum}{\ensuremath{\gamma}}
\newcommand{\gradon}[2]{\ensuremath{\gamma(#1,#2)}}

\DeclareFontFamily{U}{matha}{\hyphenchar\font45}
\DeclareFontShape{U}{matha}{m}{n}{
      <5> <6> <7> <8> <9> <10> gen * matha
      <10.95> matha10 <12> <14.4> <17.28> <20.74> <24.88> matha12
      }{}
\DeclareSymbolFont{matha}{U}{matha}{m}{n}
\DeclareFontSubstitution{U}{matha}{m}{n}
\DeclareFontFamily{U}{mathx}{\hyphenchar\font45}
\DeclareFontShape{U}{mathx}{m}{n}{
      <5> <6> <7> <8> <9> <10>
      <10.95> <12> <14.4> <17.28> <20.74> <24.88>
      mathx10
      }{}
\DeclareSymbolFont{mathx}{U}{mathx}{m}{n}
\DeclareFontSubstitution{U}{mathx}{m}{n}
\DeclareMathDelimiter{\vvvert}{0}{matha}{"7E}{mathx}{"17}
\newcommand{\nnnrm}[1]{\ensuremath{\left\vvvert #1 \right\vvvert}}

\usepackage[colorlinks,linkcolor={blue},citecolor={blue},urlcolor={black}]{hyperref}

\begin{document}

\title{Stochastic integration in quasi-Banach spaces}

\author{Petru A. Cioica-Licht}
\address{Petru A. Cioica-Licht (n\'e Cioica)\\ Department of Mathematics and Statistics\\ University of Otago\\ P.O.~Box~56\\ Dunedin~9054\\ New Zealand}
\email{pcioica@maths.otago.ac.nz}

\author{Sonja G. Cox}
\address{Sonja G. Cox, Korteweg-de Vries Instituut\\ University of Amsterdam\\P.O.~Box~94248\\1090 GE Amsterdam\\ the Netherlands} \email{s.g.cox@uva.nl}

\author{Mark C. Veraar}
\address{Mark C. Veraar, Delft Institute of Applied Mathematics\\
Delft University of Technology \\ P.O.~Box~5031\\ 2600 GA Delft\\the
Netherlands} \email{m.c.veraar@tudelft.nl}

\thanks{The first author has been supported by the Deutsche Forschungsgemeinschaft (DFG, grant
DA 360/20-1) and partially by the Marsden Fund Council from Government funding, administered by the Royal Society of New Zealand.}
\thanks{The second author is supported by the research program VENI Vernieuwingsimpuls with project number $ 639.031.549 $, which is financed by the Netherlands Organization for Scientific Research (NWO)}
\thanks{The third named author is supported by the Vidi subsidy $639.032.427$ of the Netherlands Organization for Scientific Research (NWO) }



\begin{abstract}
In this paper we develop a stochastic integration theory for processes with values in a quasi-Banach space. The integrator is a cylindrical Brownian motion. The main results give sufficient conditions for stochastic integrability. They are natural extensions of known results in the Banach space setting. We apply our main results to the stochastic heat equation where the forcing terms are assumed to have Besov regularity in the space variable with integrability exponent $p\in (0,1]$. The latter is natural to consider for its potential application to adaptive wavelet methods for stochastic partial differential equations.
\end{abstract}

\maketitle

\setcounter{tocdepth}{1}
\tableofcontents

%46A16  View Publications (1991-now) Not locally convex spaces (metrizable topological linear spaces, locally bounded spaces, quasi-Banach spaces, etc.)
%60B11  View Publications (1980-now) Probability theory on linear topological spaces [See also 28C20]
%60H05  View Publications (1973-now) Stochastic integrals

\section{Introduction}

In this paper we extend the stochastic integration theory developed in~\cite{NeeVerWei2007, NeervenWeis2005a} from Banach spaces to \emph{quasi-}Banach spaces.
This means that we assume that the complete metric $\rho$ on the state space $E$ of the integrands is not necessarily induced by a norm but by an \emph{$r$-norm} $\nrm{\cdot}$ for some $r\in (0,1]$, i.e., $\rho(x,y)=\nrm{x-y}^r$, $x,y\in E$, where $\nrm{\cdot}\colon E\to [0,\infty)$ is a function obeying all the properties of a norm except the triangle inequality; instead, the \emph{$r$-triangle inequality}
\[
\nrm{x+y}^r\leq \nrm{x}^r+\nrm{y}^r,\quad x,y\in E,
\]
holds. The Banach space setting corresponds to $r=1$. See Section~\ref{sec:QBrB} for a precise definition of quasi-Banach spaces and their relationship to $r$-Banach spaces.

%A quasi-Banach space $E$ has all the properties a Banach space has, except that the metric $d$ is not induced by a norm but merely by an $r$-norm for some $r\in(0,1]$, i.e., $d(x,y)=\nrm{x-y}$, where $\nrm{\cdot}\colon E\to [0,\infty)$ has all the properties of a norm except the triangle inequality. Instead, we merely have
%\[
%\nrm{x+y}^r
%\]
% 
%In contrast to Banach spaces, quasi-Banach spaces are metric spaces, where the metric comes from a quasi-norm.
%Quasi-Banach spaces have all the properties Banach spaces have, except for the fact that the metric does not come from a norm but from a so-called quasi-norm, i.e., a function $\nrm{\cdot}\colon E\to [0,\infty)$ 

Our main motivation comes from the numerical analysis of stochastic partial differential equations (SPDEs, for short), where quasi-Banach spaces occur naturally in form of Besov spaces:
It is well-known from approximation theory that, 
%for a function $f\colon\domain\to\bR$ defined on a Lipschitz domain $\domain\subseteq\bR^d$, 
under suitable assumptions, the optimal convergence rate of wavelet or finite element approximation methods is determined by the regularity $\alpha>0$ of the target function $f\colon\domain\to\bR$ in the scale
\begin{equation}\label{eq:Besov:scale}
B^\alpha_{\tau,\tau}(\domain),\qquad\frac{1}{\tau}=\frac{\alpha}{d}+\frac{1}{p}, \qquad \alpha>0,
\end{equation}
of Besov spaces~\cite{BinDahDeV+2002,GasMor2014,GasMor2017, DeV1998}; here, $\domain\subseteq\bR^d$ is a Lipschitz domain and $p>1$ describes the $L_p(\domain)$-norm that is used to measure the error. 
Thus, to answer the question of optimal convergence rates,
%of numerical methods 
%for SPDEs with respect to the space variable 
an analysis of the regularity in the scale~\eqref{eq:Besov:scale} is required.
However, for $\alpha>d(p-1)/p$ (note that the focus lies on large values of $\alpha$), the Besov spaces $B^\alpha_{\tau,\tau}(\domain)$ from~\eqref{eq:Besov:scale} are not Banach spaces anymore, but merely quasi-Banach spaces.
In the setting of SPDEs, this means that the regularity analysis in the Besov spaces from this scale cannot be carried out simply by using one of the abstract approaches to SPDEs, like, for instance, the semigroup approach or the variational approach, as they are all developed in the Banach space framework. 
By constructing a stochastic integral in quasi-Banach spaces we deliver the first building block towards a semigroup approach to SPDEs in quasi-Banach spaces, thus, in particular, to a regularity analysis in Besov spaces from the scale~\eqref{eq:Besov:scale}.
For more background information on the relationship between regularity, optimal convergence rates and adaptivity in the context of SPDEs, we refer to the introduction of~\cite{CioDahKin+2011} and the survey article~\cite{CDDKLRS}, see also~\cite{Cio2015}.

The main challenges for the extension of \Ito's stochastic integral from Banach to quasi-Banach spaces come from the fact that the latter are not locally convex.
As a consequence, we cannot use any convexity and duality arguments. In particular, it may happen that $E$ has trivial dual (e.g.\ if $E=L^r(0,1)$ for some $r\in (0,1)$, see e.g.~\cite[Section~1.47]{Rud1991}).
%(e.g.\ $L^p(0,1)^* = \{0\}$ if $p\in (0,1)$, see \cite[1.47]{Rud1991}), so that one cannot use duality arguments. On the other hand $B^{s}_{p,p}(\R^d)^* = B^{-s-d/p'}_{\infty, \infty}$ if $p\in (0,1)$, is still large enough to separate the points of $B^{s}_{p,p}(\R^d)$ (see \cite[Theorem 2.11.3]{Tri1983}).
On the other hand there are important cases (e.g.\ the Besov spaces from the scale~\eqref{eq:Besov:scale} with $\alpha>d(p-1)/p$) where the dual space is not norming for $E$, but still separates the points of $E$ (see, e.g.~\cite[Theorem 2.11.3]{Tri1983}). 
In order to keep things as general as possible we will avoid duality arguments where possible. In several instances we will present further results in case the dual space is separating.

Another difficulty caused by the lack of convexity is that there is no analogue of the Riemann integral, of the Bochner integral, or of the Pettis integral for quasi-Banach space valued functions, see \cite{AlbAns2013} for more details on this topic. 
Some integration theory especially developed for $r$-Banach spaces can be found in \cite{Vog1967}. 
It is worth mentioning that, as a consequence of  Corollary~\ref{cor:lattice}, we obtain some sufficient conditions for integrability in $r$-Banach function spaces as well, see also Remark~\ref{rem:QBFS:integral:det}.

Our results show that, other than the deterministic integrals mentioned above, large parts of the stochastic integration theory developed in~\cite{NeeVerWei2007, NeervenWeis2005a} do not rely on the local convexity of the space. One may wish to go even further and study stochastic integration in metric vector spaces in general. However, although many important developments on the geometry of metric spaces have been made in recent years in the Ribe program (see the survey \cite{Naor12} and references therein), at the moment it seems that too little is known about the specific part of metric geometry which we need.

There is an extensive literature on stochastic integration in Banach spaces, see e.g.~\cite{Brz1995, BrzNee00, BrzNeeVer+2008, CoxGeiss, CoxVer2011, Dettweiler83, Dirksen14, DMvN, DirkYar, Garl86, ManRud, MarRock, McC1989, NeeVerWei2007, NeervenWeis2005a, Nh, Ondr, RieGa, vNRi, Ros87, RosSuc1980, VerYar}.
In this article we deal with Gaussian noise, more specifically, we aim to define the stochastic integral of an $\calL(H,E)$-valued process with respect to a cylindrical Wiener process $W_{H}$, where $H$ is a separable Hilbert space and $E$ is a separable quasi-Banach space.
%As already mentioned above, we follow the strategy from~\cite{NeeVerWei2007, NeervenWeis2005a}, see also~\cite{NeeVerWei2015} for a survey.
As in the Banach space case, a deterministic function is stochastically integrable if (and only if) it is $\gamma$-radonifying,
see~\cite{NeeVerWei2007,NeervenWeis2005a} and Section~\ref{sec:stoch_int_functions} below.
Regarding stochastic integrands, as already mentioned above, our approach is similar in spirit to the stochastic integration theory in Banach spaces with the UMD property that was developed in \cite{NeeVerWei2007} (a survey can be found in \cite{NeeVerWei2015}).
% The class of UMD spaces includes all classical reflexive spaces. This stochastic integration theory is naturally limited to the class of UMD space in the sense that the validity of the two-sided $L^p$-estimates also characterises the UMD property.
%\todo{We should add a remark in the main text, after Proposition \ref{prop:quasi_not_UMD}. In the quasi-setting: two-sided estimates for stochastic integrals will give UMD as well. The proof can be done as \cite{Garl86}}
% Added Sonja's suggestion...
However, as explained in Section~\ref{ssec:UMD} below, the UMD property does not make sense in the quasi-Banach space setting. A different (one-sided) type of decoupling is used, which results in one-sided estimates for the stochastic integral.
This approach was previously explored in~\cite{CoxVer2011} in the Banach space setting.

Most of the results in this paper are formulated in terms of $r$-Banach spaces instead of quasi-Banach spaces (see Section~\ref{sec:QBrB} for precise definitions).
However, due to the Aoki-Rolewitz theorem, see \cite{Aok1942} and \cite{Rol1957}, this is not a restriction. It only makes our presentation easier to read and all relevant results can be translated back to the quasi-Banach space setting if required. We refer to Section~\ref{sec:QBrB} for more details.

%Nowadays, it is common to call a vector space $E$ a quasi-Banach space if it is complete with respect to a \emph{quasi-norm} $\nrm{\cdot}$, i.e., a mapping $\nrm{\cdot}\colon E\to [0,\infty)$ that obeys all the properties of a norm, except that instead of the triangle inequality we merely have that
%\[
%\nrm{x+y}\leq C\grklam{\nrm{x}+\nrm{y}},\qquad x,y\in E,
%\]
%with a constant $C\in[1,\infty)$.
%Due to the Aoki-Rolewitz theorem, see \cite{Aok1942} and \cite{Rol1957}, on every quasi-Banach space we can find an equivalent $r$-norm $\nnnrm{\cdot}$ for a suitable $r\in (0,1]$ and the space is complete with respect to the metric $\rho(x,y):=\nnnrm{x-y}^r$, $x,y\in E$; the space $(E,\nnnrm{\cdot})$ is called an $r$-norm. Most of our results will be stated for $r$-Banach spaces since in this way we avoid using a renorming in every proof. This does not create any problems in the examples we have in mind.
%
%
% the common definition of a quasi-Banach space is as presented in Definition~\ref{def:quasi} below, i.e., a space $E$ endowed with a quasi-norm
%
%Our results are formulated in terms of $r$-Banach spaces, i.e., of vector spaces endowed with an $r$-norm $\nrm{\cdot}$ which are complete with respect to the metric $\rho(x,y):=\nrm{x-y}$, $x,y\in E$. 
%
This paper is organized as follows:
In Section~\ref{sec:prelim} we discuss necessary preliminaries on quasi-Banach spaces and $r$-Banach spaces, measurability, random sums and Gaussian random variables. 
We give a rather detailed account in this first part of the paper, emphasizing conceptual and notational clarity. 
In Section~\ref{sec:gamma} we extend the theory of $\gamma$-radonifying operators to the $r$-Banach space setting. 
For this we will follow the presentation in \cite[Chapter~9]{HytNeerVerWeis2017} and indicate those instances where changes are required. An important step is to obtain the (right) ideal property for $\gamma$-radonifying operators in $r$-Banach spaces. The key trick is a simple lemma on Gaussian sums which can be proved without convexity and duality arguments (see Proposition~\ref{prop:f-r}).

Section \ref{sec:stoch_int_functions} is devoted to the characterization of stochastic integrability for non-random elements in terms of $\gamma$-radonifying operators. 
Further equivalences are obtained in the case that $E$ has a separating dual in Theorem~\ref{thm:weak_char_stoch_int} and Proposition~\ref{prop:Pettis_Ito}. 
In this case we obtain an $r$-Banach space analogue of \cite[Theorem~2.5]{NeervenWeis2005a}; in particular, the stochastic integral of deterministic integrands can be seen as a stochastic Pettis integral.
In Section \ref{sec:stoch_int_processes} we consider the more complicated setting of random integrands. After discussing the UMD property and one-sided decoupling inequalities, we show how decoupling can be used to obtain sufficient conditions for stochastic integrability. $L^p$-estimates for stochastic integrals can be found in Theorems~\ref{thm:stochInt_BDG} and~\ref{thm:stochInt_local}. The special case where $E$ has a separating dual is considered in Corollary~\ref{cor:waekconv} and this result is the natural analogue of \cite[Theorems~3.6 and~5.9]{NeeVerWei2007}.

In Section \ref{sec:SPDE} we illustrate how our integration theory can be used to study the stochastic heat equation on $\R^d$. The forcing terms are assumed to have Besov regularity in space and the integrability exponents are allowed to be any number in $(0,\infty)$. The main result of this section is Theorem \ref{thm:stoch_heat_eqn} which is a space-time regularity result for the stochastic heat equation. Some of the technical estimates are proved in an appendix. In particular, in Section \ref{sec:AppB} we present several point-wise estimates for convolutions with heat kernels which are required for estimating the Besov norm.

\smallskip

Before we start our presentation, let us fix some notation.

\smallskip

\noindent \textbf{Notation.} 
All vector spaces in this article are assumed to be real unless stated otherwise. % note: when working with Fourier transforms, e.g. 
% in the appendix, we need complex spaces.
%
Given a parameter set $\cP$ 
and mappings $A,B\colon \cP \rightarrow \R$, we write `$A(p) \lesssim B(p)$' to express 
$ 
  `\exists C\in (0,\infty)\ \forall p\in \cP\colon A(p) \leq C B(p)
$',
% that there exists a constant $C\in (0,\infty)$ such that $A(q) \leq C B(q)$ for all $q\in \cQ$, 
and given a further parameter set $\cQ$ and mappings $A,B\colon \cP\times \cQ \rightarrow \R$ 
we write `$A(p,q) \lesssim_{q} B(p,q)$' to express 
`$ \forall q \in \cQ\ \exists C_q\in (0,\infty)\ \forall p \in \cP\colon  A(p,q) \leq C_q B(p,q) $'.
% that for every 
% $p\in \cP$ there exists a constant $C_p\in (0,\infty)$ such that $A(p,q) \leq C_p B(p,q)$ for all $q\in \cQ$.
Moreover, `$A(p)\eqsim B(p)$' means `$A(p) \lesssim B(p)$ and $B(p) \lesssim A(p)$', and `$A(p,q) \eqsim_q B(p,q)$'
means `$A(p,q) \lesssim_q B(p,q)$ and $B(p,q) \lesssim_q A(p,q)$'.
% 
% For $(T_1,\tau_1)$, $(T_2,\tau_2)$ topological spaces we denote by $\cont(T_1;T_2)$ the space of continuous functions from $T_1$ tot $T_2$.
% % no need to explain.

If $E,F$ are (quasi-)Banach spaces, then we write $E\hookrightarrow F$ if $E$ embeds isomorphically into $F$, and we write $E\eqsim F$ if $E$ and $F$ are isomorphic as (quasi-)Banach spaces.
 We use the common notation $\langle\cdot,\cdot\rangle\colon E\times E^*\to\bR$ for the dual form, i.e., $\langle x,x^*\rangle :=x^*(x)$, for $x^*\in E^*$ and $x\in E$.
The notation $\langle \cdot,\cdot\rangle_H$ is used to denote the inner product on a Hilbert space $H$.

In addition,  throughout this article, $(\Omega, \cF,\bP)$ denotes a probability space and $\E$ is the corresponding expectation operator. Moreover,  $(\varepsilon_n)_{n\in\N}$ denotes a \emph{Rade\-ma\-cher sequence}, i.e., a sequence of independent and identically distributed $\{-1,1\}$-valued random variables with $\P(\varepsilon_n=1)=\P(\varepsilon_n=-1)=\frac{1}{2}$, for $n\in\bN$. Furthermore, $(\gamma_n)_{n\in\bN}$ denotes a \emph{Gaussian sequence}, i.e., a sequence of independent real-valued centered standard Gaussian random variables.  $\cB(S)$ denotes the Borel $\sigma$-algebra on a topological space $S$.

% Always use the following notation (this is for our own reference, does not
% need to be included in the final article):
% \begin{itemize}
% \item $x \minsym y := \min\{x,y\},\, x\maxsym y := \max\{ x ,y\}$. % no need to explain
% \item $E,F,\ldots$ are $r$-Banach spaces % no need to explain
% \item $X,Y,Z,\ldots$ are random variables % no need to explain
% \item $G,H,\ldots$ are Hilbert spaces, $\langle \cdot , \cdot \rangle_H$ is the inner product on $H$. % no need to explain
% \item $(M,d)$ metric space % no need to explain
% \item $T$ is the endpoint of the time interval; shouldn't be used as an operator % note to self
% \item $R,V,U,\ldots$ for operators % no need to explain
% \item $(S,\cA,\mu)$ is a measure space.% no need to explain
% \item For $(S,\cA,\mu)$ a measure space we have that $L^p(S)$ is the corresponding Lebesgue space,
% and $\ell^p= L^p(\N)$ is the sequence space. % no need to explain
% \item We also use the finite-dimensional version $\ell^p_n = L^p(\{1,\ldots,n\})$. % explained at first occurance
% \item We write $L^2(0,T;E)$ as short-hand notation for $L^2([0,T];E)$. % explained at first occurance
% \item $L^{(\gamma)}(\bR^d,\bR)$ for the calculations in the appendix. % no need to explain
% \end{itemize}

\medskip

\section{Preliminaries}\label{sec:prelim}

\subsection{Quasi-Banach spaces and \texorpdfstring{$r$}{r}-Banach spaces}\label{sec:QBrB}
As explained in the introduction, in this article we develop  a  stochastic integration theory for
processes taking values in a quasi-Banach space. More precisely, we
consider processes taking values in an $r$-Banach space, $r\in (0,1]$,
but this is essentially the same thing---see Remark~\ref{rem:quasiBS} below. In
order to guarantee conceptual clarity for the reader familiar only with the Banach space setting,
we recall some basic facts about $r$- and quasi-Banach spaces in this section.
Details can be found, e.g., in \cite{Kal2003,KalPecRob1984}.

\subsubsection{$r$-Banach spaces}

\begin{definition}\label{def:r-space}
Let $E$ be an $\R$-vector space and let $r\in (0,1]$. A mapping $\nrm{\cdot}\colon E \rightarrow \R$
is called an \emph{$r$-norm} if for all $\lambda \in \R$, $x,y\in E$ it holds that
\begin{enumerate}[label=\textup{(\arabic*)}]
 \item $\left\| \lambda x \right\| = |\lambda| \left\| x \right\|$;
 \item $\left\| x \right\| = 0$ if and only if $x=0$;
 \item $ \left\| x + y \right\|^r \leq \left\| x \right\|^r + \left\| y \right\|^r$ ($r$-triangle inequality).
\end{enumerate}
In this case we call $(E,\left\| \cdot \right\|)$ an \emph{$r$-normed space}, and if $E$ is complete with
respect to the metric $\rho_{\nrm{\cdot}}\colon E\times E \rightarrow \R$ induced by $\left\| \cdot \right\|$ (i.e.,
$\rho_{\nrm{\cdot}}(x,y)= \| x - y \|^r$ for $x,y\in E$), we call $(E,\left\| \cdot \right\|)$ an \emph{$r$-Banach space}.
\end{definition}

Note that a $1$-norm is in fact a norm. In general, $r$-normed spaces are not locally convex. As a consequence, fundamental theorems such as the Hahn-Banach theorem do not hold in general.
% In fact,  it is easily checked that  if the Hahn-Banach theorem were to hold for an $r$-normed space $(E, \left\| \cdot \right\|)$, then $E$ would be normable, i.e., there would exist a norm on $E$ that is equivalent to $\left\| \cdot \right\|$.
However, results relying mainly on completeness arguments still hold in the $r$-Banach space setting. These are, e.g., the uniform boundedness theorem, the open mapping theorem, and the closed graph theorem together with their numerous consequences, see, e.g., \cite{KalPecRob1984} for details.

For $(E_1,\nnrm{\cdot}{1})$ an $r_1$-Banach space and $(E_2,\nnrm{\cdot}{2})$ an $r_2$-Banach space, $r_1,r_2\in (0,1]$, we write $\mathcal{L}(E_1,E_2)$ for the space of bounded linear operators from $E_1$ to $E_2$, endowed with the $r_2$-norm
\[
\nnrm{R}{\mathcal{L}(E_1,E_2)}
:=
\sup_{x\in E_1\setminus\{0\}} \frac{\nnrm{Rx}{2}}{\nnrm{x}{1}},
\qquad
R\in \mathcal{L}(E_1,E_2).
\]
Note that, as in the Banach space setting, boundedness of linear operators is equivalent to their continuity.
We write $E^*:=\mathcal{L}(E,\R)$ for the topological dual of an $r$-Banach space $(E,\nrm{\cdot})$ endowed with the norm $\nnrm{\cdot}{E^*}$.
Note that the adjoint $R^*\in\mathcal{L}(E_2^*,E_1^*)$ of a linear bounded operator $R\in\mathcal{L}(E_1,E_2)$ can be defined in the usual way. However, since the Hahn-Banach theorem fails to hold, in general, we only have $\| R^* \|_{\calL( E^*_2,E^*_1)}\leq \| R \|_{\calL(E_1,E_2)}$.

Given a metric space $(M,d)$ and a set $\Gamma\subset \cont(M;\bR)$, we say that \emph{$\Gamma$ separates the points of $M$} if for every pair $x,y\in M$ with $x\neq y$ there exists a $\varphi\in \Gamma$ such that $\varphi(x)\neq \varphi(y)$.
We also say that $\Gamma$ \emph{is separating} (for $M$). If $E$ is a Banach space, then $E^*$ separates the points of $E$ by the Hahn-Banach theorem. For $r$-Banach spaces the situation is more complicated, as the following examples demonstrate.
% However, if $E$ is a (non-trivial) $r$-Banach space, $r\in (0,1)$, then $E^*$ may be trivial and thus fail to separate points. However, this is not always the case for duals of quasi-Banach spaces, although there are prominent examples where this is the case, as the following examples show.

Let  $(S,\cA,\mu)$  be a measure space and $r\in (0,\infty)$, then
 the space $L^r(S, \cA,\mu)$ of all equivalence classes of $\cA/\cB(\bR)$-measurable functions $f\colon S\to\bR$ such that
\[
\nnrm{f}{L^r(S)}^r:=\int_S \abs{f}^r\,\mathrm \,\mathrm{d}\mu < \infty,
\]
is a prominent example of a $\min\{r,1\}$-Banach space (see \cite{Day1940}, where these $r$-Banach spaces were analyzed in detail for the first time  for $0<r<1$).
 We simply write $L^r(\domain)$ if $\mu$ is the Lebesgue measure on a Borel-measurable subset $\domain$ of $\bR^d$ and $\ell^r(\bN)$ if $\mu$ is the counting measure on $\bN$. 
For $0<r<1$, the topological duals of these spaces differ dramatically, depending on the structure of the underlying measure space $(S,\cA,\mu)$.
Indeed, for  $0<r<1$, the dual of $L^r(0,1)$ is trivial, whereas the sequence space $\ell^r(\bN)$ has a very rich dual that is isomorphic to  the space of all bounded sequences, 
%$\ell^\infty(\bN)$, 
see \cite{Day1940, KalPecRob1984}. In particular, $\ell^r(\bN)^*$ separates the points of $\ell^r(\bN)$.

Another example of an $r$-Banach space, particularly relevant from the point of view of approximation theory, has been already mentioned in the introduction:
for $d\in \N$, $\domain\subseteq\bR^d$ and $p\in (1,\infty)$ the
scale  of Besov spaces $\{ B^\alpha_{r,r}(\domain) : \alpha,r\in (0,\infty), \frac{1}{r} = \frac{\alpha}{d}+ \frac{1}{p}\}$  appears naturally in the context of non-linear approximation. If $\alpha>d(p-1)/p$ and $r\in (0,\infty)$ satisfy $\frac{1}{r} = \frac{\alpha}{d}+ \frac{1}{p}$, then $r<1$, whence $B^{\alpha}_{r,r}(\domain)$ is an $r$-Banach space but not a Banach space. Note however that for such $\alpha,r$  the dual  $B^\alpha_{r,r}(\domain)^*$
separates the points of $B^\alpha_{r,r}(\domain)$. This follows from the fact that $B^\alpha_{r,r}(\domain)$ is continuously embedded in $L^p(\domain)$ (see, e.g., \cite[Theorem~1.73(i)]{Tri2006} for the case $\domain=\bR^d$, which immediately implies the general case if one uses the definition by restriction, see \cite[Section~1.11]{Tri2006}).
% Since $L^p(\domain)$ is a Banach space for $p>1$, its dual (which is embedded in the dual of $B^\alpha_{r,r}(\domain)$) separates its points and therefore also separates the points of $B^\alpha_{r,r}(\domain)$.

\subsubsection{Quasi-Banach spaces}

\begin{definition}\label{def:quasi}
Let $E$ be an $\R$-vector space. A mapping $\left\| \cdot \right\|\colon E \rightarrow \R$
is called a \emph{quasi-norm} if there exists a $C\in (0,\infty)$ such that for all $\lambda \in \R$, $x,y\in E$, it holds that
\begin{enumerate}[label=\textup{(\arabic*)}]
 \item $\left\| \lambda x \right\| = |\lambda| \left\| x \right\|$;
 \item $\left\| x \right\| = 0$ if and only if $x=0$;
 \item $ \left\| x + y \right\| \leq C(\left\| x \right\| + \left\| y \right\|)$ (quasi-triangle inequality).
\end{enumerate}
In this case we call $(E,\left\| \cdot \right\|)$ a \emph{quasi-normed space}.
Moreover, we write $\mathscr{O}_{\|\cdot\|}$ for the \emph{standard topology induced by $\nrm{\cdot}$}, i.e.,  $\mathscr{O}_{\|\cdot\|}$ is the collection of all subsets $U\subseteq E$ such that for all $x\in U$ there exists an $\eps \in (0,\infty)$ such that
$B_{x,\eps}^{\nrm{\cdot}}:=\{y\in E : \| x-y \|<\eps \} \subseteq U$.
If $E$ is complete with respect to $\mathscr{O}_{\|\cdot\|}$, we call $(E,\left\| \cdot \right\|)$ a \emph{quasi-Banach space}.
\end{definition}
%
% The prefix `quasi-' refers to the fact that the triangle inequality might not hold, so that, we might lose local convexity.
% More detailed, a \emph{quasi-norm} $\nrm{\cdot}\colonE\to [0,\infty)$ on a vector space $E$ is a mapping fulfilling all the conditions a norm does, except that the triangle inequality might not hold. Instead, a \emph{quasi-triangle inequality} is imposed, i.e.,
% \[
% \nrm{x+y}\leq C\grklam{\nrm{x}+\nrm{y}}
% \]
% for a constant $C\in[1,\infty)$ that does not depend on $x,y\in E$.
% A \emph{quasi-normed space} $(E,\nrm{\cdot})$ is a vector space $E$, endowed with a quasi-norm $\nrm{\cdot}$.
% The standard topology $\mathscr{O}_{\nrm{\cdot}}$ induced by a quasi-norm $\nrm{\cdot}$ on $E$ can be defined in the usual way, i.e., a subset $U\subset E$ is an element of $\mathscr{O}_{\nrm{\cdot}}$ if and only if for every $x\in U$, there exists an $\varepsilon>0$, such that
% \[
% B(x,\varepsilon):= B_{\nrm{\cdot}}(x,\varepsilon)
% :=
% \ggklam{y\in E: \nrm{x-y}<\varepsilon}
% \subset
% U.
% \]
% A \emph{quasi-Banach space} $(E,\nrm{\cdot})$ is a quasi-normed space that is complete with respect to the standard topology $\mathscr{O}_{\nrm{\cdot}}$ induced by $\nrm{\cdot}$.

\begin{remark}
Note that, other than $r$-norms, a quasi-norm $\nrm{\cdot}$ is not necessarily continuous and an `open ball' $B_{x,\varepsilon}^{\nrm{\cdot}}$ is not necessarily contained in the standard topology $\mathscr{O}_{\nrm{\cdot}}$, see, for instance, \cite{Aok1942}.
\end{remark}

\begin{remark}\label{rem:quasiBS}
 Every $r$-normed space is a quasi-normed space as an $r$-norm fulfills all the requirements from Definition~\ref{def:quasi} with $C=2^{\frac{1-r}{r}}$.
Conversely,  the Aoki-Rolewitz theorem, see \cite{Aok1942} and \cite{Rol1957}, guarantees that, given a quasi-normed space $(E,\nrm{\cdot})$, there exists an $r\in (0,1]$ and an $r$-norm $\nnnrm{\cdot}$ that is equivalent to $\nrm{\cdot}$.
 In particular, $\mathscr{O}_{\nrm{\cdot}}=\mathscr{O}_{\nnnrm{\cdot}}$, where the latter coincides with the standard topology on the metric space $(E,\rho_{\nnnrm{\cdot}})$, see also Definition~\ref{def:r-space}.
Moreover,  $\mathscr{O}_{\nrm{\cdot}}$ is Hausdorff, compatible with the vector space operations and the unit ball $B_{0,1}^{\nnnrm{\cdot}}$ is a bounded neighborhood of the origin.
Thus, from a topological point of view, a quasi-Banach space endowed with its standard topology is a locally bounded $F$-spaces, i.e., a complete locally bounded metrizable topological vector space.
Conversely, on any locally bounded topological vector space $(E,\mathscr{O})$, there exists a quasi-norm $\nrm{\cdot}$, so that $\mathscr{O}=\mathscr{O}_{\nrm{\cdot}}$ \cite{KalPecRob1984}.
Moreover, if $(E,\nrm{\cdot})$ is separable, then $(E,\mathscr{O}_{\nrm{\cdot}})$ is a Polish space, i.e., a complete metrizable topological space containing a dense countable subset.
\end{remark}

% From now on, if not explicitly stated otherwise, we assume that $(E,\nrm{\cdot})$ is a separable $r$-Banach space for some $r\in(0,1]$.

\subsection{Measurability}\label{ssec:measurability}

In this section we recall some standard results on measurability of functions with values in separable metric spaces $(M,d)$. Clearly, all these results apply to separable $r$-Banach spaces.

Let $\mathcal{B}(M)$ denote the Borel $\sigma$-algebra on $(M,d)$. Furthermore, let $(S,\cA)$ be an arbitrary measure space.
A function $f\colon S \rightarrow M$ is called \emph{$\cA/\mathcal{B}(M)$-measurable} (or, simply, \emph{measurable}) if $f^{-1}(B)\in\mathcal{A}$ for all $B\in\mathcal{B}(M)$.
It is called \emph{strongly $\mathcal{A}$-measurable} (or simply \emph{strongly measurable}) if is the point-wise limit of  a sequence of  $\mathcal{A}$-simple functions; recall that a function is called \emph{$\cA$-simple} if it is $\mathcal{A}/\mathcal{B}(M)$-measurable and attains at most finitely many values.
On separable metric spaces, these notions of measurability coincide, see, e.g., \cite[Proposition~I.1.9]{VakTarCho1987}, where the following statement can be found.

\begin{proposition}\label{prop:mb}
Let $(S,\cA)$ be a measure space and let $(M,d)$ be a separable metric space. Given a function $f\colon S \to M$, the following are equivalent:
\begin{enumerate}[label=\textup{(\roman*)}]
\item $f$ is $\mathcal{A}/\mathcal{B}( M)$-measurable.
\item $f$ is strongly $\mathcal{A}$-measurable.
\item There exists a sequence $(f_n)_{n\in\bN}$ of countably valued  measurable  functions such that $f=\lim_{n\to\infty}f_n$ uniformly in $S$.
\end{enumerate}
\end{proposition}

If the metric space is complete, then measurability can also be characterized in the following way, see, e.g., \cite[Theorem~I.1.2 and Proposition~I.1.10]{VakTarCho1987}.

\begin{proposition}[Pettis' theorem]\label{prop:Pettis}
Let $(S,\cA)$ be a measure space, let $(M,d)$ be a Polish space and let $\Gamma \subseteq \cC(M,\R)$ separate  the points of  $M$. Then
\[
\mathcal{B}(M)=\sigma(\varphi:\varphi\in\Gamma),
\]
and for a given function $f\colon S\to E$, the following are equivalent:
\begin{enumerate}[label=\textup{(\roman*)}]
\item $f$ is measurable.
\item For all $\varphi\in\Gamma$, $\varphi\circ f$ is measurable.
\end{enumerate}
\end{proposition}

\begin{example}\label{ex:Boreldual}
Let $E$ be a  separable $r$-Banach space with separating dual $E^*$. Then
\[
\mathcal{B}(E)=\sigma(x^*: x^*\in E^*)
\]
and a function $f\colon S\to E$ is (strongly) measurable if and only if $\xi\mapsto \langle f(\xi),x^*\rangle$ is measurable for every $x^*\in E^*$.
\end{example}

% \begin{remark} % SONJA: I think we can omit this.
% The separability of $(M,d)$ in Proposition~\ref{prop:mb} is needed only in the proof of `(i)$\Rightarrow$(ii)'.
% Actually, in order to prove this implication one only needs that the $\mathcal{A}/\mathcal{B}(M)$-measurable function $f$ is separably valued, i.e., that the range of $S$ under $f$ is contained in a separable closed subspace of $M$.
% Also the separability assumption can be dropped in the second part of Proposition~\ref{prop:Pettis}: The equivalence `(i)$\Leftrightarrow$(ii)' holds in the non-separable case, if we assume that $f$ is separably valued.
% \end{remark}

Let $(E,\left\| \cdot \right\|_E)$ be a separable $r$-Banach space for some $r\in(0,1]$, and let $(S,\cA,\mu)$ be a $\sigma$-finite measure space. We let $L^0(S;E)$ be the space of all equivalence classes of $\cA/\cB(E)$-measurable functions.
For $p\in (0,\infty)$ we define
$\left\| \cdot \right\|_{L^p(S;E)} \colon L^0(S;E) \rightarrow [0,\infty]$
by $\nnrm{f}{L^p(S;E)} = \left( \int_{S} \nnrm{f }{E}^p \,\mathrm{d}\mu \right)^{\frac{1}{p}}$,
and we set
\begin{align*}
 L^p(S;E) = \ggklam{ f \in L^0(S;E) \colon \| f \|_{L^p(S;E)} < \infty }.
\end{align*}
Note that  for $p\in (0,\infty)$ it holds that $L^p( S;E)$ is a $\min\{p,r\}$-Banach space.
However, a definition of a Bochner integral in analogy with the Banach space setting is not possible on general $r$-Banach spaces, see \cite{AlbAns2013}.
In particular, in general the fact that $f\in L^1(S;E)$ does not mean that $\int_S f \,\mathrm{d}\mu$ makes sense in the usual way of Lebesgue/Bochner type integrals.

\subsection{Random sums}

Throughout this section $E$ denotes a separable $r$-Banach space, where $r\in (0,1]$ is fixed.
We present results on the coincidence of different convergence types for sums of independent (symmetric) $r$-Banach space valued random variables.
In the finite dimensional case, these results are often referred to as L\'{e}vy's theorem, whereas the generalization to Banach spaces is attributed to Kiyoshi \Ito\  and Makaki Nisio due to their seminal paper \cite{ItoNis1968}.

Given a probability space $(\Omega,\mathcal{F},\P)$,
we say that a mapping $X\colon \Omega \rightarrow E$ is an \emph{$E$-valued random variable} (or, simply, a \emph{random variable}) if $X$ is strongly measurable,
and we let $\P_X \colon \mathcal{B}(E)\rightarrow \R$ denote the distribution of $X$ (i.e., $\P_X(B)=\P(X\in B)$, $B\in\mathcal{B}(E)$).

We start with some auxiliary results, which we frequently use in the sequel.
The first one says that independence is preserved under  convergence in  distribution.

\begin{lemma}\label{lem:Conv:pres:ind}
Let $(X,Y)$ be an $E\times E$-valued random variable, and for all $n\in \N$ let $(X_n, Y_n)$ be an $E\times E$-valued random variable
such that $X_n$ and $Y_n$ are independent.
If $(X_n,Y_n)\rightarrow (X,Y)$ in distribution as $n\rightarrow \infty$, then $X$ and $Y$ are independent.
\end{lemma}

\begin{proof}
It suffices to prove that for arbitrary open subsets $F,G\subseteq E$,
\[
\E\geklam{\one_{F}(X)\one_G(Y)}
=
\E\geklam{\one_F(X)}
\E\geklam{\one_G(Y)}.
\]
To this end, since the characteristic function on an open set can be  approximated from below by a sequence of continuous bounded functions, we only need to check that for arbitrary $f,g\in\cont_b(E)$,
\[
\E\geklam{f(X)g(Y)}
=
\E\geklam{f(X)}
\E\geklam{g(Y)}.
\]
This follows immediately from our assumption.
\end{proof}

The following result shows that adding a symmetric independent random variable increases the $L^p$-norm (up to a constant).
Recall that an $E$-valued random variable $X$ is \emph{symmetric} if $X$ and $-X$ have the same distribution.

\begin{lemma}\label{lem:qb-ind-symm}
Let $X$ and $Y$ be two $E$-valued random variables. If $Y$ is symmetric and independent of $X$, then for all $p\in (0,\infty)$
it holds that
\begin{equation}\label{eq:qb-ind-symm}
\nnrm{X}{L^p(\Omega;E)}
\leq
2^{\frac{1-(r\minsym p)}{r\minsym p}}
\nnrm{X+Y}{L^p(\Omega;E)}.
\end{equation}
% Moreover, for $0<p<r$ we have
% \[
% \nnrm{X}{L^p(\Omega;E)}
% \leq
% 2^{\frac{1-p}{p}}
% \nnrm{X+Y}{L^p(\Omega;E)}.
% \]
\end{lemma}

\begin{proof}
Follow the proof of \cite[Proposition~6.1.5]{HytNeerVerWeis2017}, thereby using the quasi-triangle inequality instead of the triangle inequality.
\end{proof}

The following example shows that the first inequality in Lemma~\ref{lem:qb-ind-symm} is sharp.

\begin{example}\label{ex:qb-ind-sym}
Consider, for $r\in (0,1]$,
$(E,\nrm{\cdot}):=(\ell^r_2,\left\| \cdot \right\|_{\ell^r_2})$
(i.e.,  the space $\bR^2$ endowed with the $r$-norm  $\nnrm{(x_1,x_2)}{\ell^r_2}^r:=\abs{x_1}^r + \abs{x_2}^r$ for $(x_1,x_2)\in \bR^2$).
% $(\ell^r(\{0,1\}),\left\| \cdot \right\|_{r})$ with $r \in (0,1)$.
Let $e_1 := (1, 0)$ and $e_2 := (0, 1)$ be the standard unit vectors.
Let $x := e_1+e_2$ and $y := e_1-e_2$.
Let $\varepsilon_1, \varepsilon_2$ be two independent $\{-1,1\}$--valued identically distributed random variables with $\P(\varepsilon_1=-1)=\P(\varepsilon_1=1)=1/2$.
For all $p\in (0,\infty)$ it holds that
$\E\| \eps_1 x \|_{\ell^r_2}^{p} = \|x\|_{\ell^r_2}^p = 2^{\frac{p}{r}}$, whereas
$\E \nnrm{\varepsilon_1x+\varepsilon_2y}{\ell^r_2}^p = \E \left\| (\eps_1 + \eps_2, \eps_1-\eps_2) \right\|_{\ell^r_2}^p=2^p$.
% So that
% \[\E\|\varepsilon_1 x+\varepsilon_2 y\|^p = 2^p = 2^{p\rklam{1-\frac{1}{r}}}2^{\frac{p}{r}}
% =
% 2^{p\rklam{1-\frac{1}{r}}} \E\|\varepsilon_1 x\|_{r}^p.\]
Therefore, the constant $2^{\frac{1-(r\minsym p)}{r\minsym p}}$ in Lemma~\ref{lem:qb-ind-symm}
is optimal for $p\geq r$. To see that the constant is also optimal for $p<r$
consider the $\R$-valued random variables  $X:=\eps_1$ and $Y:= \eps_2$.

Note moreover that even in the Gaussian case
one cannot, in general, omit the constant $2^{\frac{1-(r\minsym p)}{r\minsym p}}$
in~\eqref{eq:qb-ind-symm}.
Indeed, let $e_1$ and $e_2$ be as above, let
% $x=4(e_1+e_2)$, $y=e_1-e_2$, let
$\gamma_1$
and $\gamma_2$ be two independent standard normal
Gaussian random variables and define $X:=\frac{1}{4}\gamma_1 (e_1+e_2)$
and $Y := \frac{1}{16} \gamma_2 (e_1-e_2)$. Then
$ \E \nnrm{ X}{\ell^{1/2}_2}^2= 1$,
whereas numerical integration yields
$ \E \nnrm{ X + Y }{\ell^{1/2}_2}^2 \approx 0.987$.
\end{example}

For the remainder of this section we introduce the following:
\begin{setting}\label{set:iidsums}
Let $(X_j)_{j\in\bN}$ be a sequence of independent $E$-valued random variables and
for all $n\in \N$ let $S_n:=\sum_{j=1}^n X_j$.
\end{setting}

The following version of L\'evy's inequality about tail estimates holds in $r$-Banach spaces and plays an important role in our proofs. We refer to \cite[Lemma~2.2]{CoxVer2011} for a detailed proof.

\begin{lemma}[L\'{e}vy's inequality for $r$-Banach spaces]\label{lem:QBLevy:inequal}
Assume Setting~\ref{set:iidsums} and assume in addition that $X_j$ is symmetric for all $j\in\N$. Then for $n\in\N$ and $t>0$ one has
\begin{equation}\label{lem:QBLevy:inequal1}
\P\sgrklam{\max_{1\leq k\leq n}\nnrm{S_k}{E}>t}\leq 2 \P\sgrklam{\nnrm{S_n}{E}>2^{1-\frac{1}{r}}t}
\end{equation}
and
\begin{equation}\label{lem:QBLevy:inequal2}
\P\sgrklam{\max_{1\leq k\leq n}\nnrm{X_k}{E}>t}\leq 2 \P\sgrklam{\nnrm{S_n}{E}>2^{1-\frac{1}{r}}t}.
\end{equation}
Consequently, for all $p\in (0,\infty)$ it holds that
\begin{equation}
 \max
 \ssggklam{
  \E  \sgeklam{\max_{1\leq k\leq n}\nnrm{S_k}{E}^p},
  \E  \sgeklam{\max_{1\leq k\leq n}\nnrm{X_k}{E}^p}
 }
 \leq 2^{1+\frac{p}{r} - p} \E \| S_n \|_E^p.
\end{equation}

\end{lemma}

The following result generalizes L\'evy's theorem on the convergence of sums of independent real valued random variables to the $r$-Banach space setting.

\begin{proposition}[L\'evy]\label{prop:Levy}
Assume Setting~\ref{set:iidsums}. The following are equivalent:
\begin{enumerate}[label=\textup{(\alph*)}]
\item\label{prop:Levy:as} $(S_n)_{n\in\N}$ converges almost surely to a random variable $S$.
\item\label{prop:Levy:prob} $(S_n)_{n\in\N}$ converges in probability to a random variable $S$.
\item\label{prop:Levy:weak} $(\P_{S_n})_{n\in\N}$ converges weakly to a probability measure $\mu$.
\end{enumerate}
In this case, the limits $S=\limn S_n$ in \ref{prop:Levy:as} and \ref{prop:Levy:prob}   coincide and $\P_S = \mu$.

%Moreover, if these equivalent conditions hold and $\E\geklam{\nnrm{S}{E}^p}<\infty$ for some $0<p<\infty$, then
%Moreover, if $\sup_{n\in\N} \E\geklam{\nnrm{S_n}{E}^p}<\infty$ for some $0<p<\infty$, then $$\E\sgeklam{\sup_{n\in\N}\nnrm{S_n}{E}^p}<\infty$$ and \eqref{prop:Levy:prob}--\eqref{prop:Levy:weak} are equivalent to
%%note that $\sup_{n\in \N} \|S_n\|_{L^p(\Omega;X)}<\infty$ is equivalent to $S\in L^p$ (See Hoffmann-J).
%\begin{enumerate}
%\item[{\rm (iv)}] $(S_n)_{n\in\N}$ converges in $L^p(\Omega;E)$ to a random variable $S$.
%\end{enumerate}
\end{proposition}

\begin{proof}
Let us first consider the equivalence of \ref{prop:Levy:as} and \ref{prop:Levy:prob}.
In a first step, we assume that each random variable $X_j$ is symmetric. Then the assertion can be proven in the same way as in \cite[Theorem~V.2.1]{VakTarCho1987} (which deals with the Banach space case), i.e., by using L\'{e}vy's inequality \eqref{lem:QBLevy:inequal1} from Lemma~\ref{lem:QBLevy:inequal}.
In a second step, we can drop the symmetry assumption by using the same symmetrization technique as in the proof of \cite[Theorem~V.2.2]{VakTarCho1987}. Note that in the $r$-Banach space setting, the symmetrization inequality from \cite[Proposition~V.2.2]{VakTarCho1987} holds with $t/2$ replaced by $2^{-1/r}t$ on the right hand side.

The equivalence of \ref{prop:Levy:prob} and \ref{prop:Levy:weak} can be obtained directly without assuming the symmetry of the random variables, by repeating the arguments of the proof of \cite[Theorem~V.2.3]{VakTarCho1987}. One mainly has to use  Prokhorov's theorem and the fact that the space of probability measures on a separable $r$-Banach space $E$, endowed with the convolution operator and the weak topology, forms a topological semigroup where the neutral element is the Dirac measure $\delta_0$, see \cite[Proposition~I.4.5]{VakTarCho1987}.
%Clearly, (iv) implies \eqref{prop:Levy:prob} due to Chebyshev's inequality.
%The reverse as well as the fact that $\sup_{n\in\N} \E\geklam{\nnrm{S_n}{E}^p}<\infty$ implies $\E\sgeklam{\sup_{n\in\N}\nnrm{S_n}{E}^p}<\infty$ can be obtained as in \cite[Theorem~4.2(d) and Theorem~5.4]{Hof77} by applying Lemma~\ref{lem:QBLevy:inequal}.
\end{proof}

Under reasonable integrability assumptions, the equivalent convergence types \ref{prop:Levy:as}--\ref{prop:Levy:weak} from Proposition~\ref{prop:Levy} turn out to be equivalent to convergence in $L^p$.
The following proposition holds.

\begin{proposition}[Hoffmann-J\o{}rgensen]\label{prop:Levy:Lp}
Assume Setting~\ref{set:iidsums}, and assume in addition that $0<p<\infty$ and that $(S_n)_{n\in\bN}$ converges almost surely to a random variable $S$. Then the following are equivalent:
\begin{enumerate}[label=\textup{(\roman*)}]
\item\label{it:HoJo1} $\displaystyle\sup_{n\in\N} \E\sgeklam{\nnrm{S_n}{E}^p}<\infty$.
\item\label{it:HoJo2} $\displaystyle\E\sgeklam{\sup_{n\in\N}\nnrm{S_n}{E}^p}<\infty$.
\item\label{it:HoJo3} $\displaystyle\E\sgeklam{\nnrm{S}{E}^p}<\infty$.
\item\label{it:HoJo4} $\displaystyle\E\sgeklam{\nnrm{S}{E}^p}<\infty$ and $(S_n)_{n\in\N}$ converges in $L^p(\Omega;E)$ to $S$.
\end{enumerate}
Conversely, if $(S_n)_{n\in\N}$ converges in $L^p(\Omega;E)$ to a random variable $S$, then $(S_n)_{n\in\N}$ converges almost surely to $S$.
\end{proposition}
\begin{proof}
First assume that each $X_j$ is symmetric ($j\in\N$). Then, given the almost sure convergence of $(S_n)_{n\in\N}$, the equivalence of~\ref{it:HoJo1}--\ref{it:HoJo4} can be obtained as in~\cite[Theorem~II.4.2(d) and Theorem~II.5.4]{Hof77} by applying Lemma~\ref{lem:QBLevy:inequal}.
In order to get rid of the symmetry assumption, one can use standard symmetrization techniques in a similar manner as done in the proof of Proposition~\ref{prop:Levy}.

Conversely, if $(S_n)_{n\in\N}$ converges in $L^p(\Omega;E)$,  Markov's inequality yields convergence in probability. Thus, by Proposition~\ref{prop:Levy}, almost sure convergence follows.
\end{proof}

We have already mentioned that the dual of a quasi-Banach space might be trivial. Obviously, in this case, we cannot expect that, e.g., the $\P$-a.s. convergence of $(\langle S_n,x^*\rangle)_{n\in\bN}$ for all $x^*\in E^*$ implies the $\P$-a.s. convergence of $(S_n)_{n\in\bN}$.
However, we can guarantee this implication if we assume two things: Firstly, that $E^*$ is rich enough to separate the points of $E$ and, secondly, that each random variable $X_j$ is symmetric for $j\in\N$.
Note that the symmetry assumption is needed whenever we are dealing with infinite dimensional spaces, even for Hilbert spaces, see \cite[Remark after Theorem~4.1]{ItoNis1968}.
We can prove even more and obtain the following generalization of the so-called \Ito-Nisio theorem from Banach to $r$-Banach spaces. Note that, in analogy to the Banach space setting, we can introduce the \emph{Fourier transform} of a Borel probability measure on $E$ to be the complex valued function on $E^*$ given by the formula
\[
\hat{\mu}(x^*) := \int_{E} e^{-i\langle x,x^*\rangle}\,\mathrm{d}\mu(x), \quad x^*\in E^*.
\]
Standard arguments show that, if $E^*$ separates the points of $E$, the Fourier transform is unique, i.e., if two Borel probability measures $\mu$ and $\nu$ on $E$ have the same Fourier transform, then this measures coincide on $\mathcal{B}(E)$.

\begin{proposition}[\Ito-Nisio]\label{prop:QB:conv:test:as}

Assume Setting~\ref{set:iidsums} and assume in addition that $E^*$ separates the points of $E$ and that $X_j$ is symmetric for all $j\in\N$. Then
statements \ref{prop:Levy:as}--\ref{prop:Levy:weak} from Proposition~\ref{prop:Levy} are equivalent to each of the following
three statements:
\begin{enumerate}[label=\textup{(\alph*)}, start=4]
\item\label{prop:ItoNis:tight} $(\P_{S_n})_{n\in\N}$ is uniformly tight.
\item\label{prop:ItoNis:dualweak} There exists an $E$-valued random variable $S$, such that $\displaystyle \lim_{n\to\infty}\langle S_n,x^*\rangle = \langle S,x^*\rangle$ in probability for all $x^*\in E^*$.
\item\label{prop:ItoNis:FT} There exists a Borel probability measure $\mu$ on $E$, such that $\displaystyle \lim_{n\to\infty} \widehat{\P_{S_n}}(x^*)=\hat{\mu}(x^*)$ for all $x^*\in E^*$.
\end{enumerate}
\end{proposition}

\begin{proof}
These equivalences can be proven by following the lines of the proof of \cite[Theorem~4.1]{ItoNis1968}. Note that, when proving `\ref{prop:ItoNis:tight}$\Rightarrow$\ref{prop:Levy:as}', one needs \cite[Theorem~3.2]{ItoNis1968} saying that the uniform tightness of $(\P_{S_n})_{n\in\N}$ implies the existence of a sequence $(c_n)_{n\in\N}\subseteq E$ such that $S_n-c_n$ converges $\P$-almost surely. But this turns out to hold also in our setting, if one repeats the arguments from the proof given therein.

Also note that when adapting the proof of \cite[Theorem~4.1]{ItoNis1968} to prove `\ref{prop:ItoNis:FT}$\Rightarrow$\ref{prop:ItoNis:dualweak}' we encounter a difficulty: the sequence  $(z_n)_{n\in\bN}$ in the proof of \cite[Theorem~4.1]{ItoNis1968} is constructed by means of the Hahn-Banach theorem. However, as $E$ is a separable metric space, it is also Lindel\"of. This can be used to show that if  $E^*$ separates the  points of $E$, then there exists a sequence $(z_n)_{n\in \N}$ in $E^*$, which separates the points of $E$, too, see, e.g., \cite[Proposition B.1.11]{HytNeerVerWeis2016}. But this suffices for the proof,  see also the implication `\ref{prop:ItoNis:FT}$\Rightarrow$\ref{prop:ItoNis:dualweak}' presented in \cite[Theorem~V.2.4]{VakTarCho1987}, as well as the proof of  \cite[Theorem~IV.2.5]{VakTarCho1987}.
\end{proof}
%\begin{proposition}\label{prop:Levy:Lp:old}
%Let ...$S$.$S_n$., $S\in L^p$.
%Assume $S_n\to S$ in probability.  $S\in L^p(\Omega;E)$ for some $p\in (0,\infty)$, then the above are equivalent to
%\begin{enumerate}
%\item[(4)] $S_n\in L^p(\Omega;X)$ for all $n\in \N$ and $S_n\to S$ in $L^p(\Omega;E)$.
%\end{enumerate}
%
%\end{proposition}

In the sequel, we frequently make use of the following Kahane-Khintchine inequality for Rade\-ma\-cher and Gaussian sums.
Recall from  the `Notation' part in the introduction  that $(\varepsilon_n)_{n\in \N}$ denotes a Rademacher sequence,
and $(\gamma_n)_{n\in\N}$ denotes a Gaussian sequence.

\begin{theorem}[Kahane-Khintchine inequality]\label{thm:QBKahKhin}
For all $p,q\in (0,\infty)$ there exists a constant $\const_{p,q,r}$, depending only on the parameters $p$, $q$, and $r$, such that for all finite sequences $x_1,\ldots,x_N\subseteq E$ we have
\begin{equation}\label{eq:QBKahKhin:Rade}
\rrklam{\E\reklam{\ssgnnrm{\sum_{n=1}^N \varepsilon_n\, x_n}{E}^p}}^{\frac{1}{p}}
\leq \const_{p,q,r}
\rrklam{\E\reklam{\ssgnnrm{\sum_{n=1}^N \varepsilon_n\, x_n}{E}^q}}^{\frac{1}{q}},
\end{equation}
and
\begin{equation}\label{eq:QBKahKhin:Gauss}
\rrklam{\E\reklam{\ssgnnrm{\sum_{n=1}^N \gamma_n\, x_n}{E}^p}}^{\frac{1}{p}}
\leq \const_{p,q,r}
\rrklam{\E\reklam{\ssgnnrm{\sum_{n=1}^N \gamma_n\, x_n}{E}^q}}^{\frac{1}{q}}.
\end{equation}
\end{theorem}

\begin{proof}
The inequality \eqref{eq:QBKahKhin:Rade} for Rade\-ma\-cher sums can be proven by adapting Kahane's original proof (see \cite[Theorem~2.1]{Kal1981}).
The inequality \eqref{eq:QBKahKhin:Gauss} can be obtained from \eqref{eq:QBKahKhin:Rade} by a (finite dimensional) central limit argument as used, e.g., in the proof of \cite[Corollary~4.8]{MarPis1981}, see also \cite[page~103]{LedTal}.
\end{proof}

The Kahane-Khintchine inequalities guarantee that the convergence of Gaussian and Rademacher series in probability implies their $L^p$-convergence. The following holds.

\begin{proposition}\label{prop:ItoNisio}
For a sequence $(x_n)_{n\in\bN}\subseteq E$, the following are equivalent:
\begin{enumerate}[label=\textup{(\roman*)}]
\item\label{Gaussian:series:prob} The series $\sum_{n=1}^\infty \gamma_n x_n$ converges in probability.
\item\label{Gaussian:series:Lp:all} The series $\sum_{n=1}^\infty \gamma_n x_n$ converges in $L^p(\Omega;E)$ for all $0<p<\infty$.
\item\label{Gaussian:series:Lp:some} The series $\sum_{n=1}^\infty \gamma_n x_n$ converges in $L^p(\Omega;E)$ for some $0<p<\infty$.
\end{enumerate}
These equivalences also hold if all instances of $\sum_{n=1}^\infty \gamma_n x_n$ above are replaced by
$\sum_{n=1}^\infty \varepsilon_n x_n$.
\end{proposition}

\begin{proof}
The implication `\ref{Gaussian:series:Lp:all}$\Rightarrow$\ref{Gaussian:series:Lp:some}' is trivial and `\ref{Gaussian:series:Lp:some}$\Rightarrow$\ref{Gaussian:series:prob}' is classic due to Markov's inequality.
Thus, we only have to check `\ref{Gaussian:series:prob}$\Rightarrow$\ref{Gaussian:series:Lp:all}'. But this can be proven by mimicking the proof of the corresponding result for Banach spaces, see, for instance, the proof of \cite[Proposition~6.4.5]{AlbKal2006} and the subsequent example. An alternative strategy relying on the Paley-Zygmund inequality can be found, e.g., in \cite[Corollaries 6.4.2 and 6.4.4]{HytNeerVerWeis2017}.
\end{proof}

We close this section with  a  simple criterion for convergence of bounded random sums,  which we use to prove  Proposition~\ref{prop:gamma_sum_radon}  later on.

As for Banach spaces, an $r$-Banach space is said to have \emph{(Rademacher) cotype} $q\in[2,\infty]$ if there exists a constant $C_q\in (0,\infty)$ such that for all finite sequences $x_1,\ldots,x_N\in E$ we have
\[
\ssgrklam{\sum_{n=1}^N \nnrm{x_n}E^q}^{\frac{1}{q}}
\leq C_q
\ssgrklam{\E\ssgnnrm{\sum_{n=1}^N \varepsilon_n x_n}{E}^2}^{\frac{1}{2}};
\]
with the usual modification of $\max_{n=1\ldots,N}\nnrm{x_n}E$ replacing the sum on the left hand side for $q=\infty$ (see also, e.g., \cite{BasUri1986,Mal2004} as well as~\cite{AlbKal2006, LedTal}).
Similarly, one can define \emph{type $p$} for $p\in (0,2]$, but this notion is less useful for $r$-Banach spaces with $0<r<1$ (see~\cite{Kal1981}). The following function spaces are classical examples of $r$-Banach spaces with finite cotype:
\begin{itemize}
\item For $p\in (0,\infty)$, the space $L^p$ has cotype $p\vee 2$ (see \cite{Mal2004}). This result remains true for non-commutative $L^p$-spaces (see \cite[Corollary~5.8]{PisXu}).
\item One easily checks that if $X$ has cotype $q\in [2,\infty)$, $p\in (0,\infty)$,
and $(S,\cA,\mu)$ is a $\sigma$-finite measure space, then $L^p(S;X)$ has
cotype $p\vee q$. Consequently, for $p,q\in (0,\infty)$ and $s\in \R$, the Besov space $B^{s}_{p,q}$ has cotype $p\vee q\vee 2$.
\end{itemize}

%For the definition of cotype and basic properties we refer, e.g., to~\cite{AlbKal2006, LedTal} (the definition of cotype for quasi-Banach spaces is entirely analogous to the definition for Banach spaces).
\begin{proposition}\label{prop:HJKalternative}
The following assertions are equivalent:
\begin{enumerate}[leftmargin=*,align=right,label=\textup{(\roman*)}]
\item \label{it:asconvsymmsums} For all sequences  $(X_n)_{n\in \N}$  of independent symmetric $E$-valued random variables satisfying  $\P\grklam{ \sup_{n\in \N} \gnnrm{ \sum_{k=1}^{n}  X_k }{E} < \infty } = 1$
one has that $\sum_{k=1}^{\infty} X_k $ converges $\P$-almost surely.
\item \label{it:randomizedsequence} For all sequences $(x_n)_{n \in \N}\subseteq E$
satisfying  $\P\grklam{ \sup_{n\in \N} \gnnrm{ \sum_{k=1}^{n} \varepsilon_k x_k }{E} < \infty } = 1$
one has that $\limn x_n = 0$.
\end{enumerate}
In particular, both assertions hold if $E$ has finite cotype.
\end{proposition}
\begin{proof}
The equivalence of \ref{it:asconvsymmsums} and \ref{it:randomizedsequence} can be proven as in the Banach space case,
see e.g.~\cite[Theorem 4.2]{Neerven2010}.
To see that finite cotype implies \ref{it:randomizedsequence} let $q\in [2,\infty)$ and $C_q( E)\in [0,\infty)$
be such that  $E$  has cotype $q$ with cotype constant $C_q( E)$. Assume $\grklam{\sum_{k=1}^{n} \varepsilon_k x_k}_{n\in \N}$ is almost surely bounded and hence bounded in probability. Choose $\lambda>0$ such that for all $n\geq 1$,
\[\P\Big(\Big\|\sum_{k=1}^{n} \varepsilon_k x_k\Big\|_E>\lambda\Big)\leq \frac{1}{4 C_{2q,q,r}^2},\]
where $C_{2q,2,r}$ is the constant in the Kahane-Khintchine inequality (Theorem~\ref{thm:QBKahKhin}).
By \cite[Corollary 6.2.9]{HytNeerVerWeis2017} (which remains valid in the $r$-Banach space setting), we obtain that
\[
  \sup_{n\in \N}
  \Big\|
    \sum_{k=1}^{n} \varepsilon_k x_k
  \Big\|_{L^q(\Omega;E)}
  \leq 2^{1/q} \lambda.
\]
By definition of cotype %$q$
% and Lemma~\ref{lem:qb-ind-symm} (Sonja: this was in the original version but it does not make sense to me. Petru: Okay, but actually you would have to argue why lim=sup in this case, since you only have Lq-boundedness and not Lq-convergence. However, you can just replace lim by sup from the beginning and then everything is fine. That's what I've done.)
we find that
\begin{align*}
  \sum_{n\in \N}\|x_n\|^q_E
&  =
 \sup_{N\in\bN}
%  \lim_{N\to \infty}
  \sum_{n=1}^N \|x_n\|^q_E
  \leq
  [C_q( E)]^q
 \sup_{N\in\bN}
% \lim_{N\to \infty}
 \E
  \ssgnnrm{
    \sum_{n=1}^N \varepsilon_n x_n
  }{E}^q
  < \infty.
\end{align*}
Therefore, $\lim_{n\to \infty} \|x_n\|^q_E = 0$, and this yields~\ref{it:randomizedsequence}.
\end{proof}

\subsection{Gaussian random variables}
In this section $E$ denotes a separable $r$-Banach space for some $r\in (0,1]$.
If $r=1$, i.e., if $E$ is a Banach space, an $E$-valued random variable $X$ is usually called Gaussian if $\lb X,x^*\rb$ is Gaussian for every $x^*\in E^*$.
This makes sense, as the points of normed spaces can be separated by their duals due to Hahn-Banach theorem.
However, as already mentioned, $r$-Banach spaces do not necessarily have a rich dual; it can even be trivial.
Nevertheless, we can introduce the notion `Gaussian random variable' by using what is known as Bernstein's characterization of Gaussian measures.

\begin{definition}\label{def:Gauss:Bernstein}
An $E$-valued random variable $X$ is \emph{Gaussian} if the random variables $X+Y$ and $X-Y$ are independent for any independent copy $Y$ of $X$.
\end{definition}

\begin{remark}
The fact that if $E=\R$ this definition coincides with what is usually known to be a real-valued Gaussian random variable has been proven by Bernstein~\cite{Ber1941} and independently by Kac~\cite{Kac1939} (see \cite{BauerEN, Bryc, Feller} for historical comments and proofs).
An analogous definition has been also suggested for introducing Gaussian measures on more general structures like abelian groups, see, e.g., \cite{Byc1979,Fel1993}.
\end{remark}

\begin{remark}
In the construction of the stochastic integral we present below, all Gaussian random variables are symmetric. Therefore, throughout the paper, if not explicitly stated otherwise, when we say `Gaussian' we mean `symmetric Gaussian'. Note that, as in the Banach space setting, if $X$ is a Gaussian random variable (not necessarily symmetric) then there exists a symmetric Gaussian random variable $Y$ and a point $x_0\in E$ such that $X=Y+x_0$, see, for instance,~\cite[Corollary~3.2]{Byc1981}.
\end{remark}

Due to L\'evy's theorem on the equivalence of convergence in distribution and convergence of the characteristic functions for real-valued random variables, it is easy to prove that the convergence in distribution preserves Gaussianity. This is also the case for $r$-Banach space valued Gaussian variables.

\begin{lemma}\label{lem:gausslimit}
Let $(X_n)_{n\in\N}$ be a sequence of $E$-valued Gaussian random variables converging in distribution to an $E$-valued random variable $X$. Then $X$ is Gaussian.
\end{lemma}
\begin{proof}
Let $Y$ be an independent copy of $X$ and for all $n\in \N$
let $Y_n$ be an independent copy of $X_n$. Then $(X_n+Y_n, X_n-Y_n)$
converges in distribution to $(X+Y,X-Y)$. Moreover, for all $n\in \N$
one has that $X_n+Y_n$ and $X_n - Y_n$ are independent, whence by Lemma~\ref{lem:Conv:pres:ind}
it holds that $X+Y$ and $X-Y$ are independent.
%
%By the Skorohod coupling theorem (see \cite{Kal}) we may assume $\xi_n\to \xi$ almost surely.
%Let $((\eta_n)_{n\in \N}, \eta)$ be an independent copy of $((\xi_n)_{n\in \N}, \xi)$. Then also $\eta_n\to \eta$ almost surely. We need to show that
%$\xi + \eta$ is independent of $\xi-\eta$.
%
%For  $f,g\in C_b(E)$ by dominated convergence and the independence of $\xi_n+\eta_n$ and $\xi_n-\eta_n$, we have
%\begin{align*}
%\E\big(f(\xi+\eta) g(\xi-\eta)\big) & = \limn \E\big(f(\xi_n+\eta_n) g(\xi_n-\eta_n)\big)
%\\ & = \limn \E(f(\xi_n+\eta_n)) \E(g(\xi_n-\eta_n))=  \E(f(\xi+\eta)) \E(g(\xi-\eta))
%\end{align*}
%Now the proof can be finished in a standard way. Let $A,B\subset E$ be open sets. Choosing, sequences $f_n, g_n\in C_b(E)$ such that $f_n\uparrow \one_A$ and $g_n\uparrow \one_B$ we find
%\[\P(\xi+\eta\in A, \xi-\eta\in B) = \P(\xi+\eta\in A) \P(\xi-\eta\in B).\]
\end{proof}

The following generalization of Fernique's theorem holds for Gaussian random variables on $r$-Banach spaces, see, for instance, \cite{ByZ,JKRR,Zapala}.

\begin{theorem}\label{thm:Fernique}
Let $X$ be an $E$-valued Gaussian random variable. Then there exists an $\alpha>0$ such that
\[
\E\geklam{\mathrm{exp}({\alpha\nnrm{X}{E}^2})}<\infty.
\]
\end{theorem}

\subsection{Gaussian random variables in \texorpdfstring{$r$}{r}-Banach spaces with separating dual}
\label{ssec:Gaussian_separating_dual}

One can say more about Gaussian random variables if $E$ has a separating dual.
We start with the fact that, in this case, Definition~\ref{def:Gauss:Bernstein} extends the usual definition of Gaussian random variables via duality, see also~\cite{Bog} for the Banach space analogue. The results in this subsection will not be used in the rest of the paper.

%It is worth mentioning that Definition~\ref{def:Gauss:Bernstein} extends the usual definition of Gaussian random variables on Banach spaces via duality (see \cite{Bog}). Even more, if the dual $E^*$ separates the points of an $r$-Banach space $E$, then the following holds.

\begin{proposition}\label{prop:weakGauss}
Assume that $E^*$ separates the points of $E$. An $E$-valued random variable $X$ is Gaussian if and only if $\lb X,x^*\rb$ is Gaussian for all $x^*\in E^*$.
\end{proposition}
\begin{proof}
The `only if' part is an easy consequence of Bernstein's theorem, see~\cite{Ber1941}. The `if' part follows by the fact that if $E^*$ separates the points of $E$, then the Fourier transform is unique as $\mathcal{B}(E)=\sigma(x^*\colon x^*\in E^*)$, see Proposition~\ref{prop:Pettis} and Example~\ref{ex:Boreldual}.
\end{proof}

In the Banach space setting, Karhunen-Lo\`eve expansions can be used to extend many results for Gaussian sums, like, for instance, Theorem~\ref{thm:QBKahKhin} and Proposition~\ref{prop:ItoNisio}, to Gaussian random variables.
Unfortunately, such an expansion is not available for general $r$-Banach spaces. However, if $E^*$ separates the points of $E$, the following result follows from~\cite[Theorem~1]{Byc1987}.

%Stated with our notation, he proves the following: Let $X\colon\Omega\to E$ be a Gaussian random variable and let $\mu:=\P_X$. Then the space of all Borel measurable functions $h\colon E\to\R$, such that
%\[
%h(x+y)=h(x)+h(y)\quad\mu\otimes\mu\textup{-a.s.}
%\quad\text{and}\quad
%h(-x)=-h(x)\quad\mu\textup{-a.s.}
%\]
%is called the space of additive functionals. It is denoted by $E^2(\mu)$ and is considered as a subspace of $L^2(E,\cB(E),\mu;\R)$. As such it is a Hilbert space and every element $h\in E^2(\mu)$ is a symmetric Gaussian random variable.
%
\begin{proposition}\label{prop:KarhunenLoeve}
Assume that $E^*$ separates the points of $E$. Let $X\colon \Omega\to E$ be a Gaussian random variable. Then there exists a Gaussian sequence $(\gamma_n)_{n\in \N}\subseteq L^2(\Omega)$ and a sequence $(x_n)_{n\in\N}\subseteq E$ such that
$
\sum_{n=1}^\infty \gamma_n x_n
$
converges $\P$-almost surely to a Gaussian random variable $Y$, which has the same distribution as $X$.
\end{proposition}
\begin{proof}
Using the notation from~\cite{Byc1987}, it is easy to see that if $E^*$ separates the points of $E$, then there exists a sequence of elements of $E^2(\P_X)$ separating points of $E$ (mod $\P_X$), as $E^*\subseteq E^2(\P_X)$, see also the proof of Proposition~\ref{prop:QB:conv:test:as}.
Moreover, for every orthonormal basis $(h_n)_{n\in \N}$ of $E^2(\P_X)$, the random variables $(h_n(X))_{n\in\N}$ form a Gaussian sequence.
Thus, the assertion is just a very particular instance of~\cite[Theorem~1]{Byc1987}.
\end{proof}

\begin{remark}
It is not immediately clear from \cite{Byc1987} whether in the proof of Proposition~\ref{prop:KarhunenLoeve} one may choose the orthonormal basis $(h_n)_{n\in\bN}$ out of $E^*$. It is also unclear whether it is possible to choose $(h_n)_{n\in\bN}$ in such a way that $\sum_{n=1}^\infty h_n(X)\,x_n=X$ $\P$-a.s.
\end{remark}

Now we explain how to use Proposition~\ref{prop:KarhunenLoeve} prove the generalizations of Theorem~\ref{thm:QBKahKhin} and Proposition~\ref{prop:ItoNisio} to Gaussian random variables.
Note, however, that these results are not needed in any other proof in this paper.

\begin{proposition}\label{prop:QBKahKhi_RV}
Assume that $E^*$ separates the points of $E$. Then, for every Gaussian random variable $X\colon\Omega\to E$ and arbitrary $0<p,q<\infty$,
\begin{equation}
\grklam{\E\nnrm{X}{E}^p}^{1/p}
\leq C_{p,q,r}
\grklam{\E\nnrm{X}{E}^q}^{1/q}
\end{equation}
with $C_{p,q,r}$ as in Theorem~\ref{thm:QBKahKhin}.
\end{proposition}
\begin{proof}
This is an immediate consequence of Proposition~\ref{prop:KarhunenLoeve}, Theorem~\ref{thm:QBKahKhin} and Proposition~\ref{prop:ItoNisio}.
\end{proof}

\begin{proposition}\label{prop:Gaussian_RV_conv}
Assume that $E^*$ separates the points of $E$.
Let $(X_n)_{n\in\N}$ be a sequence of $E$-valued Gaussian random variables. Then the following are equivalent:
\begin{enumerate}[label=\textup{(\roman*)}]
\item\label{Gaussian:RV:prob} The sequence $(X_n)_{n\in\N}$ converges in probability.
\item\label{Gaussian:RV:Lp:all} The sequence $(X_n)_{n\in\N}$ converges in $L^p(\Omega;E)$ for all $0<p<\infty$.
\item\label{Gaussian:RV:Lp:some} The sequence $(X_n)_{n\in\N}$ converges in $L^p(\Omega;E)$ for some $0<p<\infty$.
\end{enumerate}
\end{proposition}

\begin{proof}
%This can be proven, for instance, with one of the strategies mentioned in the proof of Proposition~\ref{prop:ItoNisio} by exploiting Proposition~\ref{prop:QBKahKhi_RV}.
Since Proposition~\ref{prop:QBKahKhi_RV} holds one can extend the argument in the Banach space setting of \cite[Lemma 2.1]{RosSuc1980}.
\end{proof}

%TODO Check whether we can do something with this later on:
%Convergence in probability implies convergence in $L^p$ whenever Kahane-Khincthine holds. (see Rosinski-Suchanecki)
%
%Karhunen-Lo\`eve expansion: Habe ich das überhaupt wenn Punkte separiert werden k\"onnen?
%
%\begin{proposition}
%Assume $E$ has separating dual $E^*$. For a random variable $\xi\colon \O\to E$ the following are equivalent. Karhune-Lo\`eve
%\end{proposition}
%
%Open problems:
%Karhune-Lo\`eve for general Gaussian random variables.
%Kahane--Khintchine general Gaussians, concentration of measures (see \cite{LedTal})
%
%Zapala3: and Byczkowski RKHS.

\section{\texorpdfstring{$\gamma$}{g}-norms and square functions}
\label{sec:gamma}

In this section  we first introduce and prove fundamental properties of  so-called $\gamma$-radonifying operators $R\colon H\to E$, where $H$ is a separable Hilbert space and $E$ is a separable $r$-Banach space  for some fixed $r\in (0,1]$.
These operators play a key role in the development of stochastic \Ito\ integrals in the (quasi-)Banach space setting.
For details and  historical remarks  on this class of operators in the Banach space setting, we refer the reader to the survey paper~\cite{Neerven2010}.
In the final  part of this section we consider the special case of $r$-Banach function spaces $E=E(S)$ over a measure space $(S,\cA,\mu)$. In addition, we  replace  $H$ by $L^2(\tilde{S};H)$, where $(\tilde{S}, \tilde\cA, \tilde\mu)$ is a measure
space such that $L^2(S;H)$ is separable, and connect $\gamma$-radonifying operators $ R\colon L^2(\tilde S; H)\to E(S)$ to the theory of (generalized) square functions (see~\cite{KalWei2016}). Our presentation mostly follows~\cite[Chapter~9]{HytNeerVerWeis2017}.

% Throughout this section we assume the notation introduced in Section~\ref{ssec:notation}.
\smallskip

\noindent \textbf{Notation.}
Throughout this section, $H$ denotes a separable Hilbert space, $E$ denotes a separable $r$-Banach space, $r\in (0,1]$ is fixed,
and $(\gamma_n)_{n\in\N}$ denotes a Gaussian sequence.  %sequence of independent real-valued standard Gaussian random variables.

\subsection{\texorpdfstring{$\gamma$}{g}-summing and \texorpdfstring{$\gamma$}{g}-radonifying operators}

For a linear operator $R\colon H\to E$ set
$$
\gnnrm{R}{\gamma_\infty(H,E)}
:=
\sup_h \rrklam{\E \ssgnnrm{\sum_{n=1}^N \gamma_n \,R h_n}{E}^2}^{\!1/2} \in [0,\infty],$$
where the supremum is taken over all finite orthonormal systems
$h=\{h_1,\dots,h_N\}$ in $H$.
The operator $R$ is called {\em $\gamma$-summing} if $\n R\n_{\gamma_\infty(H,E)}<\infty$. We write $\gamma_\infty(H,E)$ for the set of $\gamma$-summing operators from $H$ to $E$.
By considering singletons $\{h\}$ we see that
every $\gamma$-summing operator is bounded and $\n R\n_{\linop{H}{E}}\leq \n R\n_{\gamma_\infty(H,E)}$.
 This can be used  to show that the space $\gamma_\infty(H,E)$, endowed with the $r$-norm $\left\n \cdot \right\n_{\gamma_\infty(H,E)}$,  is an $r$-Banach space.

\begin{remark}
Let $0<p<\infty$ and write $\gsum_\infty^p(H,E)$ for the set of all linear operators $R\colon H\to E$ such that \label{gammasump}
\begin{align*}
\gnnrm{R}{\gsum_\infty^p(H,E)}
:=
\sup_h
\ssgrklam{\E \ssgnnrm{\sum_{n=1}^N \gamma_n R h_n}{E}^p}^{1/p} <\infty,
\end{align*}
the supremum being taken over all finite orthonormal systems $h=\{h_1,\ldots,h_N\}$ in $H$.
Then, by the Kahane-Khintchine inequality for Gaussian sums, see Theorem~\ref{thm:QBKahKhin}, for all $p\in (0,\infty)$ it holds that
$\gamma_\infty^p(H,E) = \gamma_{\infty}(H,E)$ and
\begin{align}\label{eq:equiv_gamma_p}
\const_{2,p,r}^{-1} \left\| R \right\|_{\gsum_\infty(H,E)}
\leq \left\| R \right\|_{\gsum_\infty^p(H,E)}
\leq \const_{p,2,r} \left\| R \right\|_{\gsum_\infty(H,E)}
\end{align}
for all $R\in \gamma_{\infty}(H,E)$.
\end{remark}

%In particular, this implies the simple but very useful result that for every $x^*\in E^*$ we have
%\begin{equation}\label{eq:unif-norm-gamma2}
% \n T^*x^*\n\le \|T\|\, \|x^*\| \leq \n T\n_{\gamma_\infty^p(H,E)} \|x^*\|.
%\end{equation}

%\subsection{Finite rank operators}

For $h\in H$ and $x\in E$ we denote by $h\otimes x$ the rank one operator in $\calL(H,E)$ defined by
$$ (h\otimes x)g := \lb g,h\rb_{H} x, \qquad g\in H.$$
The linear span of these operators is usually denoted by $H\otimes E$ and can be viewed as the subspace of finite rank operators in $\calL(H,E)$.
Every such operator can be represented in the form
\begin{equation}\label{eq:finite-rank}
R = \sum_{n=1}^N h_n\otimes x_n
\end{equation}
with  $N\in\bN$, $(h_n)_{n=1}^N$ orthonormal in $H$ and $(x_n)_{n=1}^N$ a finite sequence in $E$.
It is known that, if $E$ is a Banach space (i.e., if $r=1$), then the $\gamma$-summing norm $\nnrm{R}{\gamma_\infty(H,E)}$ of a finite rank operator of this form is equal to the $L_2(\Omega;E)$-norm of $\sum_{n=1}^N\gamma_n x_n$, see, e.g.,~\cite[Proposition~9.1.3]{HytNeerVerWeis2017}.
For general $r$-Banach spaces with $r\in (0,1]$ we obtain the following generalization.

% Petru: I changed the order here, hoping that it makes things more fluent.

\begin{proposition}\label{prop:f-r}%\label{lem:f-r}
If $R=\sum_{n=1}^N h_n\otimes x_n$ is a finite rank operator with $N\in\bN$, $\{h_1,\ldots,h_N\}$ orthonormal in $H$ and $x_1,\ldots,x_N\in E$, then $R\in\gsum_\infty(H,E)$ and for all $0<p<\infty$ we have
\begin{equation}\label{eq:tradinONS}
\ssgnnrm{\sum_{n=1}^N \gamma_n x_n}{L^p(\Omega;E)}
\leq
\ssgnnrm{\sum_{n=1}^N h_n\otimes x_n}{\gsum_\infty^p(H,E)}
\leq
2^{\frac{1-(r\minsym p)}{r\minsym p}} \,
\ssgnnrm{\sum_{n=1}^N \gamma_n x_n}{L^p(\Omega;E)}.
\end{equation}
\end{proposition}

For the proof of this two-sided estimate, we first prove the following auxiliary result. It turns out to be very useful when it comes to handling $\gamma_\infty$-norms of finite rank operators.
%In order to be able to handle $\gamma_{\infty}$-norms of finite rank operators we first present the following key lemma.
Recall that, for arbitrary $n\in\N$, we write $\ell^2_n$ for the space $\R^n$ endowed with the Euclidean norm.
% For an $m\times n$ matrix $A$, we write $\nrm{A}$ for its spectral norm, i.e., $\nrm{A}:=\nnrm{A}{\calL(\ell^2_n,\ell^2_m)}$.

\begin{lemma}\label{lem:idealmatrix}
Let $0<p<\infty$ and fix $m,n\in\N$. For all $m\times n$ matrices $A = (a_{ij})_{i,j=1}^{m,n}$ and all finite sequences $(x_j)_{j=1}^n$ in $E$,
\begin{equation}\label{eq:idealmatrix}
\ssgnnrm{\sum_{i=1}^m \gamma_i \sum_{j=1}^n    a_{ij} x_j}{L^p(\Omega;E)}
\leq 2^{\frac{1-(r\minsym p)}{r\minsym p}}\nnrm{A}{\calL(\ell^2_n,\ell^2_m)} \ssgnnrm{\sum_{j=1}^n \gamma_j x_j}{L^p(\Omega;E)}.
\end{equation}
\end{lemma}
In the Banach space setting, i.e., if $r=1$,  Lemma~\ref{lem:idealmatrix} has been proven in several ways. Most proofs rely either on a duality argument using covariance operators or on the fact that every matrix $A$ with $\nrm{A}\leq 1$ is a convex combination of unitary operators. Clearly both methods fail in the $r$-Banach space setting for $0<r<1$.
However, as shown in~\cite[Proposition~6.1.23]{HytNeerVerWeis2017}, Lemma~\ref{lem:idealmatrix} can also be proven from basic principles which turn out to remain valid for $0<r<1$ and $0<p<1$.
For the convenience of the reader, we provide the proof with minor modifications.
%The  constant $C_{r,p}=2^{\frac{1-(r\minsym p)}{r\minsym p}}$ on the right hand side is due to an application of Lemma~\ref{lem:qb-ind-symm}, which, as shown in Example~\ref{ex:qb-ind-sym} does not hold  see also Remark~\ref{rem:idealmatrix:no1} below.

% \begin{proof}[Proof of Lemma~\ref{lem:idealmatrix}]
\begin{proof}[Proof of Lemma~\ref{lem:idealmatrix}]
(See also the proof of Proposition~6.1.23 in \cite{HytNeerVerWeis2017}.)
We may assume that $\nnrm{A}{\calL(\ell^2_n,\ell^2_m)}= 1$. By adding zero terms to the matrix
$A = (a_{ij})_{i,j=1}^{m,n}$, we may also assume that $m=n$.
Let $B$ be the $2n\times 2n$-matrix given by
\[
B=\left(
\begin{array}{cc}
A & (I-AA^*)^{1/2} \\
(I-A^* A)^{1/2} & -A^*
\end{array}\right).
\]
Using the simple fact that $A(I-A^* A)^{1/2} = (I-AA^*)^{1/2}A$ we find that $B^* B = I$.
Writing $B=(b_{ij})_{i,j=1}^{2n}$ we have  $a_{ij} = b_{ij}$ for $1\leq i,j\leq n$. By Lemma~\ref{lem:qb-ind-symm},
\begin{align*}
\E \ssgnrm{\sum_{i=1}^n  \gamma_i \sum_{j=1}^n a_{ij} x_j }^p
%&
\leq
2^{\frac{(1-(r\minsym p))p}{r\minsym p}}\, \E \ssgnrm{\sum_{i=1}^{2n}   \gamma_i \sum_{j=1}^{n}   b_{ij} x_j }^p
=
2^{\frac{(1-(r\minsym p))p}{r\minsym p}}\, \E \ssgnrm{\sum_{j=1}^{n} G_j x_j }^p,
\end{align*}
where $G_j = \sum_{i=1}^{2n} \gamma_i b_{ij}$, for $1\leq j\leq n$. The result now follows from the fact that $(G_j)_{j=1}^n$
is a sequence of independent real-valued standard Gaussian random variables.
\end{proof}

\begin{remark}\label{rem:idealmatrix:no1}
Contrary to the Banach space setting, the constant
$C_{r,p}=2^{\frac{1-(r\minsym p)}{r\minsym p}}$ in Inequality~\eqref{eq:idealmatrix}
cannot be replaced by $C=1$: Indeed, setting
$(E,\nrm{\cdot})=(\ell^{1/2}_2,\left\| \cdot \right\|_{\ell^{1/2}_2})$, $n=2$, $m=1$, $e_1 = (1,0)$, $e_2=(0,1)$,
$x_1=\frac{1}{4}(e_1+e_2)$, $x_2=\frac{1}{16}(e_1-e_2)$, and $A = (1\  0)$,
we recover the Gaussian case in Example~\ref{ex:qb-ind-sym}.
\end{remark}

Lemma~\ref{lem:idealmatrix} allows for the following short proof of Proposition~\ref{prop:f-r}.
\begin{proof}[Proof of Proposition~\ref{prop:f-r}]
Let $\{u_1,\ldots,u_M\}$ be an arbitrary orthonormal system in $H$. Let $A$ be the $M\times N$ matrix given by $a_{m n} = \lb u_m, h_n\rb$.
It is easily checked that $\nrm{A}\leq 1$ and therefore, by Lemma~\ref{lem:idealmatrix},
\begin{align*}
\ssgnnrm{\sum_{m=1}^M \gamma_m R u_m}{L^p(\Omega;E)}
=
\ssgnnrm{\sum_{m=1}^M \gamma_m \sum_{n=1}^N a_{mn} x_n}{L^p(\Omega;E)}
\leq
2^{\frac{1-(r\minsym p)}{r\minsym p}}
\ssgnnrm{\sum_{n=1}^N \gamma_n x_n}{L^p(\Omega;E)}.
\end{align*}
This shows that $R\in \gamma_{\infty}(H,E)$ and that the second estimate in \eqref{eq:tradinONS} holds.
To prove the first estimate in \eqref{eq:tradinONS} it suffices to use the orthonormal system $\{h_1,\dots,h_N\}$.
\end{proof}
%

%\subsection{The space \texorpdfstring{$\gamma(H,E)$}{}}\label{sec:gamma-radonif}

In Proposition~\ref{prop:f-r} we have seen that every finite rank
operator from $H$ to $E$ belongs to $\gamma_\infty(H,E)$.
It is demonstrated in e.g.~\cite{NeeVerWei2007, NeervenWeis2005a} that the closure of the space of finite rank operators in $\gamma_\infty(H,E)$ is the right space to study vector valued stochastic \Ito\ integrals, see also Sections~\ref{sec:stoch_int_functions} and~\ref{sec:stoch_int_processes} below.

%\todo{short comment on why we introduce this class}

\begin{definition}\label{def:g-rad}
The space $\gamma(H,E)$ is defined as the closure of the space of  finite rank operators
in $(\gamma_\infty(H,E),\nnrm{\cdot}{\gamma_\infty(H,E)})$. The operators in $\gamma(H,E)$ are called {\em $\gamma$-radonifying}.
\end{definition}
\begin{remark}
By definition, the space $\gamma(H,E)$ is an $r$-Banach space with respect to the $r$-norm inherited from $\gamma_\infty(H,E)$. For simplicity, we write
\[
\nnrm{R}{\gamma(H,E)}:= \nnrm{R}{\gamma_\infty(H,E)}, \qquad R\in\gamma(H,E),
\]
and, more general,
\[
\nnrm{R}{\gamma^p(H,E)}:= \nnrm{R}{\gamma_\infty^p(H,E)}, \qquad R\in\gamma(H,E),
\]
for $0<p<\infty$.
\end{remark}

%\subsection{The ideal property}\label{subsec:ideal}

For the case that $E$ is a Banach space it is well-known that the spaces $\gamma_{\infty}(H,E)$ and $\gamma(H,E)$ are operator
ideals in $L(H,E)$. This property plays an important role in the construction of stochastic \Ito\ integrals in Banach spaces (see~\cite{NeeVerWei2007,NeervenWeis2005a}).
By using Lemma~\ref{lem:idealmatrix}, we can use verbatim the same proof as presented, e.g., in~\cite[Chapter~9]{HytNeerVerWeis2017}, to
verify that the ideal property also holds in the case that $E$ is an $r$-Banach space with $0<r<1$. Indeed, we obtain the following.

\begin{theorem}[Ideal property]\label{thm:ideal}\index{ideal property!of $\gamma$-radonifying operators}
Let $R\in\gamma_\infty(H,E)$. Furthermore, let $F$ be a separable $r_0$-Banach space for some $r_0\in (0,1]$ and let $G$ be a separable Hilbert space.
%In addition to the $r$-Banach space $E$, let $F$ be an $r_0$-Banach space for some $r_0\in(0,1]$, and let $G$ be a Hilbert space.
Then for all $V\in\calL(G,H)$
and $U\in\calL(E,F)$ we have $URV\in\gamma_\infty(G,F)$, and
for all $0< p<\infty$,
\begin{equation}\label{eq:ideal}
\nnrm{ U R V }{\gamma_\infty^p(G,F)}
\leq
2^{\frac{1-(r\minsym p)}{r\minsym p}}
\nnrm{U}{\calL(E,F)} \, \nnrm{R}{\gamma_\infty^p(H,E)} \,\nnrm{V}{\calL(G,H)}.
\end{equation}
If, moreover, $R\in \gamma(H,E)$, then $URV\in \gamma(G,F)$.
\end{theorem}

\begin{proof}

Mimic the proof of~\cite[Theorem~9.1.10]{HytNeerVerWeis2017} and use Lemma~\ref{lem:idealmatrix} where needed.

%Let $(g_m)_{m=1}^M$ be an orthonormal system in $G$ and let
%$(h_n)_{n=1}^N$ in $H$ be an orthonormal basis for $\operatorname{span}(\{Sg_m:  1\leq m\leq M\})$.
%Let $A$ be the $M\times N$ matrix given by $a_{mn} = \lb Sg_m,h_n\rb$. From
%$Sg_m = \sum_{n=1}^N  a_{mn} h_n$
%one easily deduces that $\nnrm{A}{\calL(\ell^2_N,\ell^2_M)}\leq \|S\|_{\calL(G,H)}$.
%It follows from Lemma~\ref{lem:idealmatrix} that
%\begin{align*}
%\ssgnnrm{\sum_{m=1}^M \gamma_m URS g_m}{L^p(\Omega;F)}
%& =
%\ssgnnrm{U \sum_{m=1}^M \gamma_m \sum_{n=1}^N a_{mn} R h_n}{L^p(\Omega;F)}\\
%& \leq
%2^{\frac{1-(r\minsym p)}{r\minsym p}} \nnrm{U}{\calL(E,F)}\, \nnrm{A}{\calL(\ell^2_N,\ell^2_M)} \, \ssgnnrm{\sum_{n=1}^N \gamma_n R h_n}{L^p(\Omega;E)}\\
%& \leq
%2^{\frac{1-(r\minsym p)}{r\minsym p}} \nnrm{U}{\calL(E,F)} \, \nnrm{S}{\calL(G,H)} \, \nnrm{R}{\gamma_\infty^p(H,E)}.
%\end{align*}
%Hence $URS\in\gamma_\infty(G,F)$ and the required estimate follows.
%
%Next assume that $R\in \gamma(H,E)$. Choose finite rank operators $(R_n)_{n\in\N}$
%such that $R_n\to R$ in $\gamma(H,E)$. Then by the estimate just proven, we have
%\begin{align*}
%\nnrm{U R_n S - URS}{\gamma_{\infty}(G,F)}
%& =
%\nnrm{U (R_n-R) S }{\gamma_{\infty}(G,F)}\\
%& \leq
%2^{\frac{1-r}{r}}  \nnrm{U}{\calL(E,F)} \, \nnrm{R_n-R}{\gamma_{\infty}  (H,E)}\,\nnrm{S}{\calL(G,H)},
%
%\end{align*}
%the right-hand side of which tends to zero as $n\to\infty$.
%Since each $U R_n S$ is a finite rank operator, the result follows.
\end{proof}

\begin{remark}
Close inspection of the proof of Theorem~\ref{thm:ideal} reveals that the constant $2^{\frac{1-(r\minsym p)}{r\minsym p}}$ in \eqref{eq:ideal} is only due to the right multiplication, i.e., in the setting of Theorem~\ref{thm:ideal}, if $H=G$ and $V={\rm Id}$, then
\begin{equation*}
\n U R \n_{\gamma_\infty^p(H,F)} \leq \n U\n_{\calL(E,F)} \, \n R\n_{\gamma_\infty^p(H,E)} .
\end{equation*}
\end{remark}

\subsection{Convergence in \texorpdfstring{$\gamma(H,E)$}{}}

We have seen that the $\gamma$-radonifying property of an operator is preserved under left and right multiplication.
Our next result considers the situation where we multiply with
convergent sequences of operators.

\begin{theorem}\label{thm:ga-conv}
Let $R\in\gamma(H,E)$. Furthermore, let $F$ be a separable $r_0$-Banach space for some $r_0\in (0,1]$ and let $G$ be a separable Hilbert space.
Let $V, V_1, V_2, \dots \in  \calL(G,H)$ and $U, U_1, U_2, \dots \in \calL(E,F)$ be such that
\begin{enumerate}[label=\textup{(\roman*)}]
\item $\limn V\s_n h=  V\s h$ for all $ h\in H$,
\item $\limn U_n x= U x$ for all $x\in E$.
\end{enumerate}
Then $\limn U_n R V_n = URV$ in $\gamma(G,F)$.
\end{theorem}
\begin{proof}
This result can be obtained by the same strategy as in the proof of~\cite[Theorem~9.1.14]{HytNeerVerWeis2017}, i.e., by using the ideal property (Theorem~\ref{thm:ideal}) together with the uniform boundedness principle (which holds on arbitrary $r$-Banach spaces, see, e.g.,~\cite{KalPecRob1984}).
\end{proof}

The following two sample applications of this result are frequently used in forthcoming proofs.
%\todo{\emph{Mark}: Next two examples should maybe be left out or put in other proofs. \emph{Sonja}: I need the first one. \emph{Petru}: I would opt to keep them.}

\begin{example}[Approximation]\label{ex:g-approx}
Let $(h_n)_{n\in\N}$ be an orthonormal basis for $H$
and for all $n\in \N$ let $P_n$ denote the orthogonal projection
onto $\operatorname{span}(\{h_1,\dots, h_n\})$.
Then for all  $R\in \gamma(H,E)$ we have
$\limn RP_n =R$ in $\gamma(H,E)$.
\end{example}

\begin{example}[Strong measurability]\label{ex:g-meas}
Let $(S,  \cA, \mu)$ be a measure space. For a function $\Phi\colon S\to \gamma(H,E)$ and $h\in H$ define $\Phi h\colon S\to E$ by
$(\Phi h)(\xi):=  \Phi(\xi)h$, $\xi\in S$.
Then by employing the same arguments as in the proof of~\cite[Example~9.1.16]{HytNeerVerWeis2017} it can be shown that  the following assertions are equivalent:
\begin{enumerate}[label=\textup{(\roman*)}]
\item\label{it:measurable:strong} $\Phi$ is strongly measurable.
\item\label{it:measurable:weakstrong} $\Phi h$ is strongly measurable for all $h\in H$.
\end{enumerate}
%It suffices to prove that~\ref{it:measurable:weakstrong} implies~\ref{it:measurable:strong}. If $(h_n)_{n\in \N}$ is an orthonormal
%basis for $H$, then with the notations of Example~\ref{ex:g-approx} for
%all $s\in S$ we have
%\[
%\Phi(s) =\limn \Phi(s) P_n = \limn \sum_{j=1}^n \lb \,\cdot\,, h_j\rb\,\Phi(s)\,
%h_j,
%\]
%with convergence in $\gamma(H,E)$. The result now follows from the strong measurability of the functions $s\mapsto \Phi(s)h_j$ on the right-hand side.
\end{example}

We proceed with a characterization of
$\gamma$-summing and $\gamma$-radonifying operators in terms of orthonormal bases.

\begin{theorem}[Testing against an orthonormal basis]
\label{thm:summing-sep}
Let $(h_n)_{n\in\N}$ be an orthonormal basis for $H$.
\begin{enumerate}[leftmargin=*,align=right,label=\textup{(\roman*)}]
\item\label{thm:summing-sep:summing} An operator $R \in \calL(H,E)$ belongs to $\gamma_\infty(H,E)$ if and only if for some (equivalently, for all) $0<p<\infty$,
\begin{equation}\label{eq:QBgamsum:onb}
\sup_{N\in\N}\, \ssgnnrm{\sum_{n=1}^N \gamma_n\, Rh_n}{L^p(\Omega;E)}< \infty.
\end{equation}
In this case,
\begin{equation}\label{eq:gammsummingseq}
\ssgnnrm{\sum_{n=1}^N \gamma_n\, Rh_n}{L^p(\Omega;E)}
\leq
\nnrm{R}{\gamma_\infty^p(H,E)}
\leq
2^{\frac{1-(r\minsym p)}{r\minsym p}}  \sup_{N\in \N}\, \ssgnnrm{\sum_{n=1}^N \gamma_n\, Rh_n}{L^p(\Omega;E)}.
\end{equation}
\item\label{thm:summing-sep:rad}  An operator $R\in \calL(H,E)$ belongs to $\gamma(H,E)$ if and only if for some (equivalently, for all) $0<p<\infty$,
\[
\sum_{n=1}^\infty \gamma_n \,Rh_n \ \hbox{converges in $L^p(\Omega;E)$}.
\]
In this case, the sum converges almost surely and
\begin{equation}\label{eq:gammradseq}
\ssgnnrm{\sum_{n=1}^\infty \gamma_n\, Rh_n}{L^p(\Omega;E)}
\leq
\nnrm{R}{\gamma^p(H,E)}
\leq
2^{\frac{1-(r\minsym p)}{r\minsym p}} \ssgnnrm{\sum_{n=1}^\infty \gamma_n\, Rh_n}{L^p(\Omega;E)}.
\end{equation}
\end{enumerate}
\end{theorem}

\begin{proof}
We have proven all the ingredients we need in order to obtain this result by mimicking the proof of its Banach space analogue presented in~\cite[Theorem~9.1.17]{HytNeerVerWeis2017}: The Kahane-Khintchine inequality for Gaussian sums proven in Theorem~\ref{thm:QBKahKhin} guarantees that once we have the results for one $p$, we also have it for arbitrary $0<p<\infty$.
Part~\ref{thm:summing-sep:summing} follows from a simple application of Lemma~\ref{lem:idealmatrix} and Fatou's lemma.
The equivalence of the $\gamma$-radonifying norm of a finite rank operator with the $L^p(\Omega;E)$-norm of the corresponding Gaussian sum provided by Proposition~\ref{prop:f-r} in combination with Example~\ref{ex:g-approx} are the ingredients needed to prove~\ref{thm:summing-sep:rad}. The almost sure convergence follows from the extension of Hoffmann-J\o{}rgensen's  theorem presented in Proposition~\ref{prop:Levy:Lp}.
\end{proof}

Note that Theorem~\ref{thm:summing-sep} provides a simple proof of the main result in~\cite{BycIng} for the special (and more simple) case that the metric space $E$ is an $r$-Banach space.

\begin{corollary}
Let $R\in\calL(H,E)$.
If for some orthonormal basis $(h_n)_{n\in\N}$ of $H$, the series $\sum_{n=1}^\infty \gamma_n \,Rh_n$ converges in probability, then for every orthonormal basis $(h_n')_{n\in\N}$ of $H$, the series $\sum_{n=1}^\infty \gamma_n \,Rh_n'$ converges in $L^p(\Omega;E)$ for all $0<p<\infty$  and $\P$-almost surely.
Moreover, the distribution of the limit random variable does not depend on the particular basis.
\end{corollary}

\begin{proof}
The assumption and Proposition~\ref{prop:ItoNisio} yield that $\sum_{n=1}^\infty \gamma_n \,Rh_n$ converges in $L^p(\Omega;E)$ for all $0<p<\infty$. Therefore, $R\in \gamma(H,E)$ by Theorem~\ref{thm:summing-sep}. Applying  Theorem~\ref{thm:summing-sep} again with an arbitrary orthonormal basis $(h_n')_{n\in \N}$ we find that $\sum_{n=1}^\infty \gamma_n \,Rh_n'$ converges in $L^p(\Omega;E)$ for all $0<p<\infty$ and $\P$-almost surely.
The independence of the distribution of the series on the orthonormal basis can be easily verified for finite rank operators and then extended to arbitrary $R\in\gamma(H,E)$ by exploiting Theorem~\ref{thm:summing-sep}\ref{thm:summing-sep:rad}.
\end{proof}

In general, $\gamma(H,E)$ is a proper closed subspace of $\gamma_\infty(H,E)$.
In the case that $E$ is a Banach space, it was proven by Hoffmann-J{\o}rgensen and Kwapi\'en~\cite{HoffmannJorgensen1974,Kwapien1974}
that $\gamma(H,E)=\gamma_{\infty}(H,E)$ if $E$ fails to contain a copy of $c_0$. More precisely,
Hoffmann-J{\o}rgensen and Kwapie\'n proved that
for a Banach space $E$ it holds that $c_0$
fails to embed continuously into $E$ if and only if for certain $E$-valued random sums it holds that almost sure boundedness of the partial sums implies almost sure convergence. The proof uses a characterization of Bessaga and Pelczy\'nski, which seems to be unavailable for $r$-Banach spaces. Fortunately Proposition~\ref{prop:HJKalternative} in combination with Proposition~\ref{prop:ItoNisio}
implies the following quasi-Banach space analogue of
the result by Hoffmann-J{\o}rgensen and Kwapie\'n (one may
copy verbatim the  relevant parts of the  proof provided in~\cite[Theorem 4.3]{Neerven2010}).
\begin{proposition}\label{prop:gamma_sum_radon}
If $E$ has finite cotype, then $\gamma(H,E)= \gamma_\infty(H,E)$.
\end{proposition}

\subsection{\texorpdfstring{$\gamma$}{g}-norms in \texorpdfstring{$r$}{r}-Banach function spaces and square functions}

It is well-known that, if $E=E(S)$ is a Banach function space with finite cotype over a $\sigma$-finite measure space $(S,\cA,\mu)$, then  $\nnrm{\cdot}{\gamma(H,E)}$ has an equivalent square function interpretation, see, e.g.,~\cite[Theorem~9.3.6]{HytNeerVerWeis2017}.
In particular, $\gamma(H,E(S))\eqsim E(S;H)$, where $E(S;H)$ is the vector space of all (equivalence classes of) strongly measurable functions $f\colon S\to H$ such that $\xi\mapsto \nnrm{f(\xi)}{H}$ is an element of $E(S)$ (naturally normed by $\nnrm{f}{E(S;H)}:=\nnrm{\rklam{\xi\mapsto\nnrm{f(\xi)}{H}}}{E(S)}$, $f\in E(H;S)$).
In this section, we prove that this result extends verbatim to $r$-Banach function spaces with finite cotype (see Theorem~\ref{thm:sq_fnc_lattice} below).
%The case where $H$ is replaced by $L^2(\tilde{S};H)$ with  $(\tilde{S},\tilde{\cA},\tilde{\mu})$ being a $\sigma$-finite measure space such that $L^2(\tilde{S};H)$ is separable, is of particular interest for stochastic integration (see Corollary~\ref{cor:lattice}).

\smallskip

We begin with a definition of an $r$-Banach function spaces analogue to the definition
of a Banach function space as considered in e.g.~\cite[Definition F.3.1]{HytNeerVerWeis2017}. For details in the case $r=1$ we refer the reader to %\todo{In this definition, local integrability is assumed in addition to what we have and the spaces are called Köthe spaces.}
\cite[Definition 1.b.17]{LindTza1977}, \cite[Chapter~15]{Zaa1967}. In the case $0<r<1$ these spaces have been studied in \cite{Kalton2005, KalMS93}.

%\todo[inline]{In Zaanen's book a Banach function space has to have the Fatou property (see Section~30 therein). This is not necessarily the case for spaces that satisfy our definition. Counterexample: $c_0$.}
\begin{definition}
Let $(S,\cA,\mu)$ be a $\sigma$-finite measure space and let $r\in (0,1]$.
We say that a vector space $E(S)\subseteq L^0(S)$ is an \emph{$r$-Banach function space}
if it is complete under some $r$-norm $\lVert\cdot \rVert_{E(S)}\colon E(S)\rightarrow [0,\infty)$
and moreover for all $f\in L^0(S)$, $g\in E(S)$ such that $f\leq g$ it holds
that $f\in E(S)$ and $\| f \|_{E(S)} \leq \| g \|_{E(S)}$.
\end{definition}

Before we continue to our next result we recall some standard results for $r$-Banach function spaces. For details on $r$-convexity and $q$-concavity we refer to~\cite[Section 3]{KalMS93}, or, e.g., to~\cite{LindTza1977}).
If $E(S)$ has finite cotype, say cotype $t$, it satisfies a lower $t$-estimate. % THIS IS TRIVIAL
Therefore, it follows from \cite[Theorem 3.2]{KalMS93} that $X$ is $s$-convex for some $s\in (0,1]$.
%Since $E(S)$ is an $r$-Banach function space, it satisfies an upper $p$-estimate (see \cite[p. 141]{Kalton84}). Therefore, since $E(S)$ has finite cotype \cite[Theorem 2.2 and 4.1]{Kalton84} imply that $E(S)$ is $p$-convex for some $p>0$.
It follows from \cite[Theorem 4.4]{Mal2004} that $E(S)$ is $q$-concave for any $q>t$.
Therefore, as in \cite[Theorem 1.d.6]{LindTza1977} (using $s$-convexity to prove the part $\gtrsim$) one can prove the Maurey--Khintchine inequalities:
for all $0<p<\infty$, $n\in \N$, and $x_1,\ldots,x_n\in E(S)$ it holds that
\begin{equation}\label{eq:MaureyKhintchine}
\ssgnnrm{\sum_{n=1}^N \varepsilon_n x_n}{L^p(\Omega;E(S))}
\eqsim_{q,s,E(S)}
\ssgnnrm{\ssgrklam{\sum_{n=1}^N |x_n|^2}^{\frac{1}{2}}}{E(S)}.
\end{equation}
Moreover, in \eqref{eq:MaureyKhintchine} one can replace the Rademacher sum by a Gaussian sum as well. Indeed, this follows from \cite[Proposition~2.5]{Kalton2005}.
%: if an $r$-Banach space $E$ has finite coype then for all $0<p<\infty$, $n\in \N$, and $x_1,\ldots,x_n\in E$ it holds that
%\[
%\ssgnnrm{\sum_{n=1}^N \gamma_n x_n}{L^p(\Omega;E)} \eqsim_{p,E} \ssgnnrm{\sum_{n=1}^N \varepsilon_n x_n}{L^p(\Omega;E)}.
%\]
For completeness we mention that finite cotype is actually equivalent to $q$-concavity for some $q<\infty$. Indeed, the remaining implication follows from \cite[Proposition 1.f.3]{LindTza1977}.

The theorem below can be found in~\cite[Chapter~9]{HytNeerVerWeis2017} for the case that $r=1$. The proof extends to the quasi-Banach space case by taking into account that the Maurey-Khintchine inequalities remain valid in this setting, as explained above.
For an $r$-Banach function space $E(S)\subseteq L^0(S)$ we set
\begin{equation*}
E(S;H) :=\{ f\in L^0(S;H) \colon s\mapsto \| f(s) \|_{H} \in E(S) \},
\end{equation*}
which is an $r$-Banach space when endowed with the $r$-norm $\| f \|_{E(S;H)} = \| \|f\|_{H} \|_{E(S)}$, $f\in E(S;H)$.

%If $E=E(S)$ is an $r$-Banach function space and $H$ is a separable Hilbert space, we write $E(S;H)$ for the vector space of all (equivalence classes of) strongly measurable functions $f\colon S\to H$ such that $\xi\mapsto \nnrm{f(\xi)}{H}$ is an element of $E(S)$.
%$E(S;H)$ is naturally $r$-normed by $\nnrm{f}{E(S;H)}:=\nnrm{\rklam{\xi\mapsto\nnrm{f(\xi)}{H}}}{E(S)}$, $f\in E(H;S)$.
%

\begin{theorem}\label{thm:sq_fnc_lattice}
Let $0<r\leq 1$. Let $E(S)$ be a separable  $r$-Banach function  space with finite cotype over a $\sigma$-finite measure space $(S,\cA,\mu)$.
Then the mapping $U\colon  E(S;H)\to \calL(H,E(S))$ defined by
$$  (U f) h := \lb h, f(\cdot)\rb_{H}, \quad  h\in H,$$
defines an isomorphism $U$ of Banach spaces
$$ E(S;H)\eqsim \gamma(H,E(S)).$$
Furthermore,
for an  operator $R\in\calL(H,E(S))$ the following assertions are equivalent:
\begin{enumerate}[label=\textup{(\roman*)}]
\item\label{thm:sq_fnc_lattice:1} $R\in \gamma(H,E(S))$.
\item\label{thm:sq_fnc_lattice:2}  There exists a % a constant $C\ge 0$ such that
function $0\le g\in E(S)$ such that
for all
finite orthonormal systems $(h_n)_{n=1}^N\subseteq H$ we have
\[
\ssgrklam{\sum_{n=1}^N \abs{Rh_n}^2}^{\frac12}\le g  \ \ \mu\textup{-a.e.}
\]
\item\label{thm:sq_fnc_lattice:3}  There exists a function $0\le g\in E(S)$ such that for all $h\in H$ we have
\[
\abs{Rh}\le \nnrm{h}{H}\cdot g \ \ \mu\textup{-a.e.}
\]
\item\label{thm:sq_fnc_lattice:4}  There exists a function $k\in E(S;H)$ such that for all $h\in H$ we have
\[
R h = \lb h, k(\cdot)\rb_H \ \ \mu\textup{-a.e.}
\]
\item\label{thm:sq_fnc_lattice:5} The function
$\displaystyle\ssgrklam{\sumn \abs{R h_n}^2}^{\frac12}$ belongs to $E(S)$.
\end{enumerate}
In this situation, in~\ref{thm:sq_fnc_lattice:3} we may take $\displaystyle g= \ssgrklam{\sum_{n=1}^\infty \abs{R h_n}^2}^{\frac12}$ and we have
\begin{equation}\label{eq:Rsquarekappa}
\nnrm{R}{\gamma(H,E(S))}
\eqsim_{E}
\nnrm{k}{E(S;H)}
=
\ssgnnrm{\ssgrklam{\sum_{n=1}^\infty \abs{Rh_n}^2}^{\frac12}}{E(S)}.
\end{equation}
\end{theorem}

In the setting of stochastic integrals the case that $H$ in Theorem~\ref{thm:sq_fnc_lattice}
is an $L^2$-space is of particular relevance (see Example~\ref{example:stoch_int_fs}).
In this case we obtain the following corollary.

\begin{corollary}\label{cor:lattice}
Let  $0<r\leq 1$. Let $E(S)$ be a separable $r$-Banach function  space with finite cotype over a $\sigma$-finite measure space $(S,\cA,\mu)$.
Furthermore, let $(\tilde{S},\tilde{\cA},\tilde{\mu})$ be a $\sigma$-finite measure space such that $L^2(\tilde{S};H)$ is separable.
\begin{enumerate}[label=\textup{(\roman*)}]
\item\label{it:cor:lattice:1} If $\varphi\colon \tilde S \times S \to H$ is strongly $\tilde\cA\otimes\cA$-measurable and such that
\begin{equation}\label{eq:EL2H}
\ssgnnrm{\ssgrklam{\int_{\tilde S} \nnrm{\varphi(\tilde\xi,\cdot)}{H}^2 \,\mathrm{d}\tilde\mu(\tilde\xi)}^{\frac12}}{E(S)}<\infty,
\end{equation}
then $R\colon L^2(\tilde S;H)\to E(S)$ given by
\begin{equation}\label{eq:RfAform}
R f(\xi) = \int_{\tilde S} \langle \varphi(\tilde\xi,\xi),f(\tilde{\xi}) \rangle_H \,\mathrm{d}\tilde\mu(\tilde\xi),\qquad\xi\in S,
\end{equation}
for $f\in L^2(\tilde S;H)$,
is in $\gamma(L^2(\tilde S;H), E(S))$, and
\begin{equation}\label{eq:normequivlattice}
\nnrm{R}{\gamma(L^2(\tilde S;H), E(S))}
\eqsim_{E(S)}
\ssgnnrm{\ssgrklam{\int_{\tilde S} \nnrm{\varphi(\tilde\xi,\cdot)}{H}^2 \,\mathrm{d}\tilde\mu(\tilde\xi)}^{\frac12}}{E(S)}.
\end{equation}
\item\label{it:cor:lattice:2} Conversely, if $R\in \gamma(L^2(\tilde S;H), E(S))$, then we can find a $\tilde\mu\otimes\mu$-almost everywhere unique $\varphi$ such that~\eqref{eq:RfAform} holds.
\end{enumerate}
Moreover, if $E(S)$ has cotype $2$,
then given $\varphi$ such that~\eqref{eq:EL2H} holds we can define $\Phi\in L^2(\tilde S;\gamma(H,E(S)))$ by
$\Phi(\tilde\xi) h= \langle \varphi(\tilde\xi,\cdot), h\rangle_H$, $\tilde\xi\in\tilde S$, $h\in H$.
\end{corollary}

\begin{proof}
Part~\ref{it:cor:lattice:1} is an immediate consequence of Theorem~\ref{thm:sq_fnc_lattice} with $H$ replaced by $L^2(\tilde{S};H)$.
To prove part~\ref{it:cor:lattice:2}, we choose a sequence $(R_n)_{n\in\N}\subseteq L^2(\tilde{S};H)\otimes E(S)$ of finite rank operators such that $R_n\to R$ in $\gamma(L^2(\tilde S;H), E(S))$.
For every $n\in\N$, since $R_n$ is of finite rank, we can obviously find a strongly measurable $\varphi_n:S\times \tilde{S}\to H$ such that $R_n$ is given by~\eqref{eq:RfAform} with $\varphi$ replaced by $\varphi_n$. From \eqref{eq:normequivlattice} it follows that $(\varphi_n)$ is a Cauchy sequence and hence convergent to some $\varphi$ in $E(S;L^2(\tilde{S};H))$.
Thus,
\[R f = \lim_{n\to \infty} R_n f = \lim_{n\to \infty}\int_{\tilde{S}} \lb \varphi_n(\tilde{\xi},\cdot),f(\tilde{\xi})\rb_H d\mu(\tilde{\xi})  = \int_{\tilde{S}} \lb\varphi(\tilde{\xi},\cdot), f(\tilde{\xi})\rb_H d\mu(\tilde{\xi}),\]
where the convergence takes place in $E(S)$; the last step follows from the fact that evaluation against $f\in L^2(\tilde{S};H)$ is a bounded operator from $E(S;L^2(\tilde{S};H))$ into $E(S)$ by the Cauchy-Schwarz inequality. The fact that we can choose $\varphi$ to be strongly $\tilde{\cA}\otimes \cA$-measurable can be obtained by mimicking the proof of \cite[Lemma~9.3.7]{HytNeerVerWeis2017}.
We conclude that \eqref{eq:RfAform} holds. The $\tilde{\mu}\otimes\mu$-uniqueness follows by a standard Fubini argument.

For the final assertion note that cotype $2$ implies $2$-concavity by \eqref{eq:MaureyKhintchine}.
Therefore, the result follows from the continuous version of $2$-concavity which is characterized by (see the proof of \cite[Theorem 3.9(2)]{Ver13emb})
\begin{equation*}
\nnrm{ \Phi }{L^2(\tilde S;\gamma(H,E(S)))}
\eqsim_{E(S)}
\nnrm{\varphi}{L^2(\tilde S;E(S;H))}
\lesssim_{E(S)}
\ssgnnrm{\ssgrklam{\int_{\tilde S} \nnrm{\varphi}{H}^2 \,\mathrm{d}\tilde\mu }^{\frac12}}{E(S)}.\qedhere
\end{equation*}

\end{proof}

\begin{remark}
If $E$ is a quasi-Banach space isomorphic to a subspace of an $r$-Banach function space, then one can deduce a similar characterization as above. Indeed, for a characterization of $\gamma(L^2(\tilde{S};H);E)$ in the case that $E=B^{\sigma}_{p,q}(\R^d)$ (i.e., a Besov space which is not a Banach function space) and $\tilde{S} = [0,T]$ we refer to Proposition~\ref{prop:Bpqgamma}.
\end{remark}

\begin{remark}\label{rem:QBFS:integral:det}
Note that, in particular, if $\tilde\mu(\tilde S)<\infty$, Corollary~\ref{cor:lattice} provides an example of a class of functions
$\Phi\colon \tilde S\to \calL(H,E(S))$ for which an integral
$\int_{\tilde S} \Phi \,\mathrm{d}\tilde \mu := R\one_{\tilde S}$ exists in $\calL(H,E(S))$.
Recall from the introduction that some integration theory in $r$-Banach spaces
is available, see~\cite{Vog1967}, but that the results in~\cite{AlbAns2013} show that the integral does
not obey the usual properties of Bochner integrals.
\end{remark}

% \begin{example} %Sonja: this is the same example as~\ref{cor:lattice}.
% Let $E$ be an $r$-Banach function space with finite cotype. Let $H = L^2(S)$ for some measure space $(S,\cA,\mu)$. Then
% letting for $\varphi\in E(L^2(S))$,
% \[T f = \int_S \varphi(\cdot, s) f(s)\, \,\mathrm{d}\mu\]
% we obtain an isomorphism $\gamma(L^2(S),E) \eqsim E(L^2(S))$.
% Indeed, let $\sum_{n=1}^N f_n\otimes x_n \in L^2(S)\otimes E$, where $(f_n)_{n=1}^N$ is an orthonormal system. Then
% \[\Big\|\sum_{n=1}^N f_n\otimes x_n\Big\|_{\gamma(L^2(S),E)} = \Big\|\Big(\sum_{n=1}^N |x_n|^2\Big)^{\frac12}\Big\|_E \]
% On the other hand, viewing $\sum_{n=1}^N f_n\otimes x_n$ as a function from $S$ into $E$, we have
% \[\Big\|\Big(\int_{S}\Big|\sum_{n=1}^N f_n\otimes x_n\Big|^2 \,\mathrm{d}\mu\Big)^{\frac12}\Big\|_E = \Big\|\Big(\sum_{n=1}^N |x_n|^2\Big)^{\frac12}\Big\|_E.\]
% By a density argument this result extends to all of $\gamma(L^2(S),E)$ and $E(L^2(S))$
% \end{example}

% \subsection{Integral operators and \texorpdfstring{$\gamma$}{g}-radonifying operators}

\section{Stochastic integration I: deterministic integrands}%Stochastic integration of functions
\label{sec:stoch_int_functions}
The theory developed in the previous sections enables us to extend the notion of a stochastic integral from Banach spaces to quasi-Banach spaces by closely following the ideas developed in~\cite{NeeVerWei2007,NeervenWeis2005a}, see also~\cite{NeeVerWei2015} for a survey.
In this section we assume that the integrand is deterministic, i.e., we develop what is sometimes called a Wiener integral.
%
%In this section we define the stochastic integral of a function taking values in a quasi-Banach space $E$.
%In the previous sections we developed all we need in order to extend the stochastic integral in Banach spaces developed in~\cite{NeeVerWei2007, NeervenWeis2005a} to quasi-Banach spaces.
%Building on the theory we developed in the previous sections,
%we construct a stochastic integral of \Ito\ type with a similar strategy as
%
%closely follow the theory for stochastic integration in Banach spaces -- see~\cite{NeerVerWeis2012} for a survey.
%More specifically,
We first construct an abstract stochastic integral for $\gamma$-radonifying operators
on a separable  Hilbert space $\cH$ with values in separable a quasi-Banach space $E$  (Subsection~\ref{ssec:stoch_int_def}). We then proceed in Subsection~\ref{ssec:stoch_int_process}
to consider the case $\cH=L^2(0,T;H) := L^2([0,T];H)$,
where $T\in (0,\infty)$ and $H$ is a separable  Hilbert space\textemdash in this case the abstract stochastic integral introduced
in Subsection~\ref{ssec:stoch_int_def} can be interpreted as a continuous, $E$-valued stochastic process.
Examples~\ref{example:stoch_int_fs} and~\ref{example:stoch_int_besov} provide characterizations of functions on $[0,T]$
for which an $E$-valued stochastic integral is well-defined in the case that $E$ is an $r$-Banach function space or a Besov space.\par

Explicit characterizations of stochastically integrable $\cL(H,E)$-valued functions can be
given if $E$ has a separating dual. This is the topic of Section~\ref{ssec:stoch_int_sepdual}. More specifically, we show that the notion of $\gamma$-radonifying operators as defined above coincides with the classical notion of $\gamma$-radonifying operators which goes back to Gel'fand~\cite{Gel1955}, Segal~\cite{Seg1956}, and Gross~\cite{Gro1962, Gro1967}\textemdash see~\cite{Neerven2010} for more historical remarks.
%As a consequence, we obtain a weak (Pettis type) characterization of the stochastic integral.
As a consequence, we can define a weak (Pettis type) stochastic \Ito\ integral for functions $\Phi\colon [0,T]\to\cL(H,E)$ as in the Banach space setting.
% \par
% Finally, as we are specifically interested in the quasi-Banach Besov spaces (which do not have a separating dual),
% Subsection~\ref{ssec:stochInt_det_nodual}
% we explain how this problem can be handled in case we do not have a separating dual but know more about the structure of the $r$-Banach space, is presented  for two examples that are relevant in applications.

\smallskip

\noindent\textbf{Notation.} Throughout this section, $E$ is a separable $r$-Banach space for some $0<r\leq 1$, $\cH$ and $H$ are separable Hilbert spaces and $(\Omega,\cF,\P)$ is a probability space. Moreover, $\mathcal{W}_{\cH}\colon \cH \rightarrow L^2(\Omega)$ is an $\cH$-isonormal
process (see e.g.~\cite[Definition~2.1]{NeeVerWei2015}).
If $\cH=L^2(0,T;H)$, we write $W_H$ instead of $\mathcal W_{L^2(0,T;H)}$.

\subsection{An abstract definition of the stochastic integral.}\label{ssec:stoch_int_def}
% \textbf{Setting:}
%Throughout this section, $E$ is a separable $r$-Banach space ($r\in (0,1]$),
%$\cH$ is a separable Hilbert space, $(\Omega,\cF,\P)$
%is a probability space, and $\mathcal{W}_{\cH}\colon \cH \rightarrow L^2(\Omega)$ is a $\cH$-isonormal
%process (see e.g.~\cite[Definition 2.1]{NeeVerWei2015}).

The results obtained in Section~\ref{sec:gamma} allow us to introduce the notion of stochastic integration with respect to an $\cH$-isonormal process $\mathcal{W}_{\cH}$ for $\gamma$-radonifying operators. We start with finite rank operators.
\begin{definition}
Let $R\in \cH\otimes E$, i.e., $\displaystyle R=\sum_{n=1}^N h_n\otimes x_n$ with $\{h_1,\ldots,h_N\}$ orthonormal in $\cH$ and $x_1,\ldots,x_N\in E$ for some $N\in\N$. Then
\[
R \cdot \mathcal{W}_{\cH}:=\sum_{n=1}^N \mathcal{W}_{\cH}(h_n)x_n \in L^2(\Omega;E)
\]
is called \emph{the stochastic integral of $R$ with respect to $\mathcal{W}_{\cH}$}.
\end{definition}

\begin{remark}\label{rem:stochInt:f-r}
\begin{enumerate}[leftmargin=*,align=right,label=\textup{(\roman*)}]
% \item Note that, by making slight abuse of notation, we write $\mathcal{W}_{\cH}$ for both, the stochastic integral of a finite rank operator with respect to an $\cH$-isonormal process and for the $\cH$-isonormal process itself.
% This can not lead to any confusion, since for any $h\in \cH$, $\mathcal{W}_{\cH}(h)$ can be considered as the stochastic integral of the finite rank operator $h \otimes 1 \in \cH\otimes \R$ with respect to $\mathcal{W}_{\cH}$.
\item The stochastic integral of a finite rank operator $R\in \cH\otimes E$ with respect to an $\cH$-isonormal Wiener process is well-defined, i.e., if \[R=\sum_{n=1}^N h_n\otimes x_n=\sum_{m=1}^M h_m'\otimes x_m'\] with different orthonormal systems $\{h_1,\ldots,h_N\}$ and $\{h_1',\ldots,h_m'\}$ in $\cH$ and $x_1,\ldots,x_N,x_1',\ldots,x_M'\in E$ for some $N,M\in\N$, then
\[
\sum_{n=1}^N \mathcal{W}_{\cH}(h_n)x_n
=
\sum_{m=1}^M \mathcal{W}_{\cH}(h_m')x_m'\qquad (\text{in } L^2(\Omega;E)).
\]
This can be seen by using the linearity of the $\cH$-isonormal process $\mathcal{W}_{\cH}$.

\item Standard calculations show that the mapping
$\mathcal{W}_{\cH} \colon  \cH\otimes E \to L^2(\Omega;E)$,
$R  \mapsto R \cdot \mathcal{W}_{\cH}$,
is linear. We also call this mapping the \emph{stochastic integral}.

\item\label{rem:stochInt:f-r:Gauss} For any real-valued standard Gaussian random variable $\gamma$ and any $x\in E$, the $E$-valued random variable  $\gamma(\cdot) x$ is Gaussian. Also, the sum of finitely many independent $E$-valued Gaussian random variables is Gaussian. Consequently, $R \cdot \mathcal{W}_{\cH}$ is Gaussian for any $R\in \cH\otimes E$.

\item\label{rem:stochInt:f-r:Ito} Due to Proposition~\ref{prop:f-r},
\begin{equation}\label{eq:Ito:f-r}
2^{-\frac{1-(r\minsym p)}{(r\minsym p)}} \,
\gnnrm{R}{\gsum^p(\cH,E)}
\leq
\gnnrm{ R \cdot \mathcal{W}_{\cH}}{L^p(\Omega;E)}
\leq
\gnnrm{R}{\gsum^p(\cH,E)},
\quad R\in \cH\otimes E,
\end{equation}
for arbitrary $0<p<\infty$.
\end{enumerate}

\end{remark}

Since the finite rank operators form a dense subset of the space $\gradon{\cH}{E}$ of $\gamma$-radonifying operators and \eqref{eq:Ito:f-r} holds, we can extend the notion of the stochastic integral $\mathcal{W}_{\cH}\colon \cH\otimes E \to L^2(\Omega;E)$  uniquely to a bounded  linear operator $\mathcal{W}_{\cH}\colon \gradon{\cH}{E}\to L^2(\Omega;E)$.

\begin{definition}\label{def:stochInt}
Let $R\in\gradon{\cH}{E}$ and fix a sequence $(R_n)_{n\in\N}\subset \cH\otimes E$ of finite rank operators, such that
\begin{equation}\label{eq:gradon:approx}
\lim_{n\to\infty}\nnrm{R_n-R}{\gradon{\cH}{E}}=0.
\end{equation}
Then the $E$-valued random variable
\begin{equation}\label{eq:stochInt:gradon}
R\cdot \mathcal{W}_{\cH}:=\lim_{n\to\infty} R_n \cdot \mathcal{W}_{\cH}\qquad (\text{in }L^2(\Omega;E))
\end{equation}
is called \emph{stochastic integral of $R$ with respect to $\mathcal{W}_{\cH}$}.
\end{definition}

\begin{remark}\label{rem:stochInt:fnc}
\begin{enumerate}[leftmargin=*,align=right,label=\textup{(\roman*)}]
\item The stochastic integral of a $\gamma$-radonifying operator $R\in\gradon{\cH}{E}$ with respect to an $\cH$-isonormal process $\mathcal{W}_{\cH}$ is well-defined, i.e., $R \cdot \mathcal{W}_{\cH}$ in \eqref{eq:stochInt:gradon} exists as a limit in $L^2(\Omega;E)$ and does not depend on the approximating sequence $(R_n)_{n\in\N}\subset \cH\otimes E$ fulfilling \eqref{eq:gradon:approx}.
This is an immediate consequence of the completeness of $L^2(\Omega;E)$ and the estimates in~\eqref{eq:Ito:f-r}.

\item Standard calculations show that the mapping
$\mathcal{W}_{\cH} \colon  \gradon{\cH}{E} \to L^2(\Omega;E)$,
$R  \mapsto R \cdot \mathcal{W}_{\cH}$,
is linear. We also call this mapping the \emph{stochastic integral}.

\item $R \cdot \mathcal{W}_{\cH}$ is Gaussian for any $R\in\gradon{\cH}{E}$. This follows immediately from the definition due to Lemma~\ref{lem:gausslimit} and Remark~\ref{rem:stochInt:f-r}\ref{rem:stochInt:f-r:Gauss}.

\item\label{rem:stochInt:fnc:Ito} Due to Remark~\ref{rem:stochInt:f-r}\ref{rem:stochInt:f-r:Ito},
\begin{equation}\label{eq:Ito:radon}
2^{-\frac{1-(r\minsym p)}{(r\minsym p)}} \,
\gnnrm{R}{\gsum^p(\cH,E)}
\leq
\gnnrm{ R \cdot \mathcal{W}_{\cH}}{L^p(\Omega;E)}
\leq
\gnnrm{R}{\gsum^p(\cH,E)}
,\quad R\in \gradon{\cH}{E},
\end{equation}
for arbitrary $0< p<\infty$.
\end{enumerate}
\end{remark}

\subsection{Stochastic integrals as stochastic processes}\label{ssec:stoch_int_process}
%% \textbf{Setting:}
%Throughout this section, $T\in (0,\infty)$, $E$ is a separable $r$-Banach space ($r\in (0,1]$),
%$H$ is a separable Hilbert space, $(\Omega,\cF,\P)$
%is a probability space, and $W_{H}\colon L^2(0,T;H) \rightarrow L^2(\Omega)$ is an $L^2(0,T;H)$-isonormal
%process.
In this subsection  we assume  that $H$ is a separable Hilbert space  and investigate the stochastic integral introduced  in Definition~\ref{def:stochInt}
with $\cH = L^2(0,T;H)$.

In this particular case, for every $t\in[0,T]$, if $R\in\gamma(\cH,E)$, then the operator $R|_{[0,t]}:\cH\to E$, $f\mapsto R(\one_{[0,t]}f)$, is $\gamma$-radonifying as well, due to the right ideal property proven in Theorem~\ref{thm:ideal}.
Thus, we can define the stochastic process $R\cdot W_H(t):=R|_{[0,t]}\cdot W_H$, $t\in [0,T]$, which we also refer to as the \emph{stochastic integral} or the \emph{stochastic integral process of $R$ with respect to $W_H$}\textemdash recall that $W_H:=\mathcal{W}_{L^2(0,T;H)}$.
Our next goal is to prove that this process has a version with continuous paths and that it satisfies appropriate estimates.
%
%Due to the right ideal property, it is obvious that $R_n\to R$ in $\gamma(\cH,E)$ implies $R_n\cdot W_H(t)\to R\cdot W_H(t)$ in $L^2(\Omega;E)$ for all $t\in [0,T]$.

%Old version
%In this case we can define
%the stochastic process $(R\cdot W_{H} (t))_{t\in [0,T]}$ by setting
%$\left(\sum_{n=1}^{N} h_n \otimes x_n \right) \cdot W_{H}(t) = \sum_{n=1}^{N} x_n W_{H}( h_n 1_{[0,t]} )$
%($x_1,\ldots,x_n \in E$ and $(h_n)_{n=1}^{N}$ an orthonormal system in $L^2(0,T;H)$)
%and extending to general $R\cdot W_{H}$ by approximation. With slight abuse of
%notation we use $R\cdot W_{H}$ to denote the stochastic process $(R\cdot W_{H}(t))_{t\in [0,T]}$.
%
%We begin by proving that the process $R\cdot W_{H}$ has a version with continuous paths.
In the situation of normed spaces the continuity proof usually relies on Doob's maximal inequality.
By lack of convexity properties of $\nnrm{\cdot}{E}^p$, a different argument is needed for $r$-Banach spaces.

\begin{lemma}\label{lem:sup_cont_stoch_proc}
Let $(X_t)_{t\in [0,T]}$ be a continuous $E$-valued stochastic process such that $X_0 \equiv 0$
and for all $0\leq s\leq t\leq T$
%$s,t\in [0,T]$ satisfying $s\leq t$
it holds that $X_t-X_s$ is symmetric and independent
of $\sigma( \{X_u \colon u\in [0,s] \})$. Then for all $0<p<\infty$
%$p\in (0,\infty)$
it holds that
\begin{equation*}
\E \sgeklam{\sup_{t\in [0,T]} \nnrm{X_t}{E}^p}
\leq
2^{1+\frac{p}{r} - p }
\E \nnrm{X_T}{E}^p.
\end{equation*}

\end{lemma}

\begin{proof}
For $n\in \N$ set $D_n = \{ j 2^{-n} T \colon j\in \{0,1,\ldots,2^n \}\}$.
The result follows from the continuity of the process $(X_t)_{t\in [0,T]}$,
Fatou's lemma and
%L\'evy's inequality from
Lemma~\ref{lem:QBLevy:inequal}:
%\begin{equation*}
%\begin{aligned}
\begin{align*}
\E \sgeklam{\sup_{t\in [0,T]} \nnrm{ X_t }{E}^p}
&\leq \liminf_{n\rightarrow \infty} \E \sgeklam{\sup_{t\in D_n} \nnrm{ X_t }{E}^p}\\
&=  \liminf_{n \rightarrow \infty} \E \reklam{\sup_{k\in \{1,\ldots,2^n\}}
    \ssgnnrm{ \sum_{j=1}^{k} X_{j2^{-n} T} - X_{(j-1)2^{-n}T} }{E}^p}\\
&\leq 2^{1+\frac{p}{r} - p} \E \nnrm{ X_T}{E}^p.\qedhere
\end{align*}
%\end{aligned}
%\end{equation*}
\end{proof}

\begin{definition}\label{def:frsf}
An operator valued function $\Phi\colon [0,T]\to\calL(H,E)$ of the form
\[
\Phi(t)=\sum_{j=1}^J \one_{(t_{j-1},t_j]}(t)\sum_{k=1}^K  h_k \otimes x_{j,k},\qquad
t\in[0,T],
\]
with $0=t_0<t_1<\ldots<t_J=T$, $\{h_1,\ldots,h_K\}\subseteq H$ orthonormal, and $x_{j,k}\in E$, $1\leq j\leq J$, $1\leq k\leq K$, for some $J,K\in \bN$, is called \emph{finite rank step function}.
\end{definition}

Note that every finite rank step function can be canonically identified with
an element of $\gamma(L^2(0,T;H),E)$.
Indeed, if $\Phi\colon [0,T] \rightarrow \cL(H,E)$ is as in Definition~\ref{def:frsf}, then the operator $R_\Phi\colon L^2(0,T;H)\to E$ defined by
\begin{equation}\label{eq:phiRepR}
R_\Phi f
:=
\int_0^T\Phi(t)f(t)\,\dt,
% :=
% \sum_{j=1}^J \sum_{k=1}^K \int_{(t_{j-1},t_j]}\langle h_k,f(t)\rangle_H\,\dt\cdot x_{j,k},
\quad
f\in L^2(0,T;H),
\end{equation}
is $\gamma$-radonifying.
We say that $\Phi$ \emph{represents} $R_\Phi$ and sometimes write $\Phi$ instead of $R_\Phi$ (note that the mapping $\Phi\mapsto R_\Phi$ is one-to-one).
The following lemma shows that the class of operators represented by finite rank step functions is actually dense in $\gamma(L^2(0,T;H),E)$.

\begin{lemma}\label{lemma:approx_frsf}
Let  $R\in \gamma(L^2(0,T;H),E)$.
Then there exists a sequence of finite rank
step functions $(\Phi_n)_{n\in \N}$ %in $L^2(0,T;\cL(H,E))$
such that $\lim_{n\rightarrow \infty} \nnrm{R_{\Phi_n} - R }{\gamma(L^2(0,T;H),E)} =0$.
\end{lemma}

\begin{proof}
Let $(h_n)_{n\in \N}$ be an orthonormal basis for $H$, let $(f_n)_{n\in \N}$ be
the Haar basis for $L^2(0,T)$. For $n\in \N$ let $Q_n\in \mathcal{L}(H)$ be
the orthogonal projection on $\operatorname{span}(\{ h_1,\ldots,h_n\})$ and let $P_n \in \cL(L^2(0,T))$
be the orthogonal projection on the subspace $\operatorname{span}(\{ f_1,\ldots, f_n\})$. Then, for every $n\in\N$,
$R(P_n \otimes Q_n)$ is represented by a finite rank step function, and since $\{f_j\cdot h_k\colon (j,k)\in \N\times\N\}$ is an orthonormal basis of $L^2(0,T;H)$,  $R(P_n \otimes Q_n)\rightarrow R$ in $\gamma(L^2(0,T;H),E)$, see also Example~\ref{ex:g-approx}.
\end{proof}

\begin{proposition}\label{prop:contfunc}
Let  $R\in \gamma(L^2(0,T;H),E)$.
%Let $E$ be an $r$-Banach space and let $0<p< \infty$.
%For every  $R\in \gamma(L^2(0,T;H),E)$
Then the stochastic integral process $R\cdot W_H$  has a version with continuous paths and for every $0<p<\infty$,
\begin{equation}\label{eq:RWHcont}
\begin{aligned}
 %\const_{2,p,r}^{-1}
2^{-\frac{1-(r\minsym p)}{r\minsym p}}
\nnrm{R}{\gamma^p(L^2(0,T;H),E)}
&\leq \Big(\E\sgeklam{ \sup_{t\in [0,T]}\left\| R\cdot W_H(t) \right\|_E^p}\Big)^{\frac1p} \\
&\leq
2^{\frac{1}{p}+\frac{1}{r} - 1}
%\const_{p,2,r}
\nnrm{R}{\gamma^p(L^2(0,T;H),E)}.
\end{aligned}
\end{equation}
%where $\const_{2,p,r}$ and $\const_{p,2,r}$ are the constants in the Kahane-Khintchine inequality~\ref{thm:QBKahKhin}.
\end{proposition}

\begin{proof}
First assume $R$ is represented by a finite rank step function. In this case continuity of $R\cdot W_{H}$
follows from classical \Ito\ theory, observing that if $(h_k)_{k=1}^{n}$ is an orthonormal sequence in $H$,
then $((W_{H} ( 1_{[0,t]} \cdot h_k )_{t\in [0,T]})_{k=1}^{n}$ is a sequence of independent standard Brownian motions.
The first inequality in~\eqref{eq:RWHcont} follows by Remark~\ref{rem:stochInt:f-r}\ref{rem:stochInt:f-r:Ito},
and the second from Lemma~\ref{lem:sup_cont_stoch_proc} and  Remark~\ref{rem:stochInt:f-r}\ref{rem:stochInt:f-r:Ito}.\par
For general $R\in \gamma(L^2(0,T;H),E)$, the result now follows from Lemma~\ref{lemma:approx_frsf} and
the fact that~\eqref{eq:RWHcont} holds for operators represented by  finite rank step functions.
\end{proof}

As in the Banach space setting, we have the following series expansion of the stochastic integral.
For $R\in \gamma(L^2(0,T;H),E)$ and $h\in H$ we write $R(\cdot\otimes h)$ for the $\gamma$-radonifying operator $R(\cdot\otimes h)\colon L^2(0,T)\to E$, $f\mapsto R(f\otimes h)$. Moreover, we write $W_H(\cdot\otimes h)$ for the $L^2(0,T)$-isonormal process $W_H(\cdot\otimes h)\colon L^2([0,T])\to L^2(\Omega)$, $f\mapsto W_H(f\otimes h)$, and $W_Hh$ for the Brownian motion $W_Hh(t):=W_H(\one_{(0,t]}\otimes h)$, $t\in[0,T]$.

\begin{corollary}\label{cor:inf_int_rep}
Let $R\in \gamma(L^2(0,T;H),E)$. Then, for every orthonormal basis $(h_n)_{n\in\N}$ in $H$,
\begin{equation}\label{eq:inf_int_rep_full}
R\cdot W_H = \sum_{n\in\N} R(\cdot\otimes h_n)\cdot W_H(\cdot \otimes h_n),
\end{equation}
where the series converges in $L^p(\Omega;\cont([0,T];E))$, $0<p<\infty$. In particular, for every $x^*\in E^*$,
\begin{equation}\label{eq:inf_int_rep}
\begin{aligned}
 \langle R\cdot W_H , x^* \rangle
& =
  \sum_{n\in \N}
    \int_{0}^{\cdot}
      \langle R^* x^*(t), h_n \rangle_H
    \,\mathrm{d}W_Hh_n(t),
\end{aligned}
\end{equation}
where the series converges in $L^p(\Omega;\cont([0,T]))$, $0<p<\infty$.
\end{corollary}

\begin{proof}
%The assertion is clearly valid for every finite rank step function
%$R$ for which there exists an $N\in \N$ such that $R h_n = 0$ for all $n\in I \cap \{N, N+1,\ldots\}$.

%For $n\in \N\setminus I$ set $h_n = 0 \in H$, $W_n = 0 \in \cont([0,T])$,
%and for every
For $n\in \N$ let $Q_n\in \cL(H)$ be
the orthogonal projection on $\operatorname{span}(\{ h_1,\ldots,h_n\})$.
Moreover, let $(f_n)_{n\in \N}$ be
the Haar basis for $L^2(0,T)$, and for every $n\in \N$ write $P_n \in \cL(L^2(0,T))$
for the orthogonal projection on $\operatorname{span}(\{ f_1,\ldots, f_n\})$.
The representation~\eqref{eq:inf_int_rep_full} of $R\cdot W_H$ holds due to Example~\ref{ex:g-approx}, Proposition~\ref{prop:contfunc} and the fact that it obviously holds with $R$ replaced by $R(P_M\otimes Q_N)$ for arbitrary $M,N\in\N$.
To obtain~\eqref{eq:inf_int_rep}, we only have to check that for arbitrary $x^*\in E^*$ and $n\in \N$,
\[
\langle R(\cdot\otimes h_n)\cdot W_H(\cdot\otimes h_n),x^*\rangle
=
\int_0^\cdot \langle R^*x^*,h_n\rangle_H\,\mathrm{d}W_H h_n.
\]
But this is clear due to the construction of the stochastic integral.
\end{proof}
%
% If the dual $E^*$ does not separate the points of $E$, then a connection between $\gamma$-radonifying operators and operator-valued functions $\Phi\colon [0,T]\to \cL(H,E)$ as presented in Proposition~\ref{prop:Pettis_Ito} cannot be established.
% However, we could always say that a function is stochastically integrable if we can find a $\gamma$-radonifying operator (meaningfully) associated to it and call the stochastic integral of this operator the stochastic integral of the function.
% In this section we present ways to introduce stochastic integrals of functions with values in $r$-Banach functions spaces or Besov spaces, which do not necessarily have a separating dual.
% \todo{\emph{Petru:} Can we add a comment on how this relates to the usual notions in case we have a separating dual? Should we introduce the notion of `stochastically integrable function' also in these situations?}
%
In the Banach space setting the dual space can be used to relate operator-valued functions to operators $R\in \cL(L^2(0,T;H),E)$, see e.g.~\cite[Theorem 2.5]{NeervenWeis2005a}. Recall however that the dual of a quasi-Banach space may be trivial. In Subsection~\ref{ssec:stoch_int_sepdual} we provide an analogue of~\cite[Theorems 2.3 and 2.5]{NeervenWeis2005a} in the case that $E^*$ separates
points in $E$. The following two examples show that even if $E^*$ does not separate points, it can
still be possible to relate functions and operators.

\begin{example}\label{example:stoch_int_fs}
Let $(S,\cA,\mu)$ be a $\sigma$-finite measure space and let $E(S)$ be a separable $r$-Banach function space, $0<r\leq 1$, with finite cotype.
It follows from Corollary~\ref{cor:lattice}
that for every strongly measurable $\varphi \colon [0,T]\times S \rightarrow H$
such that
$
\displaystyle \ssgnnrm{\ssgrklam{\int_0^T \| \varphi(t,\cdot) \|_{H}^2 \,\mathrm{d}t }^{\frac{1}{2}} }{E}<\infty
$
it holds that $R_{\varphi} \in \gamma(L^2(0,T;H),E(S))$, where
\begin{equation*}
 (R_{\varphi} f)(s) = \int_{0}^{T} \langle \varphi(t,s), f(t) \rangle_H \,\mathrm{d}t,
\qquad
f\in L^2(0,T;H).
\end{equation*}
Consequently, $R_{\varphi}\cdot W_H$ is well-defined (see Definition~\ref{def:stochInt})
and moreover Corollary~\ref{cor:lattice} and Proposition~\ref{prop:contfunc} guarantee
that for all $p \in (0,\infty)$ there exist constants $c_{p,E}, C_{p,E} \in (0,\infty)$
(independent of the choice of $\varphi$) such that

\begin{align*}
\rnnrm{\ssgrklam{\int_0^T \| \varphi(t,\cdot) \|_{H}^2 \,\mathrm{d}t }^{\frac{1}{2}} }{E}
&\eqsim_{p,E}
\| R_{\varphi}\cdot W_H \|_{L^p(\Omega;\cont([0,T];E))}.
\end{align*}
%\begin{align*}
%c_{p,E}^{-1} \,
%\rnnrm{\ssgrklam{\int_0^T \| \varphi(t,\cdot) \|_{H}^2 \,\mathrm{d}t }^{\frac{1}{2}} }{E}
%&\leq
%\| R_{\varphi}\cdot W_H \|_{L^p(\Omega;\cont([0,T];E))}
%\\
%&\leq
%C_{p,E} \,
%\rnnrm{\ssgrklam{\int_0^T \| \varphi(t,\cdot) \|_{H}^2 \,\mathrm{d}t }^{\frac{1}{2}} }{E}.
%\end{align*}
%\begin{align*}
% c_{p,E}^{-1} \| R_{\varphi}\cdot W_H \|_{L^p(\Omega;C([0,T];E))}
% & \leq
% \left\| \left(\int_0^T \| \varphi(t,\cdot) \|_{H}^2 \,\mathrm{d}t \right)^{\frac{1}{2}} \right\|_E
%\\ &
% \leq
% C_{p,E} \| R_{\varphi}\cdot W_H \|_{L^p(\Omega;C([0,T];E))}.
%\end{align*}\par
More specifically, if $E(S)=L^q(S)$ for some $q\in (0,\infty)$,
then for all $p\in (0,\infty)$ it holds that

\begin{align*}
\int_{S} \ssgrklam{\int_0^T \nnrm{ \varphi(t,s) }{H}^2 \,\mathrm{d}t }^{\frac{q}{2}} \,\mathrm{d}s
& \eqsim_{p,q}
\gnnrm{ R_{\varphi}\cdot W_H }{L^p(\Omega;\cont([0,T];L^q(S)))}^q.
\end{align*}
%\begin{align*}
%c_{p,L^q}^{-q}
%\int_{S} \ssgrklam{\int_0^T \nnrm{ \varphi(t,s) }{H}^2 \,\mathrm{d}t }^{\frac{q}{2}} \,\mathrm{d}s
%& \leq
%\gnnrm{ R_{\varphi}\cdot W_H }{L^p(\Omega;\cont([0,T];L^q(S)))}^q \\
%& \leq
%C_{p,L^q}^{q}
%\int_{S} \ssgrklam{\int_0^T \nnrm{ \varphi(t,s) }{H}^2 \,\mathrm{d}t }^{\frac{q}{2}} \,\mathrm{d}s.
%\end{align*}
%\begin{align*}
% c_{p,L^q}^{-q} \| R_{\varphi}\cdot W_H \|_{L^p(\Omega;C([0,T];L^q(S)))}^q
% & \leq
% \int_{S} \left(\int_0^T \| \varphi(t,s) \|_{H}^2 \,\mathrm{d}t \right)^{\frac{q}{2}} \,\mathrm{d}s
%\\& \leq
% C_{p,L^q}^{q} \| R_{\varphi}\cdot W_H \|_{L^p(\Omega;C([0,T];L^q(S)))}^q.
%\end{align*}
\end{example}

\begin{example}\label{example:stoch_int_besov}
Let $d\in \N$, $0<p,q<\infty$, $\sigma\in \R$ and let $\varphi \in B^{\sigma}_{p\,q}(\R^d;L^2([0,T];H))$.
By Proposition~\ref{prop:Bpqgamma} it holds that $R_{\varphi}\in \gamma(L^2(0,T;H),B^{\sigma}_{p,q}(\R^d))$,
where $(R_{\varphi} f)(s) = \int_{0}^{T} \langle \varphi(s)(t), f(t) \rangle_H \,\mathrm{d}t $.
Consequently, $R_{\varphi}\cdot W_H$ is well-defined (see Definition~\ref{def:stochInt})
and moreover Proposition~\ref{prop:Bpqgamma} and Proposition~\ref{prop:contfunc} guarantee
that for all $\tau \in (0,\infty)$ there exist constants $c_{\tau,B^{\sigma}_{p,q}}, C_{\tau,B^{\sigma}_{p,q}} \in (0,\infty)$
(independent of the choice of $\varphi$) such that
\begin{align*}
\nnrm{ \varphi }{B^{\sigma}_{p,q}(\R^d;L^2(0,T;H))}
&\eqsim_{p,q, \sigma,\tau}
\nnrm{ R_{\varphi}\cdot W_H }{L^{\tau}(\Omega;\cont([0,T];B^{\sigma}_{p,q}(\R^d)))}.
\end{align*}
%\begin{align*}
%c^{-1}_{\tau,B^{\sigma}_{p,q}}
%\nnrm{ \varphi }{B_{p,q}^{\sigma}(\R^d;L^2(0,T;H))}
%&\leq
%\nnrm{ R_{\varphi}\cdot W_H }{L^{\tau}(\Omega;\cont([0,T];B^{\sigma}_{p,q}(\R^d)))}\\
%& \leq
%C_{\tau,B^{\sigma}_{p,q}}
%\nnrm{ \varphi }{B^{\sigma}_{p,q}(\R^d;L^2(0,T;H))}.
%\end{align*}\par
%\begin{align*}
% c^{-1}_{\tau,B^{\sigma}_{p,q}}
% \| R_{\varphi}\cdot W_H \|_{L^p(\Omega;C([0,T];B_{q,q}^{\sigma}(\R^d)))}
% & \leq
% \| \varphi \|_{B_{q,q}^{\sigma}(\R^d;L^2(0,T;H))}
%\\ &
% \leq
%C_{\tau,B^{\sigma}_{p,q}} \| R_{\varphi}\cdot W_H \|_{L^p(\Omega;C([0,T];B_{q,q}^{\sigma}(\R^d)))}.
%\end{align*}\par
\end{example}

\subsection{Stochastic integration in \texorpdfstring{$r$}{r}-Banach spaces with separating dual}\label{ssec:stoch_int_sepdual}
% \textbf{Setting:}
Recall that although the dual of a quasi-Banach space may be empty, there exist quasi-Banach spaces with separating dual.
An important example in the context of nonlinear approximation is the scale of Besov spaces $B^\alpha_{\tau,\tau}(\cO)$ satisfying $\frac{1}{\tau}=\frac{\alpha}{d}+\frac{1}{p}$ for a fixed $p\in (1,\infty)$.
As announced above Example~\ref{example:stoch_int_fs}, the aim of this section is to prove Proposition~\ref{prop:Pettis_Ito}, i.e., an analogue of~\cite[Theorem~2.5]{NeervenWeis2005a}. We begin with proving the following key result in a similar fashion as if $E$ were a Banach space, see, e.g.,~\cite[Theorem~2.3]{NeervenWeis2005a}.

%In this section, $E$ is a separable $r$-Banach space ($r\in (0,1]$) whose dual space $E^*$ separates the points in $E$,
%$\cH$ is a separable Hilbert space, $(\Omega,\cF,\P)$
%is a probability space, and $\mathcal{W}_{\cH}\colon \cH \rightarrow L^2(\Omega)$ is a $\cH$-isonormal
%process.

%The following lemma is easily verified for $R$ a finite rank operator, and extends to all $R\in \gradon{H}{E}$
%by approximation, using that $\mathcal{W}_{\cH}$ is continuous.
%
%\begin{lemma}\label{lemma:weakchar}
%Let $R\in \gradon{\cH}{E}$. Then for all $x^* \in E^*$ it holds that $\langle R\cdot \mathcal{W}_{\cH}, x^* \rangle = \mathcal{W}_{\cH}(R^*x^*)$.
%\end{lemma}

%The following result was proven for the Banach space setting -- in a similar fashion -- in~\cite[Theorem 2.3]{NeervenWeis2005a}.

\begin{theorem}\label{thm:weak_char_stoch_int}
Assume that $E^*$ separates the points of $E$ and let $R\in\linop{\cH}{E}$. Then the following are equivalent:
\begin{enumerate}[label=\textup{(\roman*)}]
 \item\label{item:Rgradon} $R\in\gradon{\cH}{E}$.
 \item\label{item:weakchar} There exists a random variable $X\colon \Omega \rightarrow E$ such that
\begin{equation}\label{eq:gradon}
\forall x^*\in E^*\colon \mathcal{W}_{\cH}(R^*x^*)=\langle X,x^*\rangle \quad \P\text{-a.s.}
\end{equation}
\end{enumerate}
In this case $X=R \cdot \mathcal{W}_{\cH}$ $\P$-a.s.
\end{theorem}

\begin{proof}
`\ref{item:Rgradon}$\Rightarrow$\ref{item:weakchar}': This implication holds even without assuming that $E^*$ separates the points of $E$:
The equality $\mathcal{W}_{\cH}(R^*x^*)=\langle R\cdot\mathcal W_\cH,x^*\rangle$ for all $x^*\in E^*$ is easily verified if $R$ is a finite rank operator and
can be extended to arbitrary $R\in\gamma(\cH,E)$ by using the continuity of $\mathcal{W}_\cH$ and the fact that, if  we choose an orthonormal basis $(h_n)_{n\in\bN}$ in $\cH$, then, due to Theorem~\ref{thm:summing-sep}\ref{thm:summing-sep:rad}, the series
\[
X:= \sum_{n=1}^\infty \mathcal{W}_{\cH}(h_n) Rh_n
=
R\cdot \mathcal{W}_\cH
\]
converges in $L^2(\Omega;E)$.

%So let $R\in\gradon{\cH}{E}$. Choose an orthonormal basis $(h_n)_{n\in\bN}$ in $\cH$.
%Then, due to Theorem~\ref{thm:summing-sep}\ref{thm:summing-sep:rad}, the series
%\[
%X:= \sum_{n=1}^\infty \mathcal{W}_{\cH}(h_n) Rh_n
%\]
%converges in $L^2(\Omega;E)$. Identity~\eqref{eq:gradon} holds trivially if $R$ is a finite rank operator,
%and may thus be extended to $R\in \gradon{\cH}{E}$ by approximation.
%
%Furthermore, for arbitrary $x^*\in E^*$,
%\begin{align*}
%\mathcal{W}(T^*x^*)
%=
%\mathcal{W}\ssgrklam{\sum_{n=1}^\infty [h_n,T^*x^*]h_n}
%=
%\sum_{n=1}^\infty [h_n,T^*x^*] \, \mathcal{W}(h_n)
%=
%\sum_{n=1}^\infty \langle Th_n, x^*\rangle \mathcal{W} (h_n),
%\end{align*}
%where the series on the right hand side converges in $L^2(\Omega):=L^2(\Omega;\bR)$.
%Thus,
%\begin{align*}
%\mathcal{W}(T^*x^*)
%=
%\sum_{n=1}^\infty \langle Th_n, x^*\rangle \mathcal{W} (h_n)
%=
%\langle \sum_{n=1}^\infty \mathcal{W} (h_n)Th_n, x^* \rangle
%=
%\langle X,x^*\rangle \quad \text{in}\quad L^2(\Omega;\bR).
%\end{align*}
%Consequently, $T$ is $\gamma$-radonifying.

`\ref{item:weakchar}$\Rightarrow$\ref{item:Rgradon}':
%We now proceed to prove~\eqref{item:weakchar} $\Rightarrow$ \eqref{item:Rgradon}.
Given a
random variable $X\colon \Omega \rightarrow E$ such that \eqref{eq:gradon} is fulfilled,
% Let us denote by $G$ the closure of the linear span of $\{\langle X,x^*\rangle: x^*\in E^*\}$ in $L^2(\Omega)$, i.e.,
we set
\[
G:=\overline{\operatorname{span}(\{\langle X,x^*\rangle: x^*\in E^*\})}^{\nnrm{\cdot}{L^2(\Omega)}}
=
\overline{\{\langle X,x^*\rangle: x^*\in E^*\}}^{\nnrm{\cdot}{L^2(\Omega)}}.
\]
By a Gram-Schmidt argument,
we can choose a sequence $(x^*_i)_{i\in I}\subset E^*$ such that
$\{\langle X,x_i^*\rangle: i\in I\}$ is an orthonormal basis of $(G,\lb\cdot,\cdot\rb_{L_2(\Omega)})$, where $I\subseteq \N$ (note that $G$ is separable since $X$ is strongly measurable). We assume that $I=\bN$, the other case can be treated analogously.
Since $\mathcal{W}_{\cH}$ is an $\cH$-isonormal process and \eqref{eq:gradon} holds, $(\langle X,x_i^*\rangle)_{i\in \bN}$ is a sequence of independent real-valued standard Gaussian random variables, i.e., a Gaussian sequence.
Put $h_i:= R^*x_i^*$, $i\in \bN$. Then $(h_i)_{i\in \bN}$ is an orthonormal basis of the closure in $\cH$ of the space spanned by $\{R^*x^*: x^*\in E^*\}$, i.e., of
\[
\cH_0
:=
\overline{\operatorname{span}(\{R^*x^*: x^*\in E^*\})}^{\nnrm{\cdot}{\cH}}
=
\overline{\{R^*x^*: x^*\in E^*\}}^{\nnrm{\cdot}{\cH}},
\]
endowed with the scalar product $\lb\cdot,\cdot\rb_\cH$ inherited from $\cH$.
This can be seen as follows. Since $\mathcal{W}_{\cH}$ is an $\cH$-isonormal process and \eqref{eq:gradon} holds,
\begin{align*}
\lb h_i,h_j\rb_\cH
=
\lb R^*x_i^*,R^*x_j^*\rb_\cH
&=
\E\geklam{\mathcal{W}_{\cH}(R^*x_i^*)\mathcal{W}_{\cH}(R^*x_j^*)}\\
&=
\E\geklam{\langle X,x_i^*\rangle \langle X,x_j^*\rangle}
=
\delta_{ij}
,\qquad i,j\in \N.
\end{align*}
Thus, $(h_i)_{i\in \bN}$ is an orthonormal system in $\cH_0$.
In order to prove the maximality of $(h_i)_{i\in \bN}$, fix an arbitrary $h\in \cH_0$ with $\lb h,R^*x_i^*\rb_\cH=0$ for all $i\in\bN$.
Then, since $\mathcal{W}_{\cH}$ is an $\cH$-isonormal process and due to \eqref{eq:gradon},
\begin{equation}\label{eq:gradon2}
\E\geklam{\mathcal{W}_{\cH}(h)\langle X,x_i^*\rangle}
=
\E\geklam{\mathcal{W}_{\cH}(h)\mathcal{W}_{\cH}(R^*x_i^*)}
=
\lb h,R^*x_i^*\rb_\cH
=
0
 \text{ for all } i\in\bN.
\end{equation}
Also, since $h\in \cH_0$, the very definitions of the spaces $\cH_0$ and $G$, together with the continuity of $\mathcal{W}_{\cH}\colon \cH\to L^2(\Omega)$, yield $\mathcal{W}_{\cH}(h)\in G$.
Thus, since $(\langle X,x_i^*\rangle)_{i\in\bN}$ is an orthonormal basis of $G$, Eq.~\eqref{eq:gradon2} yields $\mathcal{W}_{\cH}(h)=0$.
Using again the properties of $\mathcal{W}_{\cH}$, we obtain $\nnrm{h}{\cH}^2=\E\Abs{W_{\cH}(h)}^2=0$ and therefore $h=0$. Consequently, $(h_i)_{i\in\bN}$ is an orthonormal basis of $\cH_0$.

Fix $x^*\in E^*$. Recalling that $( \langle X, x_i^* \rangle )_{i\in \N}$ is an orthonormal basis for $G$
we have
\[
\langle X,x^*\rangle = \sum_{i\in \bN} c_i \langle X,x_i^*\rangle \quad \text{in }L^2(\Omega)
\]
with
\[
c_i:=\E\geklam{\langle X,x^*\rangle \langle X,x_i^*\rangle}
=
\lb R^*x^*,R^*x_i^*\rb_\cH
=
\langle Rh_i,x^*\rangle, \quad i\in \bN.
\]
Therefore,
\[
\langle X,x^*\rangle = \sum_{i\in \bN} \langle X,x_i^*\rangle \langle x^*,Rh_i\rangle \quad \text{ in }L^2(\Omega).
\]
% This series converges also $\P$-a.s.\ due to L\'{e}vy's theorem and
Therefore, by Proposition~\ref{prop:QB:conv:test:as}, the sequence
\[
S_n:=\sum_{i=1}^n \langle X,x_i^*\rangle Rh_i,\quad n\in\bN,
\]
of partial sums converges $\P$-a.s.\ to $X$. It follows from~\eqref{eq:gradon}, Theorem~\ref{thm:Fernique}, and Proposition~\ref{prop:weakGauss}
that $X \in L^2(\Omega;E)$, whence Proposition~\ref{prop:Levy:Lp} yields
\[
\sum_{i=1}^\infty \langle X,x_i^*\rangle Rh_i = X
\quad \text{in } L^2(\Omega;E).
\]
Thus, $R\in \gradon{\cH_0}{E}$ due to Theorem~\ref{thm:summing-sep}\ref{thm:summing-sep:rad}.
The assertion that $R\in \gradon{\cH}{E}$ follows now if we prove that $R$ vanishes on the orthogonal complement $\cH_0^\bot$ of $\cH_0$ in $\cH$. This is indeed the case since for a fixed $h\in \cH_0^\bot$,
$\lb Rh, x^*\rb=\lb h,R^*x^*\rb_\cH=0$ for all $x^*\in E^*$. Since $E^*$ separates the points of $E$, $Rh=0$ for all $h\in \cH_0^\bot$.

In order to prove that $X=R\cdot \mathcal W_{\cH}$ $\P$-a.s., first recall from the proof of Proposition~\ref{prop:QB:conv:test:as} that there exists a countable set in $E^*$ that separates the points of $E$. Hence every random variable fulfilling~\eqref{eq:gradon} is just a version of $R\cdot\mathcal{W}_\cH$. Next note that we have seen in the proof of `\ref{item:Rgradon}$\Rightarrow$\ref{item:weakchar}' that~\eqref{eq:gradon} holds for $X=R\cdot\mathcal{W}_\cH$.
%see~\ref{FIXME}\footnote{Sonja: A reference to the appropriate part of Section 2 should be added here.}.\par
\end{proof}

We have now all ingredients we need to prove a characterization of the stochastic integral introduced above as a Pettis type stochastic \Ito-integral, see~Proposition~\ref{prop:Pettis_Ito} below; for $r=1$ this result can be found, e.g., in~\cite[Proposition~3.2]{NeeVerWei2007}, see also~\cite{NeervenWeis2005a}.
Before we state this result, let us generalize some notions, which are part of the folklore in Banach space theory: We say a function $\Phi\colon [0,T]\to\cL(H,E)$ is \emph{$H$-measurable} if for every $h\in H$ the mapping $t\mapsto\Phi(t)h$ is strongly measurable.
Moreover, an $H$-measurable function $\Phi\colon [0,T]\to\cL(H,E)$ \emph{belongs to $L^2(0,T;H)$ scalarly} if for every $x^*\in E^*$ the mapping $\Phi^*x^*\colon [0,T]\to H$, $t\mapsto \Phi(t)^*x^*$, is square integrable. If $E^*$ separates the points of $E$, then we say that a function $\Phi\colon [0,T]\to\cL(H,E)$ that belongs to $L^2(0,T;H)$  \emph{(scalarly) represents} an operator $R\in \cL(L^2(0,T;H),E)$ if $R^*x^*=\Phi^*x^*$ in $L^2(0,T;H)$ for all $x^*\in E^*$, i.e., if for all $x^*\in E^*$,
\[
\langle Rf,x^*\rangle = \int_0^T \langle \Phi(t) f(t), x^*\rangle\,\dt,
\qquad
f\in L^2(0,T;H).
\]
Recall from~\eqref{eq:phiRepR} that every finite rank step function $\Phi\colon [0,T]\rightarrow \cL(H,E)$ represents an operator $R\in\cL(L^2(0,T;H),E)$.

\begin{proposition}\label{prop:Pettis_Ito}
Assume that $E^*$ separates the points of $E$. For a function $\Phi\colon [0,T]\to \cL(H,E)$ belonging to $L^2(0,T;H)$ scalarly, the following are equivalent:
\begin{enumerate}[label=\textup{(\roman*)}]
\item\label{item:Phi_approx_weak} There exists a sequence $(\Phi_n)_{n\in\bN}$ of finite rank step functions,
%$\Phi_n\colon [0,T]\rightarrow \calL(H,E)$,
such that
\smallskip
\begin{itemize}
\item $\displaystyle\lim_{n\to\infty} \Phi_n^*x^*=\Phi^* x^*$ in $L^2(0,T;H)$ for all $x^*\in E^*$;
\item $(R_{\Phi_n} \cdot W_H)_{n\in\bN}$ is a Cauchy sequence in $L^0(\Omega;E)$ (i.e., Cauchy in probability).
\end{itemize}
\item\label{item:Phi_weakchar} There exists a random variable $X\colon\Omega \rightarrow E$ such that for all $x^*\in E^*$ we have
\[
\langle X,x^*\rangle = \int_0^T \Phi^*(t)x^*\,\mathrm{d}W_H(t)\qquad\P\textup{-a.s.}
\]
%\[\forall x^*\in E^*\colon \lb X, x^*\rb = W_H(R^*x^*) \quad \P\text{-a.s.} \]
\item\label{item:Phi_gradon} $\Phi$ represents an operator $R\in\gamma(L^2(0,T;H),E)$.
\end{enumerate}
In this case, $R\cdot W_H = X $ in $L^p(\Omega;E)$
for all $p\in (0,\infty)$, and
\begin{equation}\label{eq:Ito}
\nnrm{R}{\gamma^p(L^2(0,T;H);E)}
\leq
\nnrm{X}{L^p(\Omega;E)}
\leq 2^{\frac{1-r\land p}{r\land p}}
\nnrm{R}{\gamma^p(L^2(0,T;H);E)}.
\end{equation}
We call a function $\Phi$ satisfying these equivalent conditions \emph{stochastically integrable with respect to $W_H$}. The random variable $X$ is called the \emph{stochastic integral of $\Phi$ with respect to $W_H$} and is denoted by
\[
\int_0^T\Phi(t)\,\mathrm{d}W_H(t):=X.
\]
\end{proposition}

\begin{proof} `\ref{item:Phi_approx_weak}$\Rightarrow$\ref{item:Phi_gradon}':
We claim that $(R_{\Phi_n}\cdot W_H)$ converges in $L^q(\Omega;E)$ for any $q\in (0,\infty)$. In order to show this let $\varepsilon>0$ be arbitrary.
Note that $R_{\Phi_n}\cdot W_H - R_{\Phi_m}\cdot W_H = R_{\Phi_n - \Phi_m}\cdot W_H$ is a Gaussian sum, so that Theorem \ref{thm:QBKahKhin} applies and we will use the constant $C_{p,q,r}$ introduced there.
Let $\delta = \varepsilon 2^{-1/q}$. Choose $N$ such that for all $m,n\geq N$,
\[\P(\|R_{\Phi_n}\cdot W_H-R_{\Phi_m}\cdot W_H\|>\delta)\leq \frac{1}{4 C_{2q,q,r}^2}.\]
Then by \cite[Corollary 6.2.9]{HytNeerVerWeis2017} (which remains valid in the $r$-Banach space setting),
we obtain that for all $n,m\geq N$,
\[\|R_{\Phi_n}\cdot W_H - R_{\Phi_m}\cdot W_H\|_{L^q(\Omega;E)}\leq 2^{1/q} \delta = \varepsilon.\]
This proves the claim. In particular,  $(R_{\Phi_n}\cdot W_H)$ converges in $L^2(\Omega;E)$. This, in turn, implies the existence of an operator $R\in \gamma(L^2(0,T;H),E)$ such that $\lim_{n\to\infty}\nnrm{R_{\Phi_n}-R}{\gamma(L^2(0,T;H),E)}= 0$ due to Inequality~\eqref{eq:Ito:f-r} and the completeness of $\gamma(L^2(0,T;H),E)$. Obviously, for all $x^*\in E^*$, we have $R_{\Phi_n}^* x^*=\Phi_n^*x^*$, and, as a consequence, $R^*x^*=\Phi^*x^*$.

`\ref{item:Phi_gradon}$\Rightarrow$\ref{item:Phi_approx_weak}':
Since $R\in \gamma(L^2(0,T;H),E)$, Lemma~\ref{lemma:approx_frsf} implies that there exists a sequence of finite rank step functions $(\Phi_n)_{n\in\N}$ such that $\lim_{n\to\infty}\nnrm{R_{\Phi_n}-R}{\gamma(L^2(0,T;H),E)}= 0$. Thus, due to Inequality~\eqref{eq:Ito:f-r}, $(R_{\Phi_n}\cdot W_H)_{n\in\N}$ is a Cauchy sequence in $L^2(\Omega;E)$, and hence in $L^0(\Omega;E)$. Moreover, for every $x^*\in E^*$, we have $\nnrm{R_{\Phi_n}^*x^*-R^*x^*}{L^2(0,T;H)}\to 0$ for $n\to \infty$. Thus, $\Phi_n^*x^*\to \Phi^*x^*$ in $L^2(0,T;H)$.

The equivalence of \ref{item:Phi_weakchar} and \ref{item:Phi_gradon} as well as the final statement are immediate consequences of Theorem~\ref{thm:weak_char_stoch_int}.
\end{proof}

\section{Stochastic integration II: random integrands}\label{sec:stoch_int_processes}%Stochastic integration of processes

In this section we extend the stochastic integral introduced above by allowing for random integrands.
More specifically, we define the stochastic integral for $R\in L^{0}(\Omega; \gamma(L^2(0,T;H),E))$,
where $T>0$, $H$ is a Hilbert space, $E$ is a suitable type of $r$-Banach space, and $R$ satisfies an adaptedness condition.
The approach is in the spirit of~\cite{NeeVerWei2007} (see also the survey~\cite{NeeVerWei2015}):
We first define the stochastic integral for stochastic processes with a very simple structure, the so-called adapted elementary processes.
Then we use a decoupling inequality to obtain a Burkholder-Davis-Gundy type estimate.
Finally, using this estimate, we extend the stochastic integral to a wider class of random integrands: first in $L^p$ and then in $L^0$ by localization.

%This is the cornerstone which we then use to extend the stochastic integral to a wider class of random integrands.

The main difference to the Banach space theory is that we do not assume that $E$ has the UMD property.
Instead, we consider a (weaker) one-sided decoupling property.
This property was investigated in the quasi-Banach space setting in~\cite{CoxVer2011}.
It turns out to be sufficient for the one-sided estimate required for the extension of the stochastic integral.
In fact, as Proposition~\ref{prop:quasi_not_UMD} below shows, the UMD property fails to make sense in the quasi-Banach space setting.

%More specifically, in Subsection~\ref{ssec:stoch_int_frsp} we introduce
%adapted elementary processes and define their stochastic integral.
%\todo{\emph{Petru:} I think we can delete this paragraph.}
%In order to obtain a suitable Burkholder-Davis-Gundy type estimate for the
%stochastic integral we need to assume the space $E$ satisfies a decoupling property.
%This is explained in Subsection~\ref{ssec:BDG}. The Burkholder-Davis-Gundy type
%estimate allows us to define the stochastic integral for adapted
%elements of $L^0(\Omega,\gamma(L^2(0,T;H),E))$, see Theorems~\ref{thm:stochInt_BDG} and ~\ref{thm:stochInt_local}.\par

%\subsection{Setting in Sections~\ref{ssec:stoch_int_frsp}--\ref{ssec:stoch_int_processes}}\label{ssec:setting}
%In Sections \ref{ssec:stoch_int_frsp}--\ref{ssec:stoch_int_processes} let
\smallskip
\noindent\textbf{Notation.} Throughout this section, $E$ is a separable $r$-Banach space for some $0<r\leq 1$,
$H$ is a separable Hilbert space, and $0<T<\infty$.
In addition, $(\Omega,\cF,\P)$ is a probability space endowed with a normal  filtration $\F=(\cF_t)_{t\in [0,T]}$,
and $W_H$ is an $\mathbb F/ L^2(0,T;H)$-isonormal process, i.e., an $L^2(0,T;H)$-isonormal process such that
for all $0\leq s < t \leq T$ and all $h\in L^2(0,T;H)$ such that $\supp h \subseteq [s,t]$
it holds that $W_{H}(h)$ is $\cF_t$-measurable and independent of $\cF_s$.

%\par
% For $R\in\gamma(L^2(0,T;H),E)$ and $t\in [0,T]$ we define
% \[
% R(t): L^2(0,t;H)\to E,\qquad R(t)f:=R(\one_{(0,t)}f),
% \]

% \subsection{Decoupling \& the definition of the stochastic integral}

\subsection{Intermezzo: The UMD property in quasi-Banach spaces}\label{ssec:UMD}
In order to extend the class of admissible integrands and to  develop a stochastic calculus in the Banach space setting
one typically assumes the Banach space has the UMD (\emph{unconditional martingale differences}) property, see, e.g., the survey~\cite{NeeVerWei2015}.
Although the concept of an $E$-valued martingale, with $E$ an $r$-Banach space and $r\in (0,1)$,
does not make sense in general (for lack of a Bochner integral), it does makes sense to consider
$E$-valued martingales with respect to a filtration of finitely generated $\sigma$-algebras.
Moreover, it is known from~\cite{Maurey:1975}, see also e.g.~\cite[Theorem 4.2.5]{HytNeerVerWeis2016}, that
in the Banach space setting the UMD property is equivalent to the dyadic UMD property (i.e., where one
considers only martingales with respect to a dyadic filtration). It is also known, see~\cite[Section 6]{CoxGeiss}, that the
decoupling inequalities needed to develop the stochastic integral of random integrands can be derived directly from the dyadic UMD
property. Unfortunately, Proposition~\ref{prop:quasi_not_UMD} below states that if an $r$-Banach space satisfies the dyadic UMD property,
then it is isomorphic to a Banach space.

Moreover, one can combine Proposition~\ref{prop:quasi_not_UMD} with \cite[Theorem~2 ]{Garl86} (the proof carries over to the r-Banach space setting mutatis mutandis) to prove that if $E$ is an r-Banach space such that a two-sided version of inequality~\eqref{eq:gaussian_decoupling} below holds (or, equivalently, a two-sided version of inequality~\eqref{eq:BDG_frsp}), then $E$ is isomorphic to a Banach space.

A sequence of $E$-valued random variables
$(d_n)_{n\in\N}$ is called a \emph{Paley-Walsh martingale} (or \emph{dyadic martingale}) if there exists a Rademacher
sequence $(r_k)_{k\in\N}$, a $v_0\in E$, and mappings $v_n\colon \{-1,1\}^{n} \rightarrow E$, $n\in \N$, such that
$d_1=r_1 v_0$ and $d_n = r_1 v_0 + \sum_{k=2}^{n} r_k v_{k-1}(r_1,\ldots,r_{k-1} )$, $n\in \{2,3,\ldots,\}$.
If $E$ is a Banach space, this is equivalent to saying that $(d_n)_{n\in\bN}$
is a martingale with respect to a dyadic filtration. 
% For this let
% $(r_k)_{k\in \N}$ be a Rademacher sequence. We call $(d_n)_{n\in \N}$ a \emph{Paley-Walsh martingale difference sequence} if $d_n=r_k\nu_n(r_1, \ldots, r_{n-1})$, where and $\nu_n$ is $\sigma(r_1,\ldots,r_{n-1})$-measurable. Now one can define the UMD$_p$ property for this special class of martingales in the usual way (see~\cite[Definition~5.2]{NeeVerWei2015}).

\smallskip

\begin{proposition}\label{prop:quasi_not_UMD}
Let $0< r <1$, $1<p <\infty$, and let $E$ be a separable $r$-Banach space
that satisfies the dyadic UMD$_{p}$-property, i.e., there exists a constant $\beta \in (0,\infty)$
such that
\[
\E\ssgnnrm{\sum_{n=1}^N\varepsilon_n d_n}{E}^p
\leq
\beta^p
\E\ssgnnrm{\sum_{n=1}^N d_n}{E}^p,
\quad\text{for all }N\in\N,
\]
%~\cite[Equation 5.2]{NeeVerWei2015}
holds for all Paley-Walsh martingales $(d_n)_{n\in \N}$ and all signs $(\eps_n)_{n\in \N}\in \{-1,1\}^{\N}$.
Then $E$ is isomorphic to a Banach space.
\end{proposition}

\begin{proof}
Let $\beta_{p}(\ell_{2^n}^{\infty})$, $n\in \N$, denote the dyadic UMD$_p$ constant
of $\ell_{2^n}^{\infty}$.
It is well-known, see e.g.~\cite[Proposition 4.2.19]{HytNeerVerWeis2016}, that
there exists a constant $c_p\in (0,\infty)$ such that
$\beta_{p}(\ell_{2^n}^{\infty})\geq c_p \sqrt{n}$ for all $n\in \N$. Note that
if, for some $\lambda \in (0,\infty)$ and $n\in \N$, $E$ contains a $\lambda$-isomorphic copy of $\ell_{2^n}^{\infty}$
in the sense of~\cite[Definition~7.1.8]{HytNeerVerWeis2017}
then $\beta_{p}(\ell_{2^n}^{\infty}) \leq \lambda \beta$. It follows that $c_0$ is not
$\eps$-finitely represented in $E$ for any $\eps \in (0,\infty)$, in the sense of~\cite[Definition~11.1.1]{AlbKal2006}.
Thus~\cite[Theorem~1]{BasUri1986} implies that there exists a $q\in [2,\infty)$
such that $E$ has cotype $q$.

Minor modifications of the Burkholder stopping time proof (see~\cite[Theorem~1.1, `(1.2)$\Rightarrow$ (1.3)']{Burkholder:1981},
and see~\cite[Lemma 3.2]{CoxVer2011} for the minor modifications) shows that
if $E$ satisfies the dyadic $\text{UMD}_{s}$-property for some $s\in (1,\infty)$, then it satisfies the dyadic $\text{UMD}_{s}$-property
for all $s\in (1,\infty)$.
It follows immediately (see e.g.~\cite[Proposition~5.3]{NeeVerWei2015})
that $E$ has Paley-Walsh martingale cotype $q$.
Now the proof of~\cite[Theorem~3.1(a)]{Pisier:1975}
(which also works for $E$ a quasi-Banach space)
provides a norm $\abs{\cdot}_E\colon E \rightarrow \R$ and constants $c,C\in (0,\infty)$
such that $c\nnrm{ x }{E} \leq \abs{ x }_{E} \leq C \nnrm{ x}{E}$ for all $x\in E$.
\end{proof}

\subsection{Stochastic integrals of adapted elementary processes and a one-sided Burkholder-Davis-Gundy inequality}\label{ssec:stoch_int_frsp}
As usual, our construction starts with the following class of integrands.

\begin{definition}\label{def:frsp}
We say that
$\Phi \colon [0,T]\times \Omega \to \calL(H,E)$ is an \emph{adapted elementary process} if there exist
$0= t_0 < t_1 < \ldots < t_J = T$,
$\{h_1,\ldots,h_K\}\subseteq H$ orthonormal
%$J,K\in \N$, an orthonormal sequence $(h_k)_{k=1}^{K}$ in $H$,
and $\cF_{t_{j-1}}$-measurable simple random variables $X_{j,k}\colon \Omega \rightarrow E$,
$j\in \{1,\ldots,J\}$, $k \in \{1,\ldots, K\}$, for some $J,K\in\N$, such that
\begin{equation*}
 \Phi(t,\omega) = \sum_{j=1}^{J}
    1_{(t_{j-1},t_j]}(t)
    \sum_{k=1}^{K}
     h_k \otimes X_{j,k}(\omega).
\end{equation*}
Moreover, we say a random variable $R\in L^0(\Omega;\gamma(L^2(0,T;H),E))$
%a random variable $R\colon \Omega\to \cL(L^2(0,T;H),E)$
is \emph{represented} by an adapted elementary process $\Phi$ if
\begin{equation}\label{eq:frsp_gradon}
 Rf=R_{\Phi}f := \int_{0}^{T} \Phi(t,\cdot) f(t)  \,\mathrm{d}t, \qquad f\in L^2(0,T;H).
\end{equation}
\end{definition}

%Note that an adapted elementary process $\Phi \colon [0,T]\times \Omega \rightarrow \calL(H,E)$
%as in Definition~\ref{def:frsp} defines a random variable
%$R_{\Phi} \in L^{\infty}(\Omega;\gamma(L^2(0,T;H),E))$
%given by
%\begin{equation}\label{eq:frsp_gradon}
% R_{\Phi}(\omega)f = \int_{0}^{T} \Phi(t,\omega) f(t)  \,\textrm{d}t, \quad f\in L^2(0,T;H).
%\end{equation}
%We say that $\Phi$ \emph{represents} $R_\Phi$ and sometimes write $\Phi$ instead of $R_\Phi$.

For this class of processes/operator valued random variables we define the stochastic integral the usual way as a path-wise Lebesgue-Stieltjes type integral.

\begin{definition}\label{def:stoch_int_frsp}
Let $R\in L^0(\Omega;\gamma(L^2(0,T;H),E))$ be represented by an adapted elementary process
$\Phi= \sum_{j=1}^{J} \sum_{k=1}^{K}  1_{(t_{j-1},t_j]} \otimes h_{k}\otimes X_{j,k}$ with $t_j$, $h_k$ and $X_{j,k}$ as in Definition~\ref{def:frsp}.
%
%Let $\Phi\colon [0,T]\times \Omega \rightarrow \cL(H,E)$ be an adapted elementary process;
%$\Phi= \sum_{j=1}^{J} \sum_{k=1}^{K}  1_{(t_{j-1},t_j]} \otimes h_{k}\otimes X_{j-1,k} $.
Then the \emph{stochastic integral of $R$ with respect to $W_{H}$} is denoted by $R \cdot W_{H}$
and defined by
\begin{equation*}
 R \cdot W_{H}
 :=
 \sum_{j=1}^{J} \sum_{k=1}^{K} W_{H}(\one_{(t_{j-1},t_j]} \otimes h_{k}) X_{j,k}.
\end{equation*}
Moreover, the \emph{stochastic integral process $(R \cdot W_{H}(t))_{t\in [0,T]}$}
is defined by
\begin{equation*}
R \cdot W_{H}(t)
 :=
 \sum_{j=1}^{J} \sum_{k=1}^{K}  W_{H}(\one_{(t_{j-1},t_j]\cap [0,t]} \otimes h_{k})X_{j,k},\qquad t\in [0,T].
\end{equation*}
\end{definition}

\begin{remark}\label{remark:stoch_int_frsp_properties}
Let $\Phi$ be an adapted elementary process.
\begin{enumerate}[leftmargin=*,align=right,label=\textup{(\roman*)}]
 \item With slight abuse of notation we use $R_\Phi\cdot W_{H}$ to denote
both the stochastic integral and the stochastic integral process. We may also write $\Phi\cdot W_H$ instead of $R_\Phi\cdot W_H$.
 \item The stochastic integral process $\Phi \cdot W_{H}$
is an $E$-valued $\F$-adapted process, and
there exists a version of $\Phi \cdot W_{H}$ that has continuous paths.
 \item Recall that the Bochner integral (and hence the conditional expectation)
 fails to extend to the quasi-Banach space setting.
 However, clearly $\langle \Phi \cdot W_H, x^* \rangle$ is a martingale for all
 $x^*\in E^*$.
\end{enumerate}
\end{remark}

In the Banach space setting the UMD property is not strictly necessary to extend the definition of the stochastic integral to a suitable class of operator valued random variables $R\in L^0(\Omega; \gamma(L^2(0,T;H),E))$.
Indeed, as shown in~\cite[Theorem~5.4(ii)]{CoxVer2011}, to obtain an extension together with a $p$-independent one-sided (`upper') Burkholder-Davis-Gundy inequality, it is enough to assume that $E$ satisfies  a decoupling property for tangent sequences
as introduced by Kwapie\'n and Woyczynski~\cite{KwapienWoyczynski1989},
see also~\cite{KwapienWoyczynski1992} for details on this topic. In this section we show that in the quasi-Banach space setting we can extend the class of integrands with a similar strategy.

% The decoupling property for tangent sequences was introduced by Kwapie\'n and Woyczynski~\cite{KwapienWoyczynski1989},
%we refer to~\cite[Definition~1.2]{CoxVer2011} for a definition in the $r$-Banach space setting,
%and to
% see also~\cite{KwapienWoyczynski1992} for details on this topic.
%In order to extend the definition of the stochastic integral to
%a suitable class of operator valued random variables
%% adapted -- in the sense of see Definition~\ref{def:adapted_process} below --
%$R\in L^0(\Omega; \gamma(L^2(0,T;H),E))$, we do not need the UMD property.
%%need to impose additional conditions on the $r$-Banach space $E$.
%
%As we show below, it is sufficient to assume that $E$ satisfies the so-called \emph{upper decoupling property for tangent sequences}.
%
%%More specifically, we need that $E$ satisfies the so-called \emph{upper decoupling property for tangent sequences}.
%This property was introduced by Kwapie\'n and Woyczynski~\cite{KwapienWoyczynski1989},
%we refer to~\cite[Definition~1.2]{CoxVer2011} for a definition in the $r$-Banach space setting,
%and to~\cite{KwapienWoyczynski1992} for details on this topic.
%We shall work with the definition of decoupling as provided in~\cite[Definition~1.2]{CoxVer2011} (where the quasi-Banach space setting is considered). More precisely, we emphasize that in this manuscript, whenever we say that a space $E$ \emph{satisfies the decoupling property} we mean that \emph{the decoupling inequality holds in $E$ for all $p\in(0,\infty)$} in the sense of \cite[Definition~1.2]{CoxVer2011}.

The precise definition of the decoupling property is technical and understanding the definition is not essential for the comprehension of this section; one may take Proposition~\ref{prop:DPgauss} as a starting point. However, for the readers' 
convenience we recall here the definition of decoupling that we have in mind (see also~\cite[Definition~1.2 and Theorem~4.1]{CoxVer2011}).

\begin{definition}
Let $E$ be a quasi-Banach space, $(\Omega,\cF,\P,(\mathcal G_n)_{n \in \N_0})$ a filtered probability space, and let $(d_n)_{n \in \N}, (e_n)_{n\in \N}$ be $L^p(\Omega;E)$-sequences such that $(d_n)_{n\in \N}$ is $(\mathcal G_n)_{n\in \N}$-adapted. We say that $(e_n)_{n\in \N}$ is a \emph{$\mathcal G_\infty:=\sigma(\mathcal G_n\colon n \in \N)$-decoupled tangent sequence of} $(d_n)_{n\in \N}$ if $(e_n)_{n \in \N}$ is a $\mathcal G_\infty$-conditionally independent sequence and for all $B\in\cB(E)$, $n\in \N$ it holds that
\begin{align*}
\P(d_n\in B |\mathcal G_{n-1})
=
\P(e_n\in B|\mathcal G_{\infty}).
\end{align*} 
\end{definition}

By passing to a larger probability space, one can always construct a $\mathcal{G}_{\infty}$-decoupled tangent sequence to a given $(\mathcal G_n)_{n \in \N}$-adapted $L^p(\Omega;E)$-sequence $(d_n)_{n \in \N}$ (see, e.g.~\cite[Section 4.3]{KwapienWoyczynski1992}).

\begin{definition}
We say that a quasi-Banach space \emph{$E$ satisfies the decoupling inequality} if for all $p\in(0,\infty)$ there exists a constant $D_p\in (0,\infty)$ such that for all filtered probability spaces $(\Omega,\cF,\P,(\mathcal G_n)_{n \in \N_0})$ and all $L^p(\Omega;X)$-sequences $(d_n)_{n \in \N}$ and $(e_n)_{n\in \N}$ such that $(d_n)_{n\in \N}$ is $(\mathcal G_n)_{n \in \N}$-adapted and $(e_n)_{n\in \N}$ is a $\mathcal G_\infty$-decoupled tangent sequence of $(d_n)_{n\in \N}$ it holds that
\begin{equation*}
\E\ssgnnrm{\sum_{n=1}^N d_n}{E}^p
\leq D_p^p\,
\E\ssgnnrm{\sum_{n=1}^N e_n}{E}^p\qquad \text{ for all }N \in \N.
\end{equation*}
\end{definition}

In contrast to the UMD property, the decoupling property is satisfied by many prominent examples of quasi-Banach spaces.
The following examples are particularly important for our purposes.

\begin{example}\label{ex:decoupling_spaces}
Every Hilbert space satisfies the decoupling
property, see e.g.~\cite[Corollary 4.9]{CoxVer2011}. Moreover, let $F$ be an $r$-Banach space
satisfying the decoupling property, $0<r\leq 1$, let $(S,\mathcal{A},\mu)$ be a $\sigma$-finite measure space,
and let $0<p<\infty$, then the $r\minsym p$-Banach space $L^p(S;F)$
satisfies the decoupling property, see~\cite[Corollary 4.6]{CoxVer2011}. Finally,
it is immediate from the definition that any closed subspace of an $r$-Banach space satisfying the decoupling property
again satisfies the decoupling property. In particular, $B_{p,q}^{\alpha}(\R^d)$
satisfies the decoupling property for all $\alpha\in \R$ and all $p,q\in (0,\infty)$.
\end{example}

Proposition~\ref{prop:DPgauss} below provides the crucial property ensuring
that quasi-Banach spaces that satisfy the decoupling property allow for a well-defined stochastic integral, see~\cite[Example 1.1, Definition~1.2, and Theorem~4.1]{CoxVer2011} for a proof.

\begin{proposition}\label{prop:DPgauss}
Let $E$ satisfy the decoupling property.
%upper decoupling property for tangent sequences.
Then for all $p \in (0,\infty)$ there exists a constant $C_{E,p}\in (0,\infty)$
such that for every filtered probability space $(\Omega,\cF,\P,\G=(\mathcal{G}_n)_{n\in \N_0})$,
all Gaussian sequences  $(\gamma_n)_{n\in \N}$
on $(\Omega,\cF,\P)$ such that $\gamma_n$ is $\mathcal{G}_n$-measurable and independent of $\mathcal{G}_{n-1}$,
and all $\G$-adapted $E$-valued stochastic processes $(v_n)_{n\in \N_0}$
it holds that
\begin{equation}\label{eq:gaussian_decoupling}
\E \ssgeklam{
\sup_{1\leq M \leq N}
\ssgnnrm{
  \sum_{n=1}^{M} \gamma_n v_{n-1}
}{E}^p
 }
 \leq
 C_{E,p}^p
 \E \ssgnnrm{
  \sum_{n=1}^{N} \gamma_n' v_{n-1}
 }{E}^p,
\end{equation}
where $(\gamma_n')_{n\in \N}$ is a copy of $(\gamma_n)_{n\in \N}$ independent
of $\mathcal{G}_{\infty}:=\sigma(\bigcup_{n\in \N}\mathcal{G}_n)$. Moreover, without loss of generality
one may assume $\limsup_{p\rightarrow \infty} C_{E,p}/p < \infty$.
\end{proposition}

The proposition above allows us to obtain a one-sided Burkholder-Davis-Gundy inequality
for the stochastic integral of an adapted elementary process.

\begin{proposition}\label{prop:BDG_frsp}
Let $E$ satisfy the
%one-sided
decoupling property.
%for tangent sequences.
Then for all
$p\in (0,\infty)$ and all adapted elementary processes $\Phi \colon [0,T]\times \Omega \rightarrow \calL(H,E)$
it holds that
\begin{equation}\label{eq:BDG_frsp}
 \E \sgeklam{\sup_{t\in [0,T]}
 \nnrm{\Phi \cdot W_{H}(t)}{E}^p}
 \leq
 C_{E,p}^p
 \E \nnrm{ R_{\Phi}}{\gamma^{p}(L^2(0,T;H),E)}^p,
\end{equation}
where $C_{E,p}$ is the constant introduced in Proposition~\ref{prop:DPgauss}
%, $K_{p,2}$ is the constant in the $L^p$-to-$L^2$ Kahane-Khintchine inequality -- see Theorem~\ref{thm:QBKahKhin} --
and $R_{\Phi}$ is defined by~\eqref{eq:frsp_gradon}.
\end{proposition}

\begin{proof}
The proof is entirely analogous to the proof of~\cite[Theorem 5.4 (2)]{CoxVer2011}. More precisely,
we follow the argument on~\cite[pp.\ 371-372]{CoxVer2011}, but replace the
equivalence labeled `$\overset{\textup{(i)}}{\eqsim}$' on~\cite[p.\ 371]{CoxVer2011} by an upper estimate
which follows from Proposition~\ref{prop:DPgauss}.
%and applying Theorem~\ref{thm:QBKahKhin}
%at~\cite[p.\ 371 (ii)]{CoxVer2011}.
\end{proof}

\begin{remark}
For the purpose of obtaining the one-sided Burkholder-Davis-Gundy inequality~\eqref{eq:BDG_frsp}
%defining an $E$-valued stochastic integral with a stochastic integrand
one merely needs the Gaussian decoupling property, i.e., one needs that \eqref{eq:gaussian_decoupling} holds.
The reason we do not directly consider quasi-Banach spaces with this property is that it does not seem possible to prove $p$-independence of this property directly. Note however that results in \cite[Section 6]{CoxGeiss} and \cite{Veraar:2007} prove that---at least in the Banach space setting---Gaussian decoupling is equivalent to Rademacher decoupling, for which the $p$-independence is well-established.
We take~\cite{CoxVer2011} as a starting point because this is already in the quasi-Banach space setting.
\end{remark}

\subsection{\texorpdfstring{$L^p$}{Lp}-stochastic integration}\label{ssec:stoch_int_processes}

With the Burkholder-Davis-Gundy type inequality from Proposition~\ref{prop:BDG_frsp} at hand, we can extend the stochastic integral to the closure in  $L^p(\Omega; \gamma(L^2(0,T;H),E))$ of the space of operators that represent elementary adapted processes.
Before we do so, we want to describe this completion.
%Throughout this section we assume the convention that
%if $\Phi$ is an adapted elementary process, then $R_{\Phi}$ denotes the element of $L^{\infty}(\Omega; \gamma(L^2(0,T;H),E))$
%defined by~\eqref{eq:frsp_gradon}.
Recall that if $F$ is an $r$-Banach space,
then $L^0(\Omega;F)$ denotes the space of $F$-valued random variables.
When endowed with the topology of convergence in probability, $L^0(\Omega;F)$ becomes a complete metric space under the metric
$
  d(f,g) = \E ( \nnrm{ f - g}{F}^r \minsym 1 ).
$
In order to state and prove Lemma~\ref{lemma:approx_frsp} without having to
deal with the case $p=0$ separately, in this section we adopt the convention
\begin{equation}\label{eq:L0_convention}
\nnrm{ f }{L^0(\Omega;F)} := \E\geklam{ \nnrm{ f }{F}^r \minsym 1 }.
\end{equation}

\begin{definition}\label{def:adapted_process}
Let $0\leq p\leq \infty$ and let $R \in L^p(\Omega; \gamma(L^2(0,T;H),E))$. We say that $R$
is \emph{$\F$-adapted}, and write $R \in L^p_{\F}(\Omega; \gamma(L^2(0,T;H),E))$, if for all
$t\in [0,T]$ and all $f\in L^2(0,T;H)$ satisfying $\supp f \subseteq [0,t]$ it holds that
$Rf\in L^p(\Omega;E)$ is (strongly) $\cF_t$-measurable.
\end{definition}

\begin{lemma}\label{lemma:approx_frsp}
% The following are equivalent:
% \begin{enumerate}[label=(\alph*)]
%  \item  $R\in L^0_{\F}(\Omega; \gamma(L^2(0,T;H),E))$;
%  \item there exists a sequence of adapted elementary processes $(\Phi_n)_{n\in \N}$
%  such that $\lim_{n\rightarrow \infty} \| R - R_{\Phi_n} \|_{\gamma(L^2(0,T;H),E)} = 0$ $\P$-a.s.
% \end{enumerate}
Let $0\leq p <\infty$  and let $R\in L^p(\Omega; \gamma(L^2(0,T;H),E))$.
% and let $R\in L^p(\Omega; \gamma(L^2(0,T;H),E))$
Then the following are equivalent:
\begin{enumerate}[leftmargin=*,align=right,label=\textup{(\roman*)}]
  \item\label{item:adapted} $R\in L^p_{\F}(\Omega; \gamma(L^2(0,T;H),E))$.
  \item\label{item:approx_frsf} There exists a sequence of adapted elementary processes $(\Phi_n)_{n\in \N}$
  such that
  $\lim_{n\rightarrow \infty} \| R - R_{\Phi_n} \|_{L^p(\Omega;\gamma(L^2(0,T;H),E))} =0$.
 \end{enumerate}
\end{lemma}

\begin{proof}
% We only prove equivalence of \ref{item:adapted} and \ref{item:approx_frsf} for the
% case $p>0$; the proof for $p=0$ is analogous. % SONJA: case p=0 is also proven, if we
% assume the notation $\| . \|_{L^0(\Omega;F)}$ introduced above.

`\ref{item:approx_frsf}$\Rightarrow$\ref{item:adapted}':
% Note that if $R$ is an adapted elementary process, then $R\in L^p_{\F}(\Omega; \gamma(L^2(0,T;H),E))$.
The desired implication follows immediately by observing that for all $t\in [0,T]$ and $f\in L^2(0,T;H)$
satisfying $\supp f\subseteq [0,t]$ it holds that
$R_{\Phi_n} f$ is an $E$-valued, $\cF_{t}$-measurable simple function and
$$\lim_{n\rightarrow \infty} \| R f - R_{\Phi_n} f  \|_{L^p(\Omega;E)} = 0 .$$  \par

`\ref{item:adapted}$\Rightarrow$\ref{item:approx_frsf}':
By rescaling we may assume $T=1$.
Let $R\in L^p_{\F}(\Omega;\gamma(L^2(0,1;H),E))$ and $\eps>0$ be given.\par
For $\eta \in [0,1]$
let $S_{\eta}\in \calL(L^2(0,1;H))$ be the left shift, i.e., for $f\in L^2(0,1;H)$ we define
\begin{equation*}
S_{\eta}f(t) =  \begin{cases}  f(t+\eta),& t\in [0,1-\eta]; \\ 0,& t \in (1-\eta,1]. \end{cases}
\end{equation*}
% Note that for all $f\in L^2(0,1;H)$ it holds that $ S_{\eta}f  \rightarrow f $  in $L^2(0,1;H)$ as $\eta \downarrow 0$.
% It follows from Theorem~\ref{thm:ideal}
% that for all
% $R\in L^p_{\F}(\Omega; \gamma(L^2(0,1;H),E))$ one has that
% $R S_{\eta} \rightarrow R$
% $\P$-a.s.\ in $\gamma(L^2(0,1;H),E))$ as $\eta \downarrow 0$.
Note that $ S_{\eta}^*f  \rightarrow f $  in $L^2(0,1;H)$ as $\eta \downarrow 0$
for all $f\in L^2(0,1;H)$,
and that $\sup_{\eta\in [0,1]}\| S_{\eta} \|_{\calL(L^2(0,1;H))} =1$.
If $p=0$ then Theorem~\ref{thm:ga-conv} implies
$
\lim_{\eta\downarrow 0} \| R - R S_{\eta} \|_{L^p(\Omega;\gamma(L^2(0,1;H),E))} = 0.
$
The same conclusion holds for $p\in (0,\infty)$ by invoking Theorem~\ref{thm:ga-conv},
Theorem~\ref{thm:ideal}, and the dominated convergence theorem.
Hence we can pick $K \in \N$ such that
\begin{equation}\label{eq:approx_frsf_h1}
 \left\| R - R S_{2^{-K}} \right\|_{L^p(\Omega; \gamma(L^2(0,1;H),E))} < \eps.
\end{equation}

Now let $(h_k)_{k\in \N}$ be an orthonormal basis for $H$, and for all $k\in \N$ let $Q_k\in \calL(H)$ be
the orthogonal projection on $\operatorname{span}(\{ h_1,\ldots,h_{k} \})$.
Let $(f_k)_{k\in \N}$ be the Haar basis for $L^2(0,1)$, and for $k\in \N$
let $P_k \in \calL(L^2(0,1))$
be the orthogonal projection on $\operatorname{span}(\{ f_1,\ldots, f_{k}\})$.
Note that
\begin{equation}\label{eq:Haar}
P_{2^k} f
=
\sum_{j=1}^{2^k} 2^{k} \langle \one_{((j-1)2^{- k}, j2^{- k}]}, f\rangle_{L^2(0,1)} \one_{((j-1)2^{- k}, j2^{- k}]},\quad \text{for all } f\in L^2(0,1).
\end{equation}
% It follows that $\{ 2^{n/2}\one_{[(j-1)2^{-n}, j2^{-n} )} \colon j \in \{1,\ldots,2^n\}$
% is an orthonormal basis for $\operatorname{span}(\{f_1,\ldots,f_{2^{n}}\})$, for all $n\in \N_0$.
% it holds that
% \begin{equation}\label{eq:Haar}
% \operatorname{span}(\{f_1,\ldots,f_{2^{n}}\})= \operatorname{span}(\{ 2^{n/2}\one_{[(j-1)2^{-n}, j2^{-n} )} \colon j \in \{1,\ldots,2^n\}).
% \end{equation}

%
By Example~\ref{ex:g-approx} we have
$\lim_{L\rightarrow \infty}
  \|
    R S_{2^{-K}}(P_{L} \otimes Q_{L})
    -
    R S_{2^{-K}}
\|_{\gamma(L^2(0,1;H),E)} = 0$ $\P$-a.s.
It follows from the above that for $p=0$
there exists an $L\geq K$
%$L\in \{K, K+1, \ldots\}$
such that
\begin{equation}\label{eq:approx_frsf_h2}
\left\|
  R S_{2^{-K}} - R S_{2^{-K}} (P_{2^{L}} \otimes Q_{2^{L}})
\right\|_{L^p(\Omega;\gamma(L^2(0,1;H),E))}
< \eps.
\end{equation}
The same conclusion is obtained for  $p\in (0,\infty)$ by invoking Theorem~\ref{thm:ideal} and the dominated convergence theorem.
Now define $R_{K,L}= R S_{2^{-K}} (P_{2^{L}} \otimes Q_{2^{L}})$
and observe that $R_{K,L}$ is represented by $\Phi_{K,L}\colon [0,1]\times \Omega \rightarrow \cL(H,E)$
in the sense of Definition~\ref{eq:frsp_gradon}, where $\Phi_{K,L}$
is given by
\begin{align*}
\Phi_{K,L}
& =
\sum_{j=1}^{2^{L}} \sum_{\ell=1}^{2^{L}}
  \one_{((j-1)2^{-L},j  2^{-L}]} \otimes h_{\ell} \otimes X_{j,\ell}^{(K,L)},
% \end{align*}
\\ \text{with} \quad
% \begin{align*}
X_{j,\ell}^{(K,L)}
& =
2^{L} R S_{2^{-K}} ( \one_{((j-1)2^{-L},j  2^{-L}]} \otimes h_{\ell} )
\\ &
=
2^L R ( \one_{((j-1)2^{-L}-2^{-K},j  2^{-L}-2^{-K}]\cap[0,1]} \otimes h_{\ell})
.
\end{align*}
As
$R\in L^p_{\F}(\Omega;\gamma(L^2(0,1;H),E))$
and $L\geq K$ we have $ X_{j,\ell}^{(K,L)} \in L^p(\Omega,\cF_{(j-1)2^{-L}};E)$.
% Note that $R S_{2^{-K}} (P_{2^{L}} \otimes Q_{2^{L}}) R_{K,L}$
% in the sense of~\eqref{eq:frsp_gradon}, see also~\eqref{eq:Haar}.
% We write
% $R_{K,L}:=R S_{2^{-K}} (P_{2^{L}} \otimes Q_{2^{L}})$.

As a final step, note that if $Y_{j,\ell} \in L^p(\Omega,\cF_{(j-1)2^{-L}},\P;E)$,
$j,\ell \in \{1,\ldots,2^L\}$, and if $R_{K,L,Y}\in L^p(\Omega;\gamma(L^2(0,1;H);E))$
is represented by
\begin{equation}
\Phi_{K,L,Y}
=
\sum_{k=1}^{2^{L}}
  \sum_{\ell=1}^{2^{L}}
    \one_{((j-1)2^{-L},j  2^{-L}]}
    \otimes h_{\ell}
    \otimes Y_{j,\ell},
\end{equation}
then by Estimate~\eqref{eq:tradinONS} in Proposition~\ref{prop:f-r} we obtain
that
\begin{equation*}
\| R_{K,L} - R_{\Phi_{K,L,Y}} \|_{L^p(\Omega;\gamma(L^2(0,1;H);E))}
\eqsim_{p,r}
\Bigg\|
  \sum_{j=1}^{2^L}\sum_{\ell=1}^{2^L}
    2^{-L/2}\gamma_{j,\ell} (X_{j,\ell} - Y_{j,\ell})
\Bigg\|_{L^p(\Omega;E)},
\end{equation*}
where $(\gamma_{j,\ell})_{j,\ell\in \{1,\ldots,2^L\}}$ is a Gaussian sequence on $(\Omega,\cF,\P)$
independent of $\cF_{\infty}:=\sigma(\bigcup_{t\in [0,\infty)} \cF_t)$.
It follows that we can pick $\cF_{(j-1)2^{-L}}$-measurable simple random variables $Y_{j,\ell}$,
$k,\ell\in \{1,\ldots,2^L\}$, such that
\begin{equation}\label{eq:approx_frsf_h3}
\|  R_{K,L} - R_{K,L,Y} \|_{L^p(\Omega;\gamma(L^2(0,1;H);E)) }<\eps.
\end{equation}
Note that $\Phi_{K,L,Y}$ is an adapted elementary process in the sense of Definition~\ref{def:frsp}.
Estimates~\eqref{eq:approx_frsf_h1}, \eqref{eq:approx_frsf_h2}, and~\eqref{eq:approx_frsf_h3}
yield $\| R - R_{K,L,Y}S \|_{L^0(\Omega;\gamma(L^2(0,1;H),E))} < 3 \eps$,
and $\| R - R_{K,L,Y}S \|_{L^p(\Omega;\gamma(L^2(0,1;H),E))} < 3^{\frac{1}{r\minsym p}} \eps$
for all $p\in (0,\infty)$.
\end{proof}

Now all ingredients are available to extend the stochastic integral to integrands in $L^p_{\F}(\Omega;\gamma(L^2(0,T;H),E))$ for $0<p<\infty$. The following theorem is an immediate consequence of Proposition~\ref{prop:BDG_frsp} and Lemma~\ref{lemma:approx_frsp}.

\begin{theorem}\label{thm:stochInt_BDG}
Let $0<p<\infty$, let $E$ satisfy the decoupling property,
and let $R\in L^p_{\F}(\Omega;\gamma(L^2(0,T;H),E))$. Then
there exists a unique continuous $E$-valued $\F$-adapted stochastic process
$(\Psi(t))_{t\in [0,T]} $ such that for every sequence of adapted elementary processes $(\Phi_n)_{n\in \N}$ satisfying
\begin{equation*}
% $
\lim_{n\rightarrow \infty} \E \left\|  R_{\Phi_n} - R \right\|_{\gamma(L^2(0,T;H),E)}^p =0
% $
\end{equation*}
it holds that
\begin{equation*}
\lim_{n\rightarrow \infty} \left\| \Phi_n\cdot W_{H} - \Psi\right\|_{L^p(\Omega;\cont([0,T];E))} = 0.
\end{equation*}
We use the notation $R\cdot W_{H} = (R\cdot W_{H}(t))_{t\in [0,T]}$ for $\Psi$ and call this process the \emph{stochastic integral of $R$ with
respect to $W_{H}$}.  It holds that
\begin{equation}\label{eq:BDG}
 \E \left\| R\cdot W_H \right\|_{\cont([0,T];E)}^p
 \leq C_{E,p}^p % K_{p,2}^p
 \E \left\| R \right\|_{\gamma(L^2(0,T;H),E)}^p,
\end{equation}
where $C_{E,p}$ is as in Proposition~\ref{prop:DPgauss}.

\end{theorem}

\subsection{Localization}\label{ssec:loc}

We now turn to the localized case. To this end, we need to introduce `stopped'
elements of $L^0_{\F}(\Omega; \gamma(L^2(0,T;H),E))$. Hence we define,
for $t\in [0,T]$, the operator $P_t \in \calL(L^2(0,T;H))$ by $P_t(f) = f\one_{[0,t]}$, $f\in L^2(0,T;H)$.

\begin{lemma}\label{lemma:stopped_int}
Let $0\leq p<\infty$, let $E$ satisfy the decoupling property,
let $\tau \colon \Omega \rightarrow [0,T]$ be an $\F$-stopping time,
and let $R \in L^{p}_{\F}(\Omega, \gamma(L^2(0,T;H),E))$. Then
$R P_{\tau} \in L^{p}_{\F}(\Omega, \gamma(L^2(0,T;H),E))$. Moreover, if $p>0$
then
\begin{equation}\label{eq:stopped_int}
 R\cdot W_H (\tau) = R P_{\tau} \cdot W_H (T)\quad\text{in }L^p(\Omega;E).
\end{equation}
\end{lemma}

\begin{proof}
If $R$ is represented by an adapted elementary process,
then it follows trivially from the fact that $\tau$ is a stopping time, Theorem~\ref{thm:ideal},
and the fact that $\sup_{t\in [0,T]} \| P_t \|_{\calL(L^2(0,T;H))} =1$ that
$R P_{\tau} \in L^{p}_{\F}(\Omega, \gamma(L^2(0,T;H),E))$. Moreover, identity~\eqref{eq:stopped_int}
follows from `classical' \Ito\ integration theory.
%\footnote{Sonja: does this require a reference? If so, maybe Karatzas \& Shreve?}

The general case is obtained by considering an approximating sequence of adapted elementary processes $(\Phi_n)_{n\in \N}$,
see Lemma~\ref{lemma:approx_frsp}, and
noting that if
$$
  \lim_{n\rightarrow \infty}
    \| R - R_{\Phi_n} \|_{L^p(\Omega;\gamma(L^2(0,T;H),E))}
  = 0,
%   \quad \P\text{-a.s.\ / in }L^p,
$$
then by Theorem ~\ref{thm:ideal}
also
$$
\lim_{n\rightarrow \infty}
  \| R P_{\tau} - R_{\Phi_n} P_{\tau} \|_{L^p(\Omega;\gamma(L^2(0,T;H),E))}
= 0.
% \quad \P\text{-a.s.\ / in }L^p.
$$
Finally, if $p>0$, then by Theorem~\ref{thm:stochInt_BDG} it holds that
$$
  \lim_{n\rightarrow \infty}
    \E \| R \cdot W_H - R_{\Phi_n} \cdot W_H \|_{\cont([0,T];E)}^p
  =
  \lim_{n\rightarrow \infty}
    \E \| R P_{\tau} \cdot W_H - R_{\Phi_n} P_{\tau} \cdot W_H \|_{\cont([0,T];E)}^p
  =0.
$$
Recalling that Identity~\eqref{eq:stopped_int} holds for $R_{\Phi_n}$, we conclude that it also holds for $R$.
\end{proof}

%
% Define, for all $M\in \N$ and $R\in L^0(\Omega,\gamma(L^2(0,T;H),E))$, the
% mapping $\tau_{M,R} \colon \Omega \rightarrow [0,T]$ by
% \begin{equation}
%  \tau_{M,R}
%  =
%  \inf\left(
%   \left\{
%     t \in [0,T] \colon
%     \| R P_t \|_{\gamma(L^2(0,T;H),E)} \leq M
%   \right\}
%   \cup
%   \{ T \}
%  \right)
% \end{equation}
% Observe that because $R$ is adapted we have, for all $t\in [0,T], M\in \N$, that
% \begin{equation}
%  \{ \tau_{M,R} > t \}
%  =
%  \left\{
%   \| R P_t \|_{\gamma(L^2(0,T;H),E)} \leq M
%  \right\}
%  =
%  \left\{
%   \| R|_{L^2(0,t;H)} \|_{\gamma(L^2(0,t;H),E)} \geq M
%  \right\} \in \cF_t,
% \end{equation}
% i.e, $\tau_M$ is a stopping time.
% It follows from Lemma~\ref{lemma:stopped_int} $R P_{\tau_M} \in L^{\infty}_{\F}(\Omega, \gamma(L^2(0,t;H),E))$.

% Note that $\lim_{t\rightarrow T} \| P_t f - f\|_{L^2(0,T;H)} = 0$ for all $f\in L^2(0,T;H)$.
% It follows from Theorem~\ref{thm:ga-conv} that $\lim_{t\rightarrow T} \| R P_t - R \|_{\gamma(L^2(0,T;H),E)} = 0$
% $\P$-a.s., for all $R\in L^0(\Omega,\gamma(L^2(0,T;H),E))$.

The following lemma is the key ingredient that enables us to define a localized version of the stochastic integral. For $r=1$ it can be found in~\cite[Lemma 5.4]{NeeVerWei2007}.

\begin{lemma}\label{lemma:conv_probability}
Let $0<p<\infty$, let $E$ satisfy the decoupling property,
and let $R\in L^p_{\F}(\Omega; \gamma(L^2(0,T;H),E))$. Then for all $\eps,\delta \in (0,\infty)$
it holds that
\begin{equation*}
 \P\grklam{ \| R \cdot W_H \|_{\cont([0,T];E)} > \eps }
 \leq
 \frac{ C_{E,p}^p \eps^p}{\delta^p}
+ \P\grklam{ \| R \|_{L^p(\Omega; \gamma(L^2(0,T;H);E))} > \delta},
\end{equation*}
where $C_{E,p}$ is the constant in Theorem~\ref{thm:stochInt_BDG}.
\end{lemma}

\begin{proof}
Define the stochastic process $\psi \colon [0,T]\times \Omega \rightarrow \R$
by $\psi(t) = \| R P_t \|_{\gamma(L^2(0,T;H),E)}$.
Note that $\lim_{t\rightarrow T} \| P_t f - f\|_{L^2(0,T;H)} = 0$ for all $f\in L^2(0,T;H)$,
whence it follows from Theorem~\ref{thm:ga-conv} that $\psi$ is a continuous $\F$-adapted process. Moreover,
by the definition of the $\gamma$-radonifying norm (see Definition \ref{def:g-rad}),
it holds that
$ \psi $ is increasing and $\psi(T) = \| R \|_{\gamma(L^2(0,T;H),E)}$.
Finally, it follows from Theorem~\ref{thm:stochInt_BDG} and Lemma~\ref{lemma:stopped_int}
that for every stopping time $\tau\colon \Omega \rightarrow [0,T]$ it holds that
\begin{equation}
 \E \| R \cdot W (\tau) \|_{\cont([0,T];E)}^p \leq C_{E,p}
 \psi(\tau).
\end{equation}
The result now follows from an argument based on Chebyshev's inequality, see~\cite[Lemma~5.1]{NeeVerWei2007}
\end{proof}

\begin{theorem}\label{thm:stochInt_local}
Let $E$ satisfy the decoupling property. Then, for every 
$R \in L^0_{\F}(\Omega;\gamma(L^2(0,T;H),E))$, 
there exists a unique continuous $E$-valued $\F$-adapted stochastic process
$(\Psi(t))_{t\in [0,T]} $ such that for every sequence of adapted elementary processes $(\Phi_n)_{n\in \N}$ satisfying
\begin{equation*}
\lim_{n\rightarrow \infty} \E \geklam{\left\| R_{\Phi_n} - R  \right\|_{\gamma(L^2(0,T;H),E)} \minsym 1 } =0
\end{equation*}
it holds that
\begin{equation*}
 \lim_{n\rightarrow \infty} \E \geklam{ \left\| \Phi_n \cdot W_H  - \Psi \right\|_{\cont([0,T],E)} \minsym 1 } =0.
\end{equation*}
We use the notation $R\cdot W_{H} = (R\cdot W_{H}(t))_{t\in [0,T]}$ for $\Psi$ and call this process the \emph{stochastic integral of $R$ with
respect to $W_{H}$}.
\end{theorem}

\begin{proof}
It follows from Lemma~\ref{lemma:conv_probability} that the stochastic integral mapping for adapted elementary processes (see Definition~\ref{def:stoch_int_frsp})
$$
\Phi \mapsto \Phi \cdot W_H \in L^0(\Omega;\cont([0,T],E))
$$
extends continuously to $L^0_{\F}(\Omega,\gamma(L^2(0,T;H),E))$ with respect to the metric of convergence in probability---here we also use that the adapted elementary processes are dense in this space by Lemma~\ref{lemma:approx_frsp}.
\end{proof}

As for deterministic integrands, see Corollary~\ref{cor:inf_int_rep}, we have the following series expansion of the stochastic integral.
%For $R\in \gamma(L^2(0,T;H),E)$ and $h\in H$ we write $R(\cdot\otimes h)$ for the $\gamma$-radonifying operator $R(\cdot\otimes h)\colon L^2([0,T])\to E$, $f\mapsto R(f\otimes h)$. Moreover, we write $W_H(\cdot\otimes h)$ for the $L^2([0,T])$-isonormal process $W_H(\cdot\otimes h)\colon L^2([0,T])\to L^2(\Omega)$, $f\mapsto W_H(f\otimes h)$, and $W_Hh$ for the Brownian motion $W_Hh(t):=W_H(\one_{(0,t]}\otimes h)$, $t\in[0,T]$.

\begin{corollary}\label{cor:inf_int_rep2}
Let $E$ satisfy the decoupling inequality, let $0\leq p<\infty$, and let $R\in L^p_{\F}(\Omega;\gamma(L^2(0,T;H),E))$. Then, for every orthonormal basis $(h_n)_{n\in\N}$ in $H$,
\begin{equation*}%\label{eq:inf_int_rep2_full}
R\cdot W_H = \sum_{n\in\N} R(\cdot\otimes h_n)\cdot W_H(\cdot \otimes h_n),
\end{equation*}
where the series converges in $L^p(\Omega;\cont([0,T];E))$. In particular, for every $x^*\in E^*$,
\begin{equation*}%\label{eq:inf_int_rep2}
\begin{aligned}
 \langle R\cdot W_H , x^* \rangle
& =
  \sum_{n\in I}
    \int_{0}^{\cdot}
      \langle R^* x^*(t), h_n \rangle_H
    \,\mathrm{d}W_Hh_n(t),
\end{aligned}
\end{equation*}
where the series converges in $L^p(\Omega;\cont([0,T]))$.
\end{corollary}

\begin{proof}
The proof is entirely analogous to the proof of Corollary~\ref{cor:inf_int_rep},
with the understanding that we apply Theorems~\ref{thm:stochInt_BDG} and~\ref{thm:stochInt_local} instead of Proposition~\ref{prop:contfunc}.
\end{proof}

\subsection{Stochastic integration in \texorpdfstring{$r$}{r}-Banach spaces with separating dual}

If $E^*$ separates points in $E$, then we have the following consequence of
Theorems~\ref{thm:stochInt_BDG} and~\ref{thm:stochInt_local}. See also~\cite[Theorems 3.6 and 5.9]{NeeVerWei2007} for the (elementary) proof.
Note that~\cite[Theorems 3.6 and 5.9]{NeeVerWei2007} also provide the reverse statement.
However, in `Step 2' of the proof of the reverse statement a two-sided decoupling inequality is used,
which is not to be available in the $r$-Banach space setting when $r<1$.

\begin{corollary}\label{cor:waekconv}
Let $E$ satisfy the decoupling property,
let $E^*$ separate the points of $E$ and let $0\leq p<\infty$. Let $\Phi\colon [0,T]\to \cL(H,E)$ be such that
$\Phi^*x^*\in L^0(\Omega;L^2(0,T;H))$ for all $x^*\in E^*$, and $\Phi h$ is strongly measurable and adapted
for all $h\in H$.
If $R\in L^p(\Omega;\gamma(L^2(0,T;H),E))$ is such that
\[\lb R f\otimes h, x^*\rb = \int_{0}^T\lb \Phi(t,\cdot)h f(t), x^*\rb dt,
\quad  \text{for all } f\in L^2(0,T), h\in H, x^*\in E^*,\]
then there exists a sequence of adapted elementary processes $(\Phi_n)_{n\in\N}$
such that
\begin{enumerate}[label=\textup{(\arabic*)}]
 \item for all $h \in H$, $x^*\in E^*$ it holds that
 $\lim_{n\rightarrow \infty} \langle \Phi_n h, x^* \rangle = \langle \Phi h, x^* \rangle$
 in $L^0(\Omega;L^2(0,T))$;
 \item there exists a process $\Psi\in L^p(\Omega; \cont([0,T], E))$
 such that
 $$\lim_{n\rightarrow \infty} \E(\left\| \Phi_n \cdot W_H - \Psi \right\|_{\cont([0,T],E)} \minsym 1 ) = 0$$
 and, if $p>0$,
 $\lim_{n\rightarrow \infty} \E \| \Phi_n \cdot W_H - \Psi \|_{L^p(\Omega; \cont([0,T]; E))} = 0$.
\end{enumerate}
\end{corollary}

\section{The stochastic heat equation\label{sec:SPDE}}

Let $d\in \N$, $p,q \in (0,1)$, and $\sigma \in \R$.
Our aim is to prove H\"older continuity in time and spatial Besov regularity, i.e., with respect to $B_{p,q}^{\sigma}(\R^d)$, for the stochastic heat equation
\begin{equation}\label{eq:SPDE}
\left\{\begin{aligned}
du(t,x) & = [\Delta u(t,x) + f(t,x)]\,\mathrm{d}t + \sum_{n\in \N} g_n(t,x) \,\mathrm{d}W_n(t), \quad x\in \R^d,\, t\in [0,\infty),\\
 u(0,x) & = u_0(x), \ \ x\in \R^d.
\end{aligned}
\right.
\end{equation}
Here $(W_n)_{n=1}^{\infty}$ is a sequence of independent $\F$-standard Brownian motions
on a filtered probability space $(\Omega,\cF,\P,\F=(\cF_t)_{t\in [0,\infty)})$,
$u_0\colon \Omega\to \mathscr{S}'(\R^d)$ is $\cF_0$-measurable, 
$f$ is an $\mathscr{S}'(\R^d)$-valued process such that $\langle f,\varphi \rangle$ is  progressively measurable  for all $\varphi\in \mathscr{S}(\R^d)$, 
and
$(g_n)_{n\in \N}$ is an $\mathscr{S}'(\R^d;\ell^2)$-valued process such that $\langle g_n, \varphi\rangle$ is  progressively measurable   for all $n\in \N$, $\varphi\in \mathscr{S}(\R^d)$.
% Sonja: Petru asked me to clarify what adapted means.
We shall work with the following concept of a solution to~\eqref{eq:SPDE}.

\begin{definition}\label{def:solution}
Assume that for all $\varphi\in \mathscr{S}(\R^d)$ it holds that $s\mapsto \lb f(s),\varphi\rb$ is in $L^1(0,T)$ a.s., and $(\lb g_n(t), \varphi\rb)_{n\in \N}$ is in $L^2(0,T;\ell^2)$ a.s.
An adapted process $U:\Omega\times[0,T]\to \mathscr{S}'(\R^d)$ is called \emph{a solution to \eqref{eq:SPDE}} if
for all $\varphi\in \mathscr{S}(\R^d)$, $t\in [0,T]$ it holds that $\lb U, \Delta \varphi\rb \in L^1(0,T)$ and
\begin{align*}
\lb U(t),\varphi\rb - \lb u_0,\varphi\rb & = \int_0^t \lb U(s),\Delta \varphi\rb \,\mathrm{d}s + \int_0^t \lb f(s), \varphi\rb \,\mathrm{d}s \\ & \qquad + \sum_{n\in \N} \int_0^t \lb g_n(s),\varphi\rb \,\mathrm{d}W_n(s).
\end{align*}
\end{definition}
Uniqueness of a solution to~\eqref{eq:SPDE} follows from the uniqueness of a solution to the heat equation in the space of tempered distributions. Theorem~\ref{thm:stoch_heat_eqn} below states existence of a solution to~\eqref{eq:SPDE} assuming certain (spatial) Besov regularity conditions 
on $f$ and $g$.\par

The main challenge in proving well-posedness of~\eqref{eq:SPDE} is that we are dealing with stochastic processes
taking values in $B_{p,q}^{\sigma}(\R^d)$, where possibly $p$ or $q$
is less than $1$, i.e., one is in the $p\minsym q \minsym 1$-Banach
space setting. More specifically, we need to work in the space $\gamma(L^2(0,T;H), B^{\sigma}_{p,q}(\R^d))$.
To this end we observe that given $G\in B^{\sigma}_{p,q}(\R^d;L^2(0,T;H))$, we can associate an operator 
$R:L^2(0,T;H)\to B^{\sigma}_{p,q}(\R^d)$ by setting
\[R f = \int_0^T \lb G(t), f(t)\rb_H dt\]
By Proposition~\ref{prop:Bpqgamma} this operator $R$ is $\gamma$-radonifying and
$\|R\|_{\gamma(L^2(0,T;H), B^{\sigma}_{p,q}(\R^d))} \eqsim_{p,q} \|G\|_{ B^{\sigma}_{p,q}(\R^d;L^2(0,T;H))}$
(note that this implies in particular that $G$ is uniquely determined by $R$).
The proof of Theorem~\ref{thm:stoch_heat_eqn} relies on this characterisation of $\gamma(L^2(0,T;H), B^{\sigma}_{p,q}(\R^d))$. 
Although it is not used in this section, we wish to stress that because $B^{\sigma}_{p,q}(\R^d)$ has a separating dual, the tools developed in Section~\ref{ssec:stoch_int_sepdual} could be employed to deal with stochastic integrals of deterministic integrands.

The results and techniques of this section can be extended to the case that
$f$ and $g$ depend on $u$ in a suitable way. Another possible extension is to consider
more general differential operators (instead of $\Delta$).
We prefer to keep the presentation as simple as possible whilst still demonstrating
the power of stochastic calculus in $r$-Banach spaces such as Besov spaces.\par

We refer to Appendix~\ref{app:BesovSpaces} for
some basic results
on vector-valued Besov spaces needed below.
In order to phrase our result
we introduce, for $r,\alpha, t \in (0,\infty)$ and $E$ a Banach space,
the weighted $L^r$-norm $\nnrm{\cdot}{L^r_{\alpha}(0,t;E)} \colon L^r(0,t;E)\rightarrow [0,\infty]$
defined by $\| f \|_{L^r_{\alpha}(0,t;E)}^r := \int_{0}^{t} (t-s)^{-\alpha r} \| f(s) \|_E^r \,\mathrm{d}s$.
\label{def:weightedLr}
Observe that for $\alpha\in (0,1/r)$, $L^\infty(0,t;E)\subseteq L^r_{\alpha}(0,t;E)$ with 
\[\|f\|_{L^r_{\alpha}(0,t;E)}\leq C_{\alpha, r} t^{\frac1r-\alpha}\|f\|_{L^\infty(0,t;E)}.\]
% By~Proposition \ref{prop:Bpqgamma} for all $g\in \mathscr{S}'(\R^d,L^2(0,t;\ell^2))$ it holds that
% $\left\| g \right\|_{B_{p,q}^{\sigma}(\R^d,L^2_{\alpha}(0,t;\ell^2))}
% \eqsim_{p}
% \left\| (t-\cdot)^{-\alpha} g \right\|_{\gamma(L^2(0,t;\ell^2),B_{p,q}^{\sigma}(\R^d))} $.

% \subsection{Well-posedness and regularity in Besov spaces for the stochastic heat equation}

Our main existence and uniqueness result for the stochastic heat equation~\eqref{eq:SPDE} in arbitrary Besov spaces reads as follows.

\begin{theorem}\label{thm:stoch_heat_eqn}
Let $\sigma \in \R$, $p,q,r,T\in (0,\infty)$, $\alpha\in [0,1)$, $\beta \in [0,\frac{1}{2})$.
Let $(\Omega,\cF,\P)$ be a probability space, $\F=(\cF_t)_{t\in [0,T]}$ a filtration on $\cF$,
and let $W_{\ell^2}$ be an $\F/L^2(0,T;\ell^2)$-isonormal process. Let $u_0\in L^r(\cF_0;B_{p,q}^{\sigma}(\R^d))$. Let $f\colon [0,T] \times \Omega \rightarrow B_{p,q}^{\sigma-2\alpha}(\R^d)$ be 
 progressively measurable, 
%Borel-measurable
%and $\F$-adapted, 
and assume
\begin{equation}\label{eq:ass_f}
C_f^{\eqref{eq:ass_f}}
:=
\sup_{t\in [0,T]}
  \| f|_{\Omega \times [0,t]} \|_{L^r(\Omega; B_{p,q}^{\sigma-2\alpha}(\R^d;L^1_{\alpha}(0,t)))}
<\infty.
\end{equation}
Let $g \colon [0,T] \times \Omega \rightarrow B_{p,q}^{\sigma-2\beta}(\R^d;\ell^2)$ 
be
 progressively measurable, 
%Borel-measurable
%and $\F$-adapted, 
and assume
\begin{equation}\label{eq:ass_g}
C_g^{\eqref{eq:ass_g}}:=\sup_{t\in [0,T]} \| g|_{\Omega\times [0,t]} \|_{L^r(\Omega;B_{p,q}^{\sigma-2\beta}(\R^d; L^2_{\beta}(0,t;\ell^2)))}<\infty.
\end{equation}
Finally, let $K_t\in \mathscr{S}(\R^d)$, $t\in (0,\infty)$,
be the standard heat kernel, i.e., %$K_0 \equiv 1$ and % for all $t\in (0,\infty)$, $x\in \R^d$ it holds that
\begin{equation}\label{eq:def_heat_kernel}
K_t(x) = \frac{1}{(4\pi t)^{d/2}} \exp\Big(-\frac{|x|^2}{4t}\Big),\quad t \in (0,\infty),\, x \in \R^d.
\end{equation}
For $\varphi \in \mathscr{S}'(\R^d)$ set $K_0 * \varphi := \varphi$.
\par

Then for all $t\in (0,\infty)$ it holds that
$K_{t- \cdot }*g|_{[0,t]} \in L^r_{ \mathbb F}(\Omega;\gamma(L^2(0,t;\ell^2),B_{p,q}^{\sigma}(\R^d)))$.
Consequently,
the stochastic integral $[\left( K_{t-\cdot} * g \right)\one_{[0,t]}] \cdot W_{\ell^2}$
is well-defined for all $t\in (0,T]$ by Example~\ref{ex:decoupling_spaces} and Theorem~\ref{thm:stochInt_BDG}.
Defining, for all $t\in [0,T]$, $\varphi\in \mathscr{S}(\R^d)$
\begin{equation}
\begin{aligned}\label{eq:mild_sol_heateqn}
\lb U(t),\varphi \rb & =   \lb K_t * u_0, \varphi\rb
\\ & \quad
  + \int_{0}^{t} \lb K_{t-s}*f(s),\varphi\rb \,\mathrm{d}s
  +  [\lb K_{t-\cdot} * g, \varphi\rb \one_{[0,t]}] \cdot W_{\ell^2}
  \quad\P\text{-a.s.,}
\end{aligned}
\end{equation}
we have that $U$ is an $\F$-adapted, $B_{p,q}^{\sigma}(\R^d)$-valued stochastic process that is a solution to~\eqref{eq:SPDE} in the sense of Definition~\ref{def:solution}.

Moreover, for all
$\lambda \in (0,\alpha\minsym \beta]$
there exists a constant $C_{d,p,\alpha,\beta,\lambda}^{\eqref{eq:stoch_heat_besov_hoelder}}$
(independent of $u_0$, $f$, and $g$)
such that
\begin{equation}\label{eq:stoch_heat_besov_hoelder}
\begin{aligned}
& \| U \|_{C^{\lambda}([0,T]; L^r(\Omega; B_{p,q}^{\sigma-2\lambda}(\R^d)))}
\leq
C_{d,p,q,\alpha,\lambda}^{\eqref{eq:stoch_heat_besov_hoelder}} e^{6 \pi^2 T}
\left(
 \| u_0 \|_{L^r(\Omega; B_{p,q}^{\sigma}(\R^d))}
 +  C_f^{\eqref{eq:ass_f}} + C_g^{\eqref{eq:ass_g}}
\right),
\end{aligned}
\end{equation}
and if $(\alpha \minsym \beta) - \frac{1}{r} >0$, then
$U \in L^r(\Omega, C^{\lambda}([0,T], B_{p,q}^{\sigma-2(\lambda + \frac{1}{r}+\eps)}(\R^d))$
for all $\lambda,\eps \in (0, \alpha\minsym \beta)$
satisfying $\lambda +\eps \leq  (\alpha \minsym \beta) - \frac{1}{r}$.
\end{theorem}

\begin{proof}%[Proof of Theorem~\ref{thm:stoch_heat_eqn}]
It is an immediate consequence of Corollary~\ref{cor:heat_homogeneous_Besov} (with $s=\lambda =0$) that
\begin{equation}\label{eq:sg_est}
\sup_{t\in [0,T]} \| K_t * u_0 \|_{L^r(\Omega;B_{p,q}^{\sigma}(\R^d))}
\leq \big(1 + C_{d,p,0}^{\eqref{eq:besov_heat2}} e^{6\pi^2 T} \big)
\| u_0 \|_{L^r(\Omega;B_{p,q}^{\sigma}(\R^d))}.
\end{equation}
It also follows from Corollary~\ref{cor:heat_homogeneous_Besov}
that for all $\lambda \in (0,\alpha\minsym \beta]$ and all $0\leq s < t \leq T$ one has
\begin{equation}\label{eq:sg_est_hoelder}
\| (K_t - K_s) * u_0 \|_{L^r(\Omega;B_{p,q}^{\sigma-2\lambda}(\R^d))}
\leq C_{d,p,\lambda}^{\eqref{eq:besov_heat2}} (t-s)^{\lambda} e^{6\pi^2 T}
\| u_0 \|_{L^r(\Omega;B_{p,q}^{\sigma}(\R^d))}.
\end{equation}
Proposition~\ref{prop:Bpqgamma}, estimate~\eqref{eq:heat_stochconv_Besov2}
in Corollary~\ref{cor:heat_stochconv_Besov} with $s=\lambda=0$,
and~\eqref{eq:ass_g}, yield 
\begin{equation*}\label{eq:stoch_conv_gamma_est}
\begin{aligned}
& \sup_{t\in (0,T]} \| K_{t-\cdot} * g|_{[0,t]} \|_{L^r(\Omega;\gamma(L^2(0,t;\ell^2),B_{p,q}^{\sigma}(\R^d))}
\\ & \eqsim_{p,q}
\sup_{t\in (0,T]} \| K_{t-\cdot} * g|_{[0,t]} \|_{L^r(\Omega;B_{p,q}^{\sigma}(\R^d;L^2(0,t;\ell^2)))}
% \\ &
  \leq C_{d,p,\beta,0}^{\eqref{eq:heat_stochconv_Besov2}} C_g^{\eqref{eq:ass_g}}  e^{5\pi^2 T},
\end{aligned}
\end{equation*}
whence by Theorem~\ref{thm:stochInt_BDG} and Example~\ref{ex:decoupling_spaces}
there exists a constant $C_{d,p,q,\beta}^{\eqref{eq:stoch_conv_est}}$ (independent of $g$) such that
\begin{equation}\label{eq:stoch_conv_est}
\sup_{t\in (0,T]}
\left\|
 [K_{t-\cdot}*g\one_{[0,t]}]\cdot W_{\ell^2}
\right\|_{L^r(\Omega;B_{p,q}^{\sigma}(\R^d))}
\leq C_{d,p,q,\beta}^{\eqref{eq:stoch_conv_est}} C^{\eqref{eq:ass_g}}_g  e^{5\pi^2 T}.
\end{equation}
Similarly,
by Theorem~\ref{thm:stochInt_BDG}, Example~\ref{ex:decoupling_spaces}, Proposition~\ref{prop:Bpqgamma},
and estimates~\eqref{eq:heat_stochconv_Besov1} and~\eqref{eq:heat_stochconv_Besov2} in Corollary~\ref{cor:heat_stochconv_Besov}
for all $\lambda\in (0,\beta]$ there exists a constant $C_{d,p,q,\beta,\lambda}^{\eqref{eq:stoch_conv_est}}$
such that
for all $0\leq s < t \leq T$:
\begin{equation}\label{eq:stoch_conv_hoelder_est}
\begin{aligned}
&
  \left\|
    [K_{t-\cdot}*g\one_{[0,t]}]\cdot W_{\ell^2}
    -
    [K_{s-\cdot}*g\one_{[0,s]}]\cdot W_{\ell^2}
  \right\|_{L^r(\Omega;B_{p,q}^{\sigma-2\lambda}(\R^d))}
\\ & \lesssim_{p,q}
  \Big(
    \left\|
    (K_{t-\cdot} - K_{s-\cdot})*g\one_{[0,s]}
    \right\|_{L^r(\Omega;B_{p,q}^{\sigma-2\lambda}(\R^d;L^2(0,s;\ell^2)))}
\\ & \qquad +
    \left\|
    K_{t-\cdot}*g\one_{[s,t]}
    \right\|_{L^r(\Omega;B_{p,q}^{\sigma-2\lambda}(\R^d;L^2(s,t;\ell^2)))}
  \Big)
\\ & \leq
  \left(
    C_{d,p,\alpha,\lambda}^{\eqref{eq:heat_stochconv_Besov1}}
    +C_{d,p,\alpha,\lambda}^{\eqref{eq:heat_stochconv_Besov2}}
  \right)
  C^{\eqref{eq:ass_g}}_g
  (t-s)^{\lambda} e^{6\pi^2 T}.
\end{aligned}
\end{equation}

As for the deterministic integral, we have, by estimate~\eqref{eq:heat_inhomogeneous_Besov2}
in Corollary~\ref{cor:heat_inhomogeneous_Besov} with $\lambda=s=0$, that
\begin{equation}\label{eq:det_conv_est}
\begin{aligned}
\sup_{t\in (0,T]}
  \left\|
    \int_{0}^{t} K_{t-s}* g(s) \,\mathrm{d}s
  \right\|_{L^r(\Omega;B_{p,q}^{\sigma}(\R^d))}
\leq C_{d,p,\alpha,0}^{\eqref{eq:heat_inhomogeneous_Besov2}} C^{\eqref{eq:ass_f}}_f  e^{5\pi^2 T}.
\end{aligned}
\end{equation}
Moreover, by estimates~\eqref{eq:heat_inhomogeneous_Besov1} and~\eqref{eq:heat_inhomogeneous_Besov2}
in Corollary~\ref{cor:heat_inhomogeneous_Besov} it follows that for all
$\lambda \in (0,\alpha]$ and all $0\leq s < t \leq T$
one has
\begin{equation}\label{eq:det_conv_hoelder_est}
\begin{aligned}
&  \left\|
    \int_{0}^{t} K_{t-u}* g(u) \,\mathrm{d}s -
    \int_{0}^{s} K_{t-u}* g(u) \,\mathrm{d}s
  \right\|_{L^r(\Omega;B_{p,q}^{\sigma-2\lambda}(\R^d))}
\\ &
\leq
\left(
  C_{d,p,\alpha,\lambda}^{\eqref{eq:heat_inhomogeneous_Besov1}}
  +
  C_{d,p,\alpha,\lambda}^{\eqref{eq:heat_inhomogeneous_Besov2}}
\right)
C^{\eqref{eq:ass_f}}_f (t-s)^{\lambda} e^{6\pi^2 T}.
\end{aligned}
\end{equation}

Estimates~\eqref{eq:sg_est},~\eqref{eq:stoch_conv_est}, and~\eqref{eq:det_conv_est}
guarantee that $U$ is well-defined as an $\F$-adapted, $B_{p,q}^{\sigma}(\R^d)$
valued process. It follows from estimates~\eqref{eq:sg_est_hoelder},~\eqref{eq:stoch_conv_hoelder_est},
and~\eqref{eq:det_conv_hoelder_est} that Estimate~\eqref{eq:stoch_heat_besov_hoelder} holds for all $\lambda \in (0,\alpha\minsym \beta]$.

It follows from Estimate~\eqref{eq:stoch_heat_besov_hoelder} and the Kolmogorov-Chentsov 
criterion---see e.g.~\cite[Theorem~14.9]{Kal}---that 
there exists a version $\bar{U}\colon [0,T]\times \Omega \rightarrow B_{p,q}^{\sigma}(\R^d)$ such that 
$$\E\| \bar{U} \|_{C^{\lambda}([0,T];B_{p,q}^{\sigma -2(\lambda + \frac{1}{r} + \eps)}(\R^d)))}^r < \infty.$$ Note that here we applied 
the Kolmogorov-Chentsov criterion with the metric
$$d(f,g) = \|f-g\|_{B^{\sigma - 2(\lambda + \frac{1}{r} + \eps)}_{p,q}(\R^d)}^{p\minsym q \minsym 1}.$$ Next, recall that for all $\lambda_1,\lambda_2\in (0,1)$ satisfying $\lambda_2>\lambda_1$ it holds that the 
`small' H\"older space $c^{\lambda_1}([0,T])$ is separable and $C^{\lambda_2}([0,T])\hookrightarrow c^{\lambda_1}([0,T])$. Hence we can conclude that $\bar{U}$ is strongly measurable as a $C^{\lambda}([0,T];B_{p,q}^{\sigma -2(\lambda + \frac{1}{r} + \eps)}(\R^d))$-valued process, and thus the claim $U\in L^p(\Omega;C^{\lambda}([0,T];B_{p,q}^{\sigma -2(\lambda + \frac{1}{r} + \eps)}(\R^d)))$ is verified.

It remains to show that $U$ satisfies Definition \ref{def:solution}. This can be proved by repeating the classical argument for mild and weak solutions (see \cite[Theorem 5.4]{DaPZab1998}). This completes the proof of the theorem.
\end{proof}

\begin{remark}\label{rem:she_solution}
Given that $(K_t)_{t\in (0,\infty)}$ is the heat kernel, one may think of the process
$U$ in Theorem~\ref{thm:stoch_heat_eqn}
as a mild solution
to the stochastic heat equation~\eqref{eq:SPDE}.
\end{remark}

\begin{remark}
In view of Proposition~\ref{prop:Bpqgamma}, the weighted norm considered in~\eqref{eq:ass_g}
may be seen as a weighted $\gamma$-norm as introduced in~\cite{VerZim2008} in the Banach space setting;
it is a useful tool in stochastic evolution equations in Banach spaces (see \cite{NeerVerWeis2008}).
Note that if $\alpha\geq 1$ or $\beta\geq \frac{1}{2}$, then~\eqref{eq:ass_f}, respectively ~\eqref{eq:ass_g},
is only satisfied for $f\equiv 0$, respectively $g\equiv 0$.
\end{remark}

\begin{remark}
% Assume the setting of Theorem~\ref{thm:stoch_heat_eqn}.
If $p=q=r$ in Theorem~\ref{thm:stoch_heat_eqn}, then it is possible to obtain the assertions of this
theorem by considering $\langle U (\varphi_m), e_n \rangle_{\ell^2}$ point-wise
in $\xi\in \R^d$ ($e_n$ being the $n^{\text{th}}$ unit vector in $\ell^2$
and $\varphi_m$ being as in Appendix~\ref{app:notation}).
Otherwise, however, the theory developed in Sections~1--5 seems necessary to obtain this result.
In addition to the theory developed in Sections~1--5,
a key element for the proof concerns a technical result on point-wise boundedness
of convolutions by maximal functions (see Lemma \ref{lem:pointwisetechnical}).
\end{remark}
\appendix

\section{Vector-valued Besov spaces}\label{app:BesovSpaces}

\subsection{Notation}\label{app:notation}

In this article we consider vector-valued Besov spaces
defined by means of a Littlewood-Paley type decomposition.
We recall some basic definitions and results below, and refer
to~\cite{Amann1997, Marumatu1974, Pruss1993, Schmeisser1987} for
more details on vector-valued Besov spaces defined via this approach
(these references `only' deal with the Banach-space parameter range, i.e.,
the spaces $B_{p,q}^{\sigma}(\R^d;E)$
with $p,q \in [1,\infty]$ and $\sigma \in \R$).
We refer to~\cite{Tri1983, Tri1992} for a general treatment of the real-valued
case (including the quasi-Banach space parameter range $p \in (0,1)$
or $q\in (0,1)$).\par

Let $E$ be a $\C$-Banach space. For every $d\in \bN$ let $\mathscr{S}(\R^d;E)$ denote the space of $E$-valued Schwartz functions, we set $\mathscr{S}(\R^d):=\mathscr{S}(\R^d;\C)$.
Let $\mathscr{S}'(\R^d;E)=\calL(\mathscr{S}(\R^d),E)$ denote the space of tempered distributions; we write $\mathscr{S}'(\R^d):= \mathscr{S}'(\R^d;\C)$. 
For $\varphi \in \mathscr{S}(\R^d)$ and $f\in \mathscr{S}'(\R^d;E)$ define $\varphi * f \in \mathscr{S}'(\R^d,E)$ by $(\varphi * f)(g) = f(\varphi(-\cdot ) * g)$,
where $(\varphi(-\cdot) * g)(\xi) = \int_{\R^d} \varphi( y-\xi) g(y)\,\mathrm{d}y$.

In what follows, we use a slightly different convention than in the main text and define the Fourier transform $\cF \colon \mathscr{S}(\R^d) \rightarrow \mathscr{S}(\R^d)$ on the Schwartz space to be given by $\cF(f)(s) = \int_{\R^d} e^{-2\pi i\langle s, x\rangle} f(x) \,\mathrm{d}x$ for $f\in \mathscr{S}(\R^d)$ and $s\in \R^d$ 
(this has the consequence that
$\cF^{-1}(f)(s)  = \int_{\R^d} e^{2\pi i\langle s, x\rangle} f(x) \,\mathrm{d}x$
and $\cF(\nabla f)(s)=2\pi i s\cF(f)(s)$).
We also use $\wh{f}$ to denote $\cF(f)$. For $\lambda\in (0,\infty)$,
$\alpha \in \R$ and $f \in \mathscr{S}(\R^d)$ 
set
\begin{align}
(1-\lambda \Delta)^{\alpha}f := \cF^{-1}((1+(2\pi)^2 \lambda |\cdot|^2)^{\alpha} \wh{f}) \in \mathscr{S}(\R^d);
\end{align}
note that for $\lambda = 1$
it corresponds with the Bessel potential operator (and for $\alpha \in \N$
this corresponds with the classical definition of the Laplacian $\Delta$).

In addition, let $\cF \colon \mathscr{S}'(\R^d;E) \rightarrow \mathscr{S}'(\R^d;E)$
denote the Fourier transform on the space of tempered distributions, i.e., for all $f\in \mathscr{S}'(\R^d;E)$
it holds that $\cF(f)\in \mathscr{S}'(\R^d;E)$ is defined by $\langle \cF(f), g\rangle = \langle f, \cF(g)\rangle$, $g\in \mathscr{S}(\bR^d)$.
Again we may write $\wh{f}$ instead of $\cF(f)$. Moreover, for $\lambda\in (0,\infty)$,
$\alpha \in \R$ and $f\in \mathscr{S}'(\R^d;E)$
we define
$(1-\lambda \Delta)^{\alpha} f \in \mathscr{S}'(\R^d;E)$ by setting
$\langle(1-\lambda \Delta)^{\alpha} f, g\rangle = \langle f , (1-\lambda \Delta)^{\alpha}g\rangle$ for $g\in \mathscr{S}(\R^d)$. \par

Let $\varphi\in \mathscr{S}(\R^d)$ be
such that $0\leq \wh{\varphi}(\xi)\leq 1$ for $\xi\in \R^d$, $\wh{\varphi}(\xi) = 1$ for $|\xi|\leq 1$,
and $\wh{\varphi}(\xi)=0$ for $|\xi|\geq \frac32$. Let $\wh{\varphi}_0 = \wh{\varphi}$,
% $\wh{\varphi}_1(\xi) = \wh{\varphi}(\xi/2) - \wh{\varphi}(\xi)$
and for $k\in \N$ define $\wh{\varphi}_k = \wh{\varphi}(2^{-k}\cdot) - \wh{\varphi}(2^{-k+1}\cdot)$.
% and
% \[
% \wh{\varphi}_k(\xi)
% = \wh{\varphi}_1(2^{-k+1} \xi)
% = \wh{\varphi}(2^{-k}\xi) - \wh{\varphi}(2^{-k+1}\xi),  \qquad \xi\in \R^d,\  k \in \N.
% \]
Clearly, for all $n\in \N$ it holds that
$\sum_{k=0}^n \wh{\varphi}_k = {\wh \varphi}(2^{-n} \cdot)$ and hence $\sum_{k=0}^{\infty} {\wh\varphi}_k = 1$ point-wise, and
one may check that
\begin{align}
\supp \wh{\varphi}_k &\subseteq   \{\xi\in\R^d:2^{k-1}\leq|\xi|\leq 3\cdot 2^{k-1}\}, \quad  k\in \N.
\end{align}
For all $\sigma\in \R$, $p,q\in (0,\infty]$, let the Besov $(p\minsym q \minsym 1)$-norm
$\left\| \cdot\right\|_{B_{p,q}^{\sigma}(\R^d;E)}\colon \mathscr{S}'(\R^d;E)\rightarrow [0,\infty]$
be defined by
\begin{equation}\label{eq:defBesovNorm}
 \|f\|_{B^{\sigma}_{p,q}(\R^d;E)} := \Big\| \big( 2^{k\sigma}\varphi_k* f\big)_{k\in \N} \Big\|_{\ell^q(L^p(\R^d;E))},
 \quad f \in \mathscr{S}'(\R^d;E).
\end{equation}
The vector-valued Besov space $B^{\sigma}_{p,q}(\R^d;E)$ is given by
\begin{equation*}
 B^{\sigma}_{p,q}(\R^d;E) = \left\{ f \in \mathscr{S}'(\R^d;E) \colon \|f\|_{B^{\sigma}_{p,q}(\R^d;E)} < \infty \right\} .
\end{equation*}
One easily verifies that~\eqref{eq:defBesovNorm} indeed defines a $(p\minsym q \minsym 1)$-norm on $B^{\sigma}_{p,q}(\R^d;E)$,
and one may check that a different choice of $\varphi$ leads to an equivalent norm
(see e.g.~\cite[Theorem 2.3.3]{Tri1983} for a precise formulation of this result in the real-valued case---the
proof caries over to the vector valued case).

Note that $B^{\sigma}_{p,q}(\R^d;E)$ is a $(p\minsym q \minsym 1)$-Banach space, and that 
$\mathscr{S}(\R^d;E^*)\subseteq [B^{\sigma}_{p,q}(\R^d;E)]^*$. In particular, $[B^{\sigma}_{p,q}(\R^d;E)]^*$ 
contains a subspace that separates points of $B^{\sigma}_{p,q}(\R^d;E)$. A precise characterisation of $[B^{\sigma}_{p,q}(\R^d)]^*$ can be found in \cite[2.11.2 and 2.11.3]{Tri1983}.

%\footnote{
%Sonja: originally, there was a remark here of the following type:
%``By~\cite[2.11.2 and 2.11.3]{Tri1983} one has, for all $p,q\in (0,\infty)$, $\sigma\in \R$, that
%by setting $p'=\frac{p}{(p-1)\maxsym 0}$ and $q'=\frac{q}{(q-1)\maxsym 0}$ (where $\frac{1}{0} := \infty$)
%it holds that
%% \begin{equation}
%$   B_{p',q'}^{-\sigma + \frac{d(1-p)\maxsym 0}{p}}(\R^d)
%   = [B_{p,q}^{\sigma}(\R^d)]'$,
%% \end{equation}
%in particular, $B_{p,q}^{\sigma}(\R^d)$ has a separating dual.'' (Something was also said about the
%vector-valued case, but you have to be a bit more careful there of course.) Anyway, I was wondering if
%we could omit this, as it seems to me that we do not need it. Or do you want to keep it?
%{\color{red} Mark: I would like to keep this. It is a model example of such a situations. Concerning the vector-valued case: the relevant part for us is that $B_{p',q'}^{-\sigma + \frac{d(1-p)\maxsym 0}{p}}(\R^d;E^*)$ is contained in the dual and this is of course true. The converse is more tricky but probably alright in the usual cases.}
%}. \par

\subsection{\texorpdfstring{$\gamma$}{gamma}-Fubini in Besov spaces}
\begin{proposition}\label{prop:Bpqgamma}
Let $H$ be a separable Hilbert space, $p,q\in (0,\infty)$, and $\sigma \in \R$.
Then for all finite rank operators $G\in \calL(H,B^{\sigma}_{p,q}(\R^d))$ one has
\[\|G\|_{\gamma(H,B^{\sigma}_{p,q}(\R^d))} \eqsim_{p,q} \|G\|_{B^{\sigma}_{p,q}(\R^d;H)}.\]
Moreover, $I \colon B^{\sigma}_{p,q}(\R^d;H) \rightarrow  \gamma(H,B^{\sigma}_{p,q}(\R^d))$;
$I(G)h = \langle G, h \rangle_H $, $G\in B^{\sigma}_{p,q}(\R^d;H)$, $h\in H$, defines an isomorphism.
\end{proposition}

The idea of the proof involves a Fubini argument that is classical for
characterizing $\gamma(H,L^q(S;E))$ with $E$ a Banach space, see e.g.~\cite[p.~5]{NeeVerWei2015},  and has already been used in the  proof of Theorem~\ref{thm:sq_fnc_lattice}---note however that $B^\sigma_{p,q}(\bR^d)$ is not a quasi-Banach function space, so that Proposition~\ref{prop:Bpqgamma} is not a straightforward consequence of Theorem~\ref{thm:sq_fnc_lattice}. 
%\footnote{Sonja:
%Originally,
%the following two sentences where here:
%``A difference is  that Besov spaces are in general only isometric to closed subspace of a Banach function space $L^p$.
%Another difference is that since the function space $L^p$ is so simple one can give a more direct proof of the norm equivalence.''
%I did not find these sentences particularly informative so I took them out. I think perhaps what Mark meant was:
%``A difference is that a Besov space is in general only
%isometric to \emph{closed subspace} of the Banach function space $\ell^q(L^p(\R^d))$.''
%This could be added. In my opinion however the second sentence should be omitted entirely, as we already knew perfectly
%well how to deal with $\ell^q(L^p(\R^d))$, there is no
%difference I can see.}

\begin{proof}
% Let $E = B^{\sigma}_{p,q}(\R^d)$.
Let $G \in \gamma(H,B^{\sigma}_{p,q}(\R^d))$ be a finite rank operator, i.e.,
$G = \sum_{n=1}^N h_n\otimes g_n$, where $N\in \N$, $(h_n)_{n=1}^N$ is an orthonormal system in $H$,
and $g_n\in B^{\sigma}_{p,q}(\R^d)$.
Let $(\gamma_n)_{n=1}^{N}$ be a sequence of i.i.d.\ standard $\R$-valued
Gaussian random variables, and let $(\varphi_k)_{k\in\N}$ be as in Appendix~\ref{app:notation}.
By Proposition \ref{prop:f-r}, the Kahane-Khintchine inequalities (applied thrice)
and Fubini's theorem (applied twice) we have
\begin{align*}
% &
\|G\|_{\gamma(H,B^{\sigma}_{p,q}(\R^d))}
&
\eqsim_{p,q} \bigg\|\sum_{n=1}^{N} \gamma_{n} g_n \bigg\|_{L^2(\Omega;B^{\sigma}_{p,q}(\R^d))}
\eqsim_{q} \bigg\|\sum_{n=1}^{N} \gamma_{n} g_n \bigg\|_{L^q(\Omega;B^{\sigma}_{p,q}(\R^d))}
% \\
% &
\\ & =
  \bigg(
    \sum_{k\in \N}
        2^{k\sigma q}
	\bigg\|
	  \sum_{n=1}^{N} \gamma_{n} (\varphi_k * g_n)
	\bigg\|_{L^q(\Omega;L^p(\R^d))}^q
  \bigg)^{\frac{1}{q}}
\\ & \eqsim_{p}
  \bigg(
    \sum_{k\in \N}
        2^{k\sigma q}
	\bigg\|
	  \sum_{n=1}^{N} \gamma_{n} (\varphi_k * g_n)
	\bigg\|_{L^p(\Omega;L^p(\R^d))}^q
  \bigg)^{\frac{1}{q}}
\\ & \eqsim_p
  \bigg(
    \sum_{k\in \N}
        2^{k\sigma q}
	\bigg\|
	  \bigg(
	    \sum_{n=1}^{N} | \varphi_k * g_n |^2
	  \bigg)^{\frac{1}{2}}
	\bigg\|_{L^p(\R^d)}^q
  \bigg)^{\frac{1}{q}}
\\ & = \|G\|_{B^{\sigma}_{p,q}(\R^d;H)}.
\end{align*}\par
In order to complete the proof of Proposition~\ref{prop:Bpqgamma}
it suffices to prove that the functions represented by finite rank operators are dense in $B^{\sigma}_{p,q}(\R^d;H)$
(they are dense in $\gamma(H,B^{\sigma}_{p,q}(\R^d))$ by definition).\par
Recall that $\mathscr{S}(\R^d;H)$ is dense in $B_{p,q}^{\sigma}(\R^d;H)$ (see e.g.~\cite[Theorem 2.3.3]{Tri1983},
the proof also applies to the vector-valued case).
Let $G\in \mathscr{S}(\R^d;H)$. 
Moreover, let
$(h_n)_{n\in I}$ be an orthonormal basis of $\operatorname{span}(\operatorname{Range}(G))$ for some $I\subseteq\bN$. For $n\in \N$
let $P_n \colon H \rightarrow H$ be the orthogonal projection onto $\operatorname{span}\{h_j\colon j\leq n\}$.
Observe that $P_n G$ represents a $B_{p,q}^{\sigma}(\R^d)$-valued finite rank operator, and
$\lim_{n\rightarrow \infty} \| (\operatorname{Id} - P_n) G \|_{B_{p,q}^{\sigma}(\R^d;H)}=0$
by the Dominated Convergence Theorem.
In conclusion, every $G\in \mathscr{S}(\R^d;H)$ can be approximated in $B_{p,q}^{\sigma}(\R^d;H)$
by a sequence of finite rank operators. This completes the proof of the proposition.
\end{proof}

\section{The Hardy-Littlewood maximal function and a technical lemma\label{sec:AppB}}

Lemma~\ref{lem:pointwisetechnical} below provides a point-wise estimate for the mapping $f \mapsto K_t* f$,
where $f\in \mathscr{S}(\R^d)$ is such that the Fourier transform $\wh{f}$ of $f$ has compact support, and $(K_t)_{t\in (0,\infty)}$
is the standard heat kernel (see~\eqref{eq:def_heat_kernel}). As a corollary, we obtain estimates in the
Besov-quasi-norm for the heat kernel applied to an element of a Besov space,
see Corollaries~\ref{cor:heat_homogeneous_Besov},~\ref{cor:heat_stochconv_Besov}, and~\ref{cor:heat_inhomogeneous_Besov}.
The proof is in spirit a Fourier multiplier argument closely related to~\cite[Section 1.5.2]{Tri1983}.

Recall the notation from Appendix~\ref{app:notation}.
For $f\colon \R^d \rightarrow \R$ measurable let $M(f) \colon \R^d \rightarrow [0,\infty]$ denote the
\emph{Hardy-Littlewood maximal function of $f$}, i.e., 
\begin{equation}\label{eq:maxfunc}
 M(f)(x)
 =
 \sup_{r\in (0,\infty)}
 \left(
  \tfrac{\Gamma(\frac{d}{2}+1)}{(\pi r)^d}
  \int_{|y|\leq r} | f(x+y) | \,\mathrm{d}y
 \right), \quad  x \in \R^d.
\end{equation}
% (Note that $\frac{\Gamma(\frac{d}{2}+1)}{(\pi r)^d}$ is the volume of the a ball with radius $r$ in $\R^d$.)

\medskip

%, i.e.,% for all $t\in (0,\infty)$, $x\in \R^d$ it holds that
% \begin{equation}\label{eq:def_heat_kernel}
% K_t(x) = \tfrac{1}{(4\pi t)^{d/2}} \exp\Big(-\tfrac{|x|^2}{4t}\Big),\quad t \in (0,\infty), x \in \R^d.
% \end{equation}

\begin{lemma}\label{lem:pointwisetechnical}
Let $M$ and $(K_t)_{t\in (0,\infty)}$ be as in~\eqref{eq:maxfunc} and~\eqref{eq:def_heat_kernel}, respectively,
and define
\begin{equation*}
 \cD(\R^d) = \left\{
 f\in \mathscr{S}(\R^d) \colon
  \substack{ \supp \wh{f}  \subseteq\{\xi\in \R^d \colon |\xi|\leq \frac{3}{2}\} \text { or } \\
  \exists n\in \N\colon \supp \wh{f}  \subseteq \{\xi\in \R^d \colon  2^{n-1}\leq |\xi|\leq 3\cdot 2^{n-1}\} }
 \right\}.
\end{equation*}
Then for all $\alpha\in \R$, $\lambda \in [0,1)$, $r\in (0,1]$ there
exist constants $C_{d,r,\alpha}^{\eqref{eq:Gconv}}, C_{d,r,\alpha,\lambda}^{\eqref{eq:Gconv2}} \in (0,\infty)$
such that for all $f\in \cD(\R^d)$, $0< s<t<\infty$, and $\xi \in \R^d$ it holds that
\begin{equation}\label{eq:Gconv}
  |(1- (2\pi)^{-2} \Delta)^{\alpha} ( K_t * f )(\xi)|
  \leq
  C_{d,r,\alpha}^{\eqref{eq:Gconv}} t^{-(\alpha\maxsym 0)} e^{5 \pi^2 t} (M(|f|^r)(\xi))^{\frac1r},
\end{equation}
and
\begin{equation}\label{eq:Gconv2}
\begin{aligned}
&
|(1- (2\pi)^{-2} \Delta)^{\alpha} (K_t * f - K_s *f)(\xi)|
\\ &
\leq C_{d,r,\alpha,\lambda}^{\eqref{eq:Gconv2}} s^{(-\alpha-\lambda)\minsym 0} (t-s)^{\lambda} e^{6\pi^2 t} (M (|f|^r)(\xi))^{\frac1r}.
\end{aligned}
\end{equation}
Moreover, if $\alpha\leq 0$ then~\eqref{eq:Gconv} holds for $t=0$, and
if $\alpha \leq - \lambda$ then~\eqref{eq:Gconv2} holds for $s=0$ (recall
that we set $K_0 * \varphi :=\varphi$ for all $\varphi \in \mathscr{S}'(\R^d)$).
% with
% \begin{equation}\label{eq:Gconv2_constant}
% \begin{aligned}
% &
% C_{d,r,\alpha,\lambda}^{\eqref{eq:Gconv2_constant}}(s,t)
%   = (2\pi)^2
% \\ &
% \times
% \left\{
% \begin{array}{ll}
%  (C_{d,r,\alpha}^{\eqref{eq:Gconv}}+C_{\alpha+1,d,r}^{\eqref{eq:Gconv}}) t^{1-\lambda},
%  &
%  \alpha \leq -1;
%  \\
%  C_{d,r,\alpha}^{\eqref{eq:Gconv}} t^{1-\lambda}
%  +
%  C_{\alpha+1,d,r}^{\eqref{eq:Gconv}}
%  \big| \tfrac{1-\lambda}{\alpha+\lambda} \big|^{1-\lambda} s^{(-\alpha-\lambda) \minsym 0} t^{(-\alpha-\lambda)\maxsym 0},
%  & -1 < \alpha \leq 0;
%  \\
%  C_{d,r,\alpha}^{\eqref{eq:Gconv}}
%  \big| \tfrac{1-\lambda}{\alpha+\lambda} \big|^{1-\lambda} s^{-\alpha-\lambda}
% \\
%  \quad  +
%  C_{\alpha+1,d,r}^{\eqref{eq:Gconv}}
%  \big| \tfrac{1-\lambda}{1-\alpha-\lambda}\big|^{1-\lambda} s^{(1-\alpha-\lambda)\minsym 0} t^{(1-\alpha-\lambda)\maxsym 0},
%   & \alpha > \lambda.
% \end{array}
% \right.
% \end{aligned}
% \end{equation}
\end{lemma}
\begin{proof}
Fix $\alpha\in \R$ and $r\in (0,1]$. We first prove the statement involving estimate~\eqref{eq:Gconv}.
To this end we introduce $h\colon (0,\infty)\rightarrow \R$; $h(x)= x^{\alpha}e^{- 4\pi^2 x}$,
and $g_t, m_t\colon \R^d \rightarrow \R$, $t\in [0,\infty)$, given by $g_t(\xi) = t(1+|\xi|^2)$ and
$m_t(\xi) = h(g_t(\xi))$.
Using the Leibniz rule
% and the fact that for all $k\in \N$, $t\in [0,1]$
% it holds that
% $$
%   \sup_{x\in [t,1]} x^{\alpha+\beta-k}
%   \leq
%   1 + t^{\alpha+\beta-k}
% $$
one may verify that for all $\beta,\gamma \in \N_0$
there exists a constant $c_{\alpha,\beta,\gamma}^{\eqref{eq:h_bounded_deriv}}$ such that
for all $t\in (0,\infty)$ it holds that
\begin{equation}\label{eq:h_bounded_deriv}
  \sup_{x \in [t,\infty)}  x^{\beta} h^{(\gamma)}(x)
  \leq c_{\alpha,\beta,\gamma}^{\eqref{eq:h_bounded_deriv}} (1+ t^{(\alpha+\beta-\gamma)\minsym 0}).
\end{equation}
% for all $\beta \in \N$ it holds that
% \begin{equation}
% h^{\beta} (x)
% =
% \sum_{j=1}^{\beta\minsym \alpha}
% \binom{\beta}{j} \frac{\alpha!}{(\alpha-j)!} (-4\pi^2)^{\beta-j} x^{\alpha-j} e^{-4\pi^2 x}, \quad x \in (0,\infty).
% \end{equation}

Let $\frac{\partial}{\partial \xi}$ denote the Fr\'echet derivative with respect to $\xi\in \R^d$.
Using that for $\xi, x^{(1)}, x^{(2)} \in \R^d$ we have
$\frac{\partial}{\partial \xi} g_t(\xi)(x^{(1)}) = 2t \langle \xi, x^{(1)} \rangle_{\R^d}$,
$\frac{\partial^2}{\partial \xi^2} g_t(\xi)(x^{(1)},x^{(2)}) = 2t \langle x^{(1)}, x^{(2)} \rangle_{\R^d}$,
and $\frac{\partial^k}{\partial \xi^k}g_t \equiv 0$ for all $k\geq 3$, it follows, by the Fa\`a di Bruno formula,
that for all $\gamma \in \N_0$, $t\in (0,\infty)$, $\xi,x^{(1)},\ldots, x^{(\gamma)}\in \R^d$ one has
\begin{equation}\label{eq:m_deriv}
\begin{aligned}
& \tfrac{\partial^{\gamma}}{\partial \xi^{\gamma}} m_t(\xi)\big(x^{(1)},\ldots,x^{(\gamma)}\big)
\\&   =
  \sum_{k =\lceil \gamma/2 \rceil}^{\gamma}
%    \frac{\gamma!}{ (2k -\gamma )!(\gamma - k)! 2^{\gamma-k} }
   2^{k} t^{k} h^{(k )}\left(g_t(\xi)\right)
   \sum_{\substack{B\subseteq \{1,\ldots,\gamma\};\\ |B| = 2k-\gamma}}
   \prod_{i\in B} \langle \xi , x^{(i)} \rangle_{\R^d}
   \sum_{\substack{(n_j, m_j)_{j=1}^{\gamma-k} \in \\  P_2(\{1,\ldots,n\}\setminus B)}}
   \prod_{j=1}^{\gamma -k} \langle x^{(n_j)},x^{(m_j)} \rangle_{\R^d},
\end{aligned}
\end{equation}
where for $A\subseteq \{1,\ldots,n\}$ a set with an even number of elements,
$P_2(A)$ is the set of all pairings of elements of $A$ (i.e., $P_2(A)$ itself
contains $\frac{|A|!}{2^{|A|/2} ( |A|/2)!}$ elements).
Note that $\operatorname{Range} g_t = [t,\infty)$. It thus
follows from~\eqref{eq:h_bounded_deriv} and~\eqref{eq:m_deriv} that for all $\gamma\in \N_0$
there exists a constant $c_{\alpha,\gamma}^{\eqref{eq:deriv_mt_est1}}$
such that for all $t\in (0,\infty)$, $\xi\in \R^d$ it holds that
\begin{equation}\label{eq:deriv_mt_est1}
\begin{aligned}
& \left\| \tfrac{\partial^{\gamma}}{\partial \xi^{\gamma}} m_t(\xi) \right\|_{L^{(\gamma)}(\R^d,\R)}
 \leq
 \sum_{k =\lceil \gamma/2 \rceil}^{\gamma}
   \tfrac{2^{2k-\gamma} \gamma!}{ (2k -\gamma )!(\gamma - k)! }
   t^{k} h^{(k )}\left(g_t(\xi)\right) |\xi|^{2k-\gamma}
\\& \leq
 \sum_{k =\lceil \gamma/2 \rceil}^{\gamma}
   \tfrac{2^{2k-\gamma}\gamma!}{ (2k -\gamma )!(\gamma - k)! }
   t^{\gamma/2} [g_t(\xi)]^{k-\gamma/2} h^{(k )}\left(g_t(\xi)\right)
 \leq
 c_{\alpha,\gamma}^{\eqref{eq:deriv_mt_est1}}(t^{\gamma/2}+t^{\alpha\minsym 0}).
\end{aligned}
\end{equation}
Moreover, by a similar argument there exists a constant $c_{\alpha,\gamma}^{\eqref{eq:deriv_mt_est2}}$
such that for all $t\in (0,\infty)$, $\xi\in \R^d$ it holds that
\begin{equation}\label{eq:deriv_mt_est2}
\begin{aligned}
 \left\|
  |\xi|^{\gamma}
  \tfrac{\partial^{\gamma}}{\partial \xi^{\gamma}} m_t(\xi)
 \right\|_{L^{(\gamma)}(\R^d,\R)}
&  \leq
  \sum_{k =\lceil \gamma/2 \rceil}^{\gamma}
   \tfrac{2^{2k-\gamma}\gamma!}{ (2k -\gamma )!(\gamma - k)!  }
   (g_t(\xi))^{k} h^{(k)}\left(g_t(\xi)\right)
\\ &  \leq c_{\alpha,\gamma}^{\eqref{eq:deriv_mt_est2}}(1+t^{\alpha \minsym 0}).
\end{aligned}
\end{equation}

Now fix $f\in \cD(\R^d)$. In the case that there exists an $n\in \N$ such that
$\supp \wh{f} \subseteq \{\xi\in \R^d \colon  2^{n-1}\leq |\xi|\leq 3\cdot 2^{n-1}\}$,
choose $\psi\in \mathscr{S}(\R^d)$ such that $\wh{\psi}(\xi) = 1$
for $1\leq |\xi|\leq 3$ and $\wh{\psi}(\xi) = 0$ if $|\xi|>4$ or $|\xi|\leq \frac{1}{2}$.
Otherwise we choose $\psi$ such that $\wh{\psi}(\xi) = 1$ for $|\xi|\leq \frac{3}{2}$
and $\wh{\psi}(\xi) = 0$ if $|\xi|\geq 2$ and set $n=1$ (so we have $n=1$ if
$\supp \wh{f} \subseteq \{\xi \in \R^d \colon  1\leq |\xi|\leq 3\}$, but also if
$\supp \wh{f} \subseteq \{\xi \in \R^d \colon |\xi| \leq \frac{3}{2} \}$).
We have the following straightforward estimate for all $n\in \N$, $\xi \in \operatorname{supp}\wh{\psi}$:
\begin{equation}\label{eq:simple_support_estimate}
 2^{\ell(n-1)} \leq 1 + 2^{\ell n}|\xi|^{\ell}.
\end{equation}

Note that $\wh{K}_t(\xi) = \exp(- 4 \pi^2  t |\xi|^2)$ and
$\wh{f}=\wh{\psi}(2^{-n+1}\cdot)\wh{f}$ (by assumption on $\operatorname{supp}f$),
i.e., for all $t\in (0,\infty)$, $\xi\in \R^d$ one has
\begin{align}\label{eq:fourier_aux1}
 e^{- 4 \pi^2 t} t^{\alpha} |(1- (2\pi)^{-2} \Delta)^{\alpha} ( K_t * f )|
 =
 |\cF^{-1} (m_t \wh{f})|
 =
 |\cF^{-1} ( m_t \wh{\psi}(2^{-n+1}\cdot)\wh{f})|.
\end{align}

Also note that for all
$\varphi \in L^1_{\mathrm{loc}}(\R^d)$, $c\in (0,\infty)$, and $\xi \in \R^d$ it holds that
$M(\varphi( c\, \cdot))(c^{-1} \xi) = M(\varphi)(\xi)$.
It follows from this observation and~\cite[Theorem 1.3.1 Equation (2)]{Tri1983}
% with,
% in the notation of that theorem, $\Omega = B_3(0) =\{ \xi \in \R^d \colon | \xi | \leq 3\}$
% and $\varphi = g(3\cdot 2^{n-1}\cdot)$,
that there exists a
constant $c_{d,r}^{\eqref{eq:fourier_aux2}}$ such that for all $g \in \mathscr{S}(\R^d)$ and $R\in (0,\infty)$
satisfying $\operatorname{supp} \wh{g} \subseteq B_{4R}(0)$ it holds that
\begin{equation}\label{eq:fourier_aux2}
 \forall \xi,y \in \R^d \colon\quad
 \frac{g(\xi - 2^{-R}y)}{ 1 + |y|^{\frac{d}{r}}}
 \leq c_{d,r}^{\eqref{eq:fourier_aux2}}
 \left[ M(|g|^r)(\xi) \right]^{\frac{1}{r}}.
\end{equation}

We thus have, using that for all $\varphi \in \mathscr{S}(\R^d)$, $c\in (0,\infty)$,
and $\xi \in \R^d$ it holds that $\cF^{-1}( \varphi)(\xi) = c^d \cF^{-1}( \varphi(c\,\cdot))(c \xi)$,
and then applying~\eqref{eq:fourier_aux2}, that for all $t\in (0,\infty)$, $\xi\in \R^d$ one has
\begin{equation}\label{eq:fourier_aux2b}
\begin{aligned}
 & | \cF^{-1}(m_t \wh{\psi}(2^{-n+1}\cdot) \wh{f} )(\xi)|
 =
 \left| \left[ \cF^{-1}\left(m_t \wh{\psi}(2^{-n+1}\cdot) \right)  * f \right] (\xi) \right|
 \\
 & =
 \int_{\R^d}
   \cF^{-1}\left(m_t \wh{\psi}(2^{-n+1}\cdot) \right)(y)
   f(\xi -y)
 \,\mathrm{d}y
\\
 & =
 \int_{\R^d}
   \cF^{-1}\left(m_t(2^{n-1}\cdot) \wh{\psi} \right) (u)
   f(\xi - 2^{-n+1}u)
 \,\mathrm{d}u
\\
 & \leq
 c_{d,r}^{\eqref{eq:fourier_aux2}}
 \left[ M(|f|^r)(\xi) \right]^{\frac{1}{r}}
 \int_{\R^d}
   (1+|u|^{\frac{d}{r}}) \cF^{-1}\left(m_t(2^{n-1}\cdot) \wh{\psi} \right) (u)
 \,\mathrm{d}u.
\end{aligned}
\end{equation}

Thus it remains to estimate the last integral in the expression above. To do so
let $j\in \N$ be the smallest even integer such that $j>d+\frac{d}{r}$,
then with $c_{d,r}^{\eqref{eq:fourier_aux3}}= \| (1 + | \cdot |^2)^{d/2r - j/2} \|_{L^1(\R^d)} < \infty$
for all $t\in (0,\infty)$ it
holds that
\begin{equation}
\begin{aligned}\label{eq:fourier_aux3}
& \left\|(1+|\cdot|^{d/r})\cF^{-1}(m_t(2^{n-1}\cdot)\wh{\psi} ) \right\|_{L^1(\R^d)}
\\ & \leq
2 \left\| (1+ |\cdot|^2)^{d/2r} \cF^{-1}(m_t(2^{n-1}\cdot)\wh{\psi} ) \right\|_{L^1(\R^d)}
\\ &
\leq
2 \|(1+ |\cdot|^2)^{d/2r -j/2} \|_{L^1(\R^d)}
% \\ & \quad \times
\|\cF^{-1} [(1- (2\pi)^{-2} \Delta)^{j/2} ( m_t(2^{n-1}\cdot)\wh{\psi} )]\|_{L^\infty(\R^d)}
\\ &
\leq
2 c_{d,r}^{\eqref{eq:fourier_aux3}} \|(1- (2\pi)^{-2} \Delta)^{j/2} (m_t(2^{n-1}\cdot)\wh{\psi} )\|_{L^1(\R^d)}
\\ &
\leq
\frac{ 2 c_{d,r}^{\eqref{eq:fourier_aux3}} 4^d \pi^{d/2}}{\Gamma(d/2+1)}
\sup_{\xi\in \R^d} |(1- (2\pi)^{-2} \Delta)^{j/2} (m_t(2^{n-1} \cdot )\wh{\psi})(\xi) |,
\end{aligned}
\end{equation}
where we used that the inverse Fourier transform is contractive from $L^1$ into $L^\infty$ and the fact
that $\supp \wh\psi  = \{ \xi \in \R^d \colon |\xi| \leq 4\}$.

The first inequality below is obtained from the Binomium of Newton, the Leibniz rule and the triangle inequality,
the second inequality follows from~\eqref{eq:simple_support_estimate}.
In the third inequality we use the estimates on the derivative of $m_t$ obtained in~\eqref{eq:deriv_mt_est1} and~\eqref{eq:deriv_mt_est2}.
Altogether we obtain that there exist constant
$c_{d,r}^{\eqref{eq:fourier_aux4}}, c_{d,r,\alpha}^{\eqref{eq:fourier_aux4}}$ depending only on $d$ and $r$, respectively
$d$, $r$, and $\alpha$, such that
for all $t \in (0,\infty)$, $\xi\in \R^d$ it holds that
\begin{equation}
\begin{aligned} \label{eq:fourier_aux4}
% &\sup_{\xi\in \R^d}
& \left|(1-(2\pi)^{-2} \Delta)^{j/2}\left(m_t(2^{n-1} \cdot )\wh{\psi}\right) (\xi)\right|
\\ & \leq
c_{d,r}^{\eqref{eq:fourier_aux4}}
% \sup_{\xi\in \R^d}
\sum_{k=0}^{j/2}
%  \binom{j/2}{k}
 \sum_{\ell=0}^{2k}
  2^{\ell (n-1)}
  \left\|
    \left( \tfrac{\partial^{\ell}}{\partial \xi^{\ell}} m_t\right)
    (2^{n-1} \xi)
  \right\|_{L^{(\ell)}(\R^d,\R)}
  \left\| \left( \tfrac{\partial^{2k-\ell}}{\partial \xi^{2k-\ell}} \wh{\psi} \right)(\xi)
  \right\|_{L^{(2k-\ell)}(\R^d,\R)}
\\ & \leq
c_{d,r}^{\eqref{eq:fourier_aux4}}
% \sup_{\xi\in \R^d}
\sum_{k=0}^{j/2}
%  \binom{j/2}{k}
 \sum_{\ell=0}^{2k}
  (1+ 2^{\ell}| 2^{n-1}\xi |^{\ell})
  \left\|
    \left( \tfrac{\partial^{\ell}}{\partial \xi^{\ell}} m_t\right)(2^{n-1} \xi)
  \right\|_{L^{(\ell)}(\R^d,\R)}\\
&\qquad\qquad\qquad\qquad\qquad \qquad\qquad\qquad\qquad\qquad \times  \left\|
    \left( \tfrac{\partial^{2k-\ell}}{\partial \xi^{2k-\ell}}  \wh{\psi}\right)(\xi)
  \right\|_{L^{(2k-\ell)}(\R^d,\R)}
\\ & \leq
 c_{d,r,\alpha}^{\eqref{eq:fourier_aux4}}(1+t^{\alpha \minsym 0})e^{\pi^2 t}.
\end{aligned}
\end{equation}
Combining~\eqref{eq:fourier_aux1}, \eqref{eq:fourier_aux2b}, \eqref{eq:fourier_aux3}, and \eqref{eq:fourier_aux4}
completes the proof of~\eqref{eq:Gconv}.\par
\smallskip

We now turn to the statement involving estimate~\eqref{eq:Gconv2}.
Fix $f\in \cD(\R^d)$, $\lambda\in [0,1)$,
and $0< s<t<\infty$. Recall that for all $g\in \mathscr{S}(\R^d)$ it holds that
$\frac{\partial}{\partial t} (K_t * g) = \Delta (K_t * g)$. It follows that
\[
  (1-(2\pi)^{-2}\Delta)^{\alpha} (K_t * f - K_s * f)
  =
  \int_s^{t} (1-(2\pi)^{-2} \Delta)^{\alpha} \Delta ( K_\tau * f) \,\mathrm{d} \tau.
\]
Therefore, by~\eqref{eq:Gconv} we can estimate, for all $\xi\in \R^d$,
\begin{align*}
& |(1-(2\pi)^{-2}\Delta)^{\alpha} (K_t * f - K_s * f)(\xi)|
\\ & \leq (2\pi)^2
  \int_s^{t}
    |(1-(2\pi)^{-2}\Delta)^{\alpha+1}(K_\tau * f)(\xi)|
    +
    |(1-(2\pi)^{-2}\Delta)^{\alpha}(K_\tau * f)(\xi)|
    \,\mathrm{d} \tau
\\ & \leq
  (2\pi)^2 e^{5\pi^2 t}
  \int_s^t
    C_{d,r,\alpha+1}^{\eqref{eq:Gconv}}  \tau^{-((\alpha+1)\maxsym 0)}
    +
    C_{d,r,\alpha}^{\eqref{eq:Gconv}} \tau^{-(\alpha\maxsym 0)}
  \,\mathrm{d} \tau
  \ (M (|f|^r)(\xi))^{\frac1r}.
\end{align*}
Estimate~\eqref{eq:Gconv2} %and Equation~\eqref{eq:Gconv2_constant}
follows from the above and from the fact that by H\"older's inequality
one has, for all relevant $\beta \in \R$:
\begin{align*}
\int_s^t \tau^{\beta} \,\mathrm{d}\tau
\leq
(t-s)^{\lambda}
\big| \tfrac{1-\lambda}{1-\lambda+\beta} \big|^{1-\lambda}
s^{(1-\lambda+\beta)\minsym 0} t^{(1-\lambda + \beta)\maxsym 0}
\lesssim_{\lambda,\beta}
(t-s)^{\lambda}
s^{(1-\lambda+\beta)\minsym 0}
e^{\pi^2 t}.
\end{align*}
\par
The statement regarding $t=0$ in~\eqref{eq:Gconv} can be proven
e.g.\ by replacing $m_t$ in the proof above by $m\in \mathscr{S}(\R^d)$,
$m(\xi) = (1 + |\xi|^2)^{\alpha}$. The statement concerning $s=0$ in~\eqref{eq:Gconv2}
can then be proven entirely analogously to the above.
\end{proof}

For convenience we recall the following extension of the
Fefferman-Stein theorem~\cite[Theorem~1]{FeffermanStein:1971};
for a proof see~\cite[Theorem~2.6]{GallaratiEtAl:2016}.

\begin{proposition}\label{prop:fefferman_stein}
Let $p,q\in (0,\infty)$, let $r \in (0, p\minsym q)\cap(0,1] $,
and let $(S,\cA,\mu)$ be a $\sigma$-finite measure space. 
For every measurable $f \colon \R^d\times S \rightarrow \R$ define
$\overline{M}(f) \colon \R^d \times S \rightarrow [0,\infty]$ by
$\overline{M}(f)(x,s) = M(f(\cdot, s))(x)$.
Then there exists a constant $C_{p/r,q/r}^{\eqref{eq:FeffermanStein}}$ such that for all
$f \colon \R^d\times S \rightarrow \R$ measurable
it holds that
\begin{equation}\label{eq:FeffermanStein}
 \left\|
    [\overline{M}( |f|^{r} )]^{\frac{1}{r}}
  \right\|_{L^p(\R^d;L^{q}(S))}
 \leq
 C_{p/r,q/r}^{\eqref{eq:FeffermanStein}} \| f \|_{L^p(\R^d;L^q(S))}.
\end{equation}
\end{proposition}

Combining Lemma~\ref{lem:pointwisetechnical}, Proposition~\ref{prop:fefferman_stein}, and the lifting
property for Besov spaces we obtain the following three corollaries.

\begin{corollary}\label{cor:heat_homogeneous_Besov}
Let $d\in \N$, $p,q\in (0,\infty)$, $\sigma\in \R$, and let $\lambda \in [0,1)$.
Then there exists a constant
%s $C_{d,p}^{\eqref{eq:besov_heat1}}$ and
$C_{d,p,\lambda}^{\eqref{eq:besov_heat2}}$ such that for all
$f \in B_{p,q}^{\sigma}(\R^d)$ and all $0\leq s<t<\infty$
it holds that
% \begin{equation}\label{eq:besov_heat1} % THIS FOLLOWS FROM \eqref{eq:besov_heat2} with $\lambda=s=0$.
%   \| K_t * f \|_{B_{p,q}^{\sigma}(\R^d)}
%   \leq
%   C_{d,p}^{\eqref{eq:besov_heat1}} e^{5 \pi^2 t} \| f \|_{B_{p,q}^{\sigma}(\R^d)};
% \end{equation}
% and
\begin{equation}\label{eq:besov_heat2}
\begin{aligned}
&  \| (K_t - K_s) * f \|_{B_{p,q}^{\sigma-2\lambda}(\R^d)}
% \\ &
 \leq
  C_{d,p,\lambda}^{\eqref{eq:besov_heat2}}
  (t-s)^{\lambda} e^{6\pi^2 t}
  \| f \|_{B_{p,q}^{\sigma}(\R^d)}.
\end{aligned}
\end{equation}
\end{corollary}

\begin{proof}
% We prove only the second statement, the proof of the first is entirely analogous.
%
Let $(\varphi_k)_{k\in\N}$ be as in Section~\ref{app:notation} and let $r \in (0,p)\cap (0, 1]$.
It follows from the lifting property (see~\cite[2.3.8]{Tri1983}),
Lemma~\ref{lem:pointwisetechnical}, and the Fefferman-Stein theorem that
for all $f \in B_{p,q}^{\sigma}(\R^d)$, $0\leq s<t<\infty$ we have
\begin{align*}
&
\| (K_t - K_s)*f \|_{B_{p,q}^{\sigma-2\lambda}(\R^d)}
% \\
% &
\eqsim_{\lambda}
 \| (1-(2\pi)^{-2}\Delta)^{-\lambda} (K_t - K_s)*f \|_{B_{p,q}^{\sigma}(\R^d)}
\\ &=
 \left\| \left(
  2^{\sigma k}
  (1-(2\pi)^{-2}\Delta)^{-\lambda} (K_t - K_s) * f * \varphi_k
 \right)_{k\in \N} \right\|_{\ell^q(L^p(\R^d))}
\\ & \leq
  C_{d,r,-\lambda,\lambda}^{\eqref{eq:Gconv2}}
  (t-s)^{\lambda} e^{6\pi^2 t}
  \left\| \left(
    2^{\sigma k}
    [ M( |f * \varphi_k|^{r} )]^{\frac{1}{r}}
  \right)_{k\in \N} \right\|_{\ell^q(L^p(\R^d))}
\\ & \leq
  C^{\text{F.-S.}}_{p/r}
  C_{d,r,-\lambda,\lambda}^{\eqref{eq:Gconv2}}
  (t-s)^{\lambda} e^{5\pi^2 t}
  \left\| \left(
    2^{\sigma k}
    f * \varphi_k
  \right)_{k\in \N} \right\|_{\ell^q(L^p(\R^d))}
\\ & =
 C^{\text{F.-S.}}_{p/r}
  C_{d,r,-\lambda,\lambda}^{\eqref{eq:Gconv2}}
  (t-s)^{\lambda} e^{5\pi^2 t}
  \| f \|_{B_{p,q}^{\sigma}(\R^d;L^r(S))}.
\end{align*}
Taking the infimum over all $r \in  (0,p) \cap (0, 1]$ on the right-hand side above
above we obtain~\eqref{eq:besov_heat2}.
\end{proof}

\begin{corollary}\label{cor:heat_stochconv_Besov}
Let $d\in \N$, $p,q\in (0,\infty)$, $\sigma\in \R$, $\alpha \in [0,\infty)$ and $\lambda \in [0,\alpha]\cap [0,\frac{1}{2})$,
and recall the definition of $L_{\alpha}^{2}(0,t;\ell^2)$ from page~\pageref{def:weightedLr}.
Then there exists a constant $C_{d,p,\alpha,\lambda}^{\eqref{eq:heat_stochconv_Besov1}}$
such that for all $0 < s  < t < \infty$ and
all $f\colon [0,s] \rightarrow B_{p,q}^{\sigma}(\R^d;\ell^2)$ measurable
it holds that
% \begin{equation}\label{eq:heat_stochconv_Besov0} % THIS FOLLOWS FROM \eqref{eq:heat_stochconv_Besov2}
%  \left\| K_{s-\cdot} * f \right\|_{B_{p,q}^{\sigma+2\alpha}(\R^d;L^2(0,s;\ell^2))}
%  \leq
%  C_{d,p,\alpha}^{\eqref{eq:heat_stochconv_Besov0}}
%  e^{5 \pi^2 s}
%  \left\| f \right\|_{B_{p,q}^{\sigma}(\R^d;L^2_{\alpha}(0,s;\ell^2))};
% \end{equation}
% and
\begin{equation}
\begin{aligned}\label{eq:heat_stochconv_Besov1}
& \left\| ( K_{t-\cdot}  - K_{s-\cdot} ) * f \right\|_{B_{p,q}^{\sigma+2(\alpha-\lambda)}(\R^d;L^2(0,s;\ell^2))}
\\ & \qquad \qquad\qquad\qquad\qquad \leq
 C_{d,p,\alpha,\lambda}^{\eqref{eq:heat_stochconv_Besov1}}
 (t-s)^{\lambda} e^{6 \pi^2 t}
 \left\| f \right\|_{B_{p,q}^{\sigma}(\R^d;L^2_{\alpha}(0,s;\ell^2))}.
\end{aligned}
\end{equation}
Moreover, there exists a constant $C_{d,p,\alpha,\lambda}^{\eqref{eq:heat_stochconv_Besov2}}$
such that for all $0\leq s < t <\infty$ and all $f\colon [s,t] \rightarrow B_{p,q}^{\sigma}(\R^d;\ell^2)$ measurable it holds that
\begin{equation}
\begin{aligned}\label{eq:heat_stochconv_Besov2}
 &\left\| K_{t-\cdot} * f \right\|_{B_{p,q}^{\sigma+2(\alpha-\lambda)}(\R^d;L^2(s,t;\ell^2))}
\\ & \qquad\qquad\qquad\qquad
 \leq
 C_{d,p,\alpha,\lambda}^{\eqref{eq:heat_stochconv_Besov2}}
 (t-s)^{\lambda} e^{5 \pi^2 t}
 \left\| f\one_{[s,t]} \right\|_{B_{p,q}^{\sigma}(\R^d;L^2_{\alpha}(0,t;\ell^2))}.
\end{aligned}
\end{equation}
\end{corollary}

\begin{proof}
Let $(\varphi_k)_{k\in\N}$ be as in Section~\ref{app:notation} and let $r \in (0,p)\cap (0, 1]$.
For brevity we introduce $J:=(1-(2\pi)^{-2}\Delta)$.
By the lifting property  (see~\cite[2.3.8]{Tri1983}, or~\cite[Theorem 6.1]{Amann1997}
for the vector-valued case), Lemma~\ref{lem:pointwisetechnical} and Proposition \ref{prop:fefferman_stein}
(with $q=2$ and $S=[0,T]\times \N$)
it follows for all $0 < s < t < \infty$ and
all measurable $f\colon [0,t] \rightarrow B_{p,q}^{\sigma}(\R^d;\ell^2)$ that
\begin{align*}\allowdisplaybreaks
&
\left\|
  ( K_{t-\cdot}  - K_{s-\cdot}) * f
\right\|_{B_{p,q}^{\sigma+2(\alpha-\lambda)}(\R^d;L^2(0,s;\ell^2))}
\\ & \eqsim_{\alpha-\lambda}
\left\| J^{\alpha-\lambda}
 \left( K_{t-\cdot}  - K_{s-\cdot} \right) * f
\right\|_{B_{p,q}^{\sigma}(\R^d;L^2(0,s;\ell^2))}
\\ & =
\bigg\|
  \left(
  2^{k\sigma}
	\left(
	\int_{0}^{s}
	  \left\|
	    J^{\alpha-\lambda}
% 	  (1-(2\pi)^{-2}\Delta)^{\alpha-\lambda}
	    \left( K_{t-u} - K_{s-u} \right) * f(u) *\varphi_k
	  \right\|_{\ell^2}^2
	\,\mathrm{d}u
	\right)^{\frac{1}{2}}
  \right)_{k\in \N}
\bigg\|_{\ell^q(L^p(\R^d))}
\\ & \leq
  C_{d,r,\alpha-\lambda,\lambda}^{\eqref{eq:Gconv2}}
  (t-s)^{\lambda} e^{6\pi^2 t}
\\ & \qquad \times
  \bigg\|
  \bigg(
  2^{k\sigma}
   \left(
      \int_{0}^{s}
	(s-u)^{-\alpha}
	\left\|
	  \left[ M\left( |f(u) * \varphi_k|^r \right) \right]^{\frac{1}{r}}
	\right\|_{\ell^2}^2
      \,\mathrm{d}u
   \right)^{\frac{1}{2}}
  \bigg)_{k\in \N}
\bigg\|_{\ell^q(L^p(\R^d))}
\\ & \leq
  C_{p/r,2/r}^{\eqref{eq:FeffermanStein}} C_{d,r,\alpha-\lambda,\lambda}^{\eqref{eq:Gconv2}}
  (t-s)^{\lambda} e^{6\pi^2 t}
\\ & \qquad \times
\bigg\|
  \bigg(
  2^{k\sigma}
    \left(
      \int_{0}^{s}
	(s-u)^{-\alpha}
	\left\|
	  (f(u) * \varphi_k)
	\right\|_{\ell^2}^2
      \,\mathrm{d}u
    \right)^{\frac{1}{2}}
   \bigg)_{k\in \N}
\bigg\|_{\ell^q(L^p(\R^d))}
\\ & =
  C_{p/r,2/r}^{\eqref{eq:FeffermanStein}} C_{d,r,\alpha-\lambda,\lambda}^{\eqref{eq:Gconv2}}
  (t-s)^{\lambda} e^{6\pi^2 t}
  \left\| f \right\|_{B_{p,q}^{\sigma}(\R^d;L^2_{\alpha}(0,s;\ell^2))}
  .
\end{align*}
Taking the infimum over all $r\in (0,p)\cap (0,1]$ on the right-hand side above
we obtain~\eqref{eq:heat_stochconv_Besov1}.\par
Similarly, by the lifting property, Lemma~\ref{lem:pointwisetechnical} and
Proposition~\ref{prop:fefferman_stein} 
(with $q=2$ and $S=[0,T]\times \N$)
it follows for all $0\leq s  < t < \infty$ and
all measurable $f\colon [s,t] \rightarrow B_{p,q}^{\sigma}(\R^d;\ell^2)$ that
\begin{align*}
&
\left\|
  K_{t-\cdot} * f
\right\|_{B_{p,q}^{\sigma+2(\alpha-\lambda)}(\R^d;L^2(s,t;\ell^2))}
\\ & \eqsim_{\alpha-\lambda}
\left\| J^{\alpha-\lambda}
 \left( K_{t-\cdot} * f \right)
\right\|_{B_{p,q}^{\sigma}(\R^d;L^2(s,t;\ell^2))}
% \\ & =
% \left\|
%   \left(
%   2^{k\sigma}
% 	\int_{s}^{t}
% 	 \left(
% 	  \left\|
% 	    J^{\alpha-\lambda}
% % 	  (1-(2\pi)^{-2}\Delta)^{\alpha-\lambda}
% 	    \left( K_{t-u} * g(u) * \varphi_k \right)
% 	  \right\|_{\ell^2}^2
% 	\,\mathrm{d}u
% 	\right)^{\frac{1}{2}}
%   \right)_{k\in \N}
% \right\|_{\ell^q(L^p(\R^d))}
\\ & \leq
  C_{d,r,\alpha-\lambda}^{\eqref{eq:Gconv}}
  e^{5\pi^2 t}
\\ & \qquad \times
  \bigg\|
  \bigg(
  2^{k\sigma}
   \left(
      \int_{s}^{t}
	(t-u)^{-\alpha+\lambda}
	\left\|
	  \left[ M\left( |f(u) * \varphi_k|^r \right) \right]^{\frac{1}{r}}
	\right\|_{\ell^2}^2
      \,\mathrm{d}u
   \right)^{\frac{1}{2}}
  \bigg)_{k\in \N}
\bigg\|_{\ell^q(L^p(\R^d))}
\\ & \leq
  C_{p/r,2/r}^{\eqref{eq:FeffermanStein}} C_{d,r,\alpha-\lambda}^{\eqref{eq:Gconv}}
  (t-s)^{\lambda} e^{5\pi^2 t}
\\ & \qquad \times
\bigg\|
  \bigg(
  2^{k\sigma}
    \left(
      \int_{s}^{t}
	(t-u)^{-\alpha}
	\left\|
	  (f(u) * \varphi_k)
	\right\|_{\ell^2}^2
      \,\mathrm{d}u
    \right)^{\frac{1}{2}}
   \bigg)_{k\in \N}
\bigg\|_{\ell^q(L^p(\R^d))}
\\ & =
  C_{p/r,2/r}^{\eqref{eq:FeffermanStein}} C_{d,r,\alpha-\lambda}^{\eqref{eq:Gconv}}
  (t-s)^{\lambda} e^{5\pi^2 t}
  \left\| f \one_{[s,t]} \right\|_{B_{p,q}^{\sigma}(\R^d;L^2_{\alpha}(0,t;\ell^2))}
  .
\end{align*}
Again taking the infimum over all $r\in (0,p)\cap [0,1]$ on the right-hand side
above we arrive at~\eqref{eq:heat_stochconv_Besov2}.
\end{proof}

\begin{corollary}\label{cor:heat_inhomogeneous_Besov}
Let $d\in \N$, $p,q\in (0,\infty)$, $\sigma\in \R$, $\alpha \in [0,\infty)$ and $\lambda \in [0,\alpha]\cap [0,1)$
and recall the definition of $L_{\alpha}^{1}(0,t)$ from page~\pageref{def:weightedLr}.
Then there exists a constant $C_{d,p,\alpha,\lambda}^{\eqref{eq:heat_inhomogeneous_Besov1}}$
such that for all $0 < s  < t < \infty$ and
all $f\colon [0,s] \rightarrow B_{p,q}^{\sigma}(\R^d)$ measurable
it holds that
\begin{equation}\label{eq:heat_inhomogeneous_Besov1}
\begin{aligned}
& \left\|
    \int_{0}^{s} (K_{t- u} - K_{s- u}) * f(u) \, \mathrm{d}u
 \right\|_{B_{p,q}^{\sigma+2(\alpha-\lambda)}(\R^d)}
\\&\qquad\qquad\qquad\qquad \leq
 C_{d,p,\alpha,\lambda}^{\eqref{eq:heat_inhomogeneous_Besov1}}
 (t-s)^{\lambda} e^{6 \pi^2 t}
 \| f \|_{B^{\sigma}_{p,q}(\R^d;L^1_{\alpha}(0,s))},
\end{aligned}
\end{equation}
where $\int_{0}^{s} K_{r-u} * f \,\mathrm{d}u \in \mathscr{S}'(\R^d)$, $r\in \{s,t\}$,
is defined by $\left\langle \int_{0}^{s} K_{r-u} * f \,\mathrm{d}u , \varphi \right \rangle 
= \int_{0}^{s} \langle K_{r-u} * f, \varphi\rangle \,\mathrm{d}u $,
$\varphi \in \mathscr{S}(\R^d)$.\par
Moreover, there exists a constant $C_{d,p,\alpha,\lambda}^{\eqref{eq:heat_inhomogeneous_Besov2}}$
such that for all $0 \leq  s  < t < \infty$ and
all $f\colon [s,t] \rightarrow B_{p,q}^{\sigma}(\R^d)$ measurable
it holds that
\begin{equation}
\begin{aligned}\label{eq:heat_inhomogeneous_Besov2}
& \left\|
    \int_{s}^{t} K_{t - u} * f(u) \, \mathrm{d}u
 \right\|_{B_{p,q}^{\sigma+2(\alpha-\lambda)}(\R^d)}\\
&\qquad\qquad\qquad\qquad \leq
 C_{d,p,\alpha,\lambda}^{\eqref{eq:heat_inhomogeneous_Besov2}}
 (t-s)^{\lambda} e^{5 \pi^2 t}
 \| f\one_{[s,t]} \|_{B^{\sigma}_{p,q}(\R^d;L^1_{\alpha}(0,t))}.
\end{aligned}
\end{equation}
\end{corollary}

\begin{proof}
The proof is entirely analogous to the proof of Corollary~\ref{cor:heat_stochconv_Besov},
hence we only provide the argument for proving~\eqref{eq:heat_inhomogeneous_Besov1}.
Let $(\varphi_k)_{k\in\N}$ be as in Section~\ref{app:notation}, $J:=(1-(2\pi)^{-2}\Delta)$,
and let $r \in (0,p)\cap (0, 1]$.  By the lifting property, Lemma~\ref{lem:pointwisetechnical} and
Proposition~\ref{prop:fefferman_stein}
it follows for all $0 < s  < t < \infty$ and all $f\colon [0,s] \rightarrow B_{p,q}^{\sigma}(\R^d)$
it holds that
\begin{equation}
\begin{aligned}
&
\left\|
  \int_{0}^{s} (K_{t-u} - K_{s-u}) * f(u)
  \, \mathrm{d}u
\right\|_{B_{p,q}^{\sigma+2(\alpha-\lambda)}(\R^d)}
\\ &
  \eqsim_{\alpha-\lambda}
\left\|
  \int_{0}^{s} J^{\alpha-\lambda} (K_{t-u}  - K_{s-u} )* f(u)
  \, \mathrm{d}u
\right\|_{B_{p,q}^{\sigma}(\R^d)}
\\ & \leq C_{d,r,\alpha-\lambda, \lambda}^{\eqref{eq:Gconv2}} (t-s)^{\lambda} e^{6 \pi^2 t}
% \\ & \qquad \times
\bigg\|
\bigg( 2^{\sigma k}
  \int_{0}^{s} (s-u)^{-\alpha}  [M ( | ( f(u) * \varphi_k |^r )]^{\frac{1}{r}}
  \, \mathrm{d}u
\bigg)_{k\in \N}
\bigg\|_{\ell^q(L^p(\R^d))}
\\ & \leq C_{p/r,1/r}^{\eqref{eq:FeffermanStein}} C_{d,r,\alpha-\lambda, \lambda}^{\eqref{eq:Gconv2}}
(t-s)^{\lambda} e^{6 \pi^2 t} \| f \|_{B_{p,q}^{\sigma}(\R^d;L^1_{\alpha}(0,t))}
\end{aligned}
\end{equation}
Estimate~\eqref{eq:heat_inhomogeneous_Besov1} follows by taking the infimum over $r\in (0,p)\cap (0,1]$
on the right-hand side above.
\end{proof}

\bibliographystyle{plain}

%
%\bibliography{literature}

\providecommand{\bysame}{\leavevmode\hbox to3em{\hrulefill}\thinspace}
\providecommand{\MR}{\relax\ifhmode\unskip\space\fi MR }
\providecommand{\MRhref}[2]{%
  \href{http://www.ams.org/mathscinet-getitem?mr=#1}{#2}
}
\providecommand{\href}[2]{#2}
\def\cprime{$'$} \def\cprime{$'$}

%\section*{Comments/TODOs/Questions}
%
%\begin{itemize}
%
%\item spelling check (Sonja did one on Feb 14th 2018 according to American spelling)
%
%\end{itemize}

\end{document}